\documentclass[11pt]{amsart}

\usepackage{amsmath,amssymb,amscd,amsthm,amsfonts,amstext,amsbsy,mathrsfs,hyperref,upgreek,mathtools,stmaryrd,enumitem}

\usepackage{fontenc}
\usepackage{amsmath}
\usepackage{amsfonts}
\usepackage{amssymb}
\usepackage{amsthm}
\usepackage{graphicx}
\usepackage[english]{babel}
\usepackage{array}

\usepackage{ dsfont }

\usepackage{hyperref}

\usepackage{amsmath}
\usepackage{amssymb}

\usepackage{amsthm}

\usepackage{ marvosym }

\usepackage{amsthm}
\usepackage{amsmath}
\usepackage{amssymb}
\usepackage{amsfonts}
\usepackage{amsxtra}
\usepackage{mathrsfs}
\usepackage{color}
\usepackage{graphicx}
\usepackage[all]{xy}

\usepackage[utf8x]{inputenc}

\usepackage{amsmath, amssymb, amsthm}
\usepackage{stmaryrd}
\usepackage{graphicx}
\usepackage{cancel}
\usepackage{txfonts}
\usepackage{mathrsfs}
\usepackage{textcomp}
\usepackage{relsize}
\usepackage{cite}

\makeatletter
\def\imod#1{\allowbreak\mkern10mu({\operator@font mod}\,\,#1)}
\makeatother

\newtheorem{definition2}{Definition}[section]
\newtheorem{observation}[definition2]{Observation}

\theoremstyle{definition}
\newtheorem{defi}[definition2]{Definition}
\newtheorem{ques}[definition2]{Question}
\newtheorem{rema}[definition2]{Remark}

\theoremstyle{plain}
\newtheorem{theo}[definition2]{Theorem}
\newtheorem{prop}[definition2]{Proposition}
\newtheorem{lemm}[definition2]{Lemma}
\newtheorem{coro}[definition2]{Corollary}

\newtheorem{maintheo}[definition2]{Main Theorem}

\newtheorem{exam}[definition2]{Example}

\newtheorem{clai}[definition2]{Claim}
\newtheorem*{clai*}{Claim}

\DeclareMathOperator{\forces}{\Vdash}

\DeclareMathOperator{\PP}{\mathbb P}

\newcommand{\dom}{\mathrm{dom}}

\newcommand{\ideal}{\mathcal}

\newcommand{\F}{\ideal{F}}

\newcommand{\restr}{\upharpoonright}

\newcommand{\zero}{\mathbf 0}

\newcommand{\cf}{\mathrm{cf}}

\newcommand{\with}{\mid}

\DeclareMathSymbol{\shortminus}{\mathbin}{AMSa}{"39}

\newcommand{\om}{\omega}

\newcommand{\powerset}{\mathcal{P}}

\newcommand{\subs}{\subseteq}

\newcommand{\name}[1]{\dot{#1}}

\newcommand{\incomp}{\perp}
\newcommand{\comp}{\not\perp}

\newcommand{\QQ}{\mathbb{Q}}
\newcommand{\RR}{\mathbb{R}}

\newcommand{\CC}{\mathbb{C}}

\newcommand{\h}{\mathfrak{h}}
\newcommand{\tfrak}{\mathfrak{t}}

\newcommand{\bfrak}{\mathfrak{b}}

\newcommand{\afrak}{\mathfrak{a}}

\newcommand{\cc}{\mathfrak{c}}

\newcommand{\fin}{\text{fin}}

\newcommand{\concat}{^{\smallfrown}}

\newcommand{\fresh}[1]{\textsf{FRESH}}

\newcommand{\com}[1]{\textsf{COM}}

\newcommand{\weakest}{\mathds{1}}

\newcommand{\rang}{\mathrm{rng}}
\newcommand{\ord}{\mathrm{Ord}}

\newcommand{\FKWW}{\mathbb{Q}}

\newcommand{\init}{\lhd}
\newcommand{\initeq}{\unlhd}

\newcommand{\Resch}{{\textsc{\textsf{red}}}}
\newcommand{\OldResch}[1]{{p}^{\textsc{\textsf{red}}}}

\newcommand{\Osvaldo}[1]{({#1}^{<\omega})^+}
\newcommand{\B}{\mathcal{B}}
\newcommand{\Mathias}[1]{\mathbb{M}(#1)}

\newcommand{\doppelrestr}{\upharpoonright\!\upharpoonright}

\newcommand{\compl}{\lessdot}
\newcommand{\mtx}{\mathbb{Q}}

\newcommand{\tomspec}{spec(\tfrak)}

\newcommand{\homspec}{spec(\h)}

\newcommand{\Pom}{\powerset(\omega)/\fin}

\newcommand{\aomspec}{spec(\afrak)}

\newcommand{\coom}[1]{\textsf{COM}^{cof}}
\newcommand{\coombase}[1]{\textsf{COM}^{base+cof}}
\newcommand{\combase}[1]{\textsf{COM}^{base}}

\newcommand{\precc}{\mu}

\newcommand{\newalpha}{i}
\newcommand{\newbeta}{j}

\newcommand{\lala}{\lambda^{<\lambda}}
\newcommand{\gaga}{\gamma^{<\gamma}}

\newcommand{\below}[1]{\prescript{<#1}{}}
\newcommand{\belowt}[2]{\prescript{<#1+#2}{}}

\newcommand{\newi}{\nu}
\newcommand{\newj}{{\bar{\nu}}}

\newcommand{\untenbeta}{\beta}

\newcommand{\Fsum}{\mathop{\bigoplus}}
\newcommand{\fsum}{\mathop{\oplus}}

\newcommand{\FBa}{\mathfrak{F}}
\newcommand{\neu}{\textsf{new}}
\newcommand{\alt}{\textsf{old}}

\newcommand{\leftup}{\textrm{left-up}}

\newcommand{\reg}{{Reg}}

\newcommand{\strcom}[1]{||_\lambda^{strong}}

\newcommand{\omikron}{{\alpha}}
\newcommand{\zeromikron}{{0}}

\newcommand{\seqlangle}{\langle}
\newcommand{\seqrangle}{\rangle}

\newcommand{\genlangle}{\boldsymbol{\langle}}
\newcommand{\genrangle}{\boldsymbol{\rangle}}

\begin{document}

\title{On heights of distributivity matrices}

\author{Vera Fischer}
\address{Institute of Mathematics, University of Vienna, Kolingasse 14--16, 1090 Wien, Austria}
\email{vera.fischer@univie.ac.at}

\author{Marlene Koelbing}
\address{Institute of Mathematics, University of Vienna, Kolingasse 14--16, 1090 Wien, Austria}
\email{marlenekoelbing@web.de}

\author{Wolfgang Wohofsky}
\address{Institute of Mathematics, University of Vienna, Kolingasse 14--16, 1090 Wien, Austria}
\email{wolfgang.wohofsky@gmx.at}

\thanks{\emph{Acknowledgments.} The authors would like to thank the Austrian Science Fund (FWF) for the generous support through grants Y1012, I4039 (Fischer, Wohofsky) and P28420 (Koelbing). The second author is also grateful for the support by the  \"OAW Doc fellowship.}

\subjclass[2000]{03E35, 03E17}

\keywords{cardinal characteristics; distributivity matrices; forcing; generalized Baire spaces}

\maketitle

\begin{abstract}
We construct a model in which there exists a distributivity matrix of
regular height~$\lambda$ larger than~$\h$;
both $\lambda = \cc$ and $\lambda < \cc$
are possible. A distributivity matrix is a refining system of mad families without common refinement. Of particular interest in our proof is the preservation of $\B$-Canjarness.
\end{abstract}

\section{Introduction}

The Boolean algebra~$\Pom$ has attracted a lot of attention in the last decades.
One of the characteristics of
a partial order is its distributivity.
The distributivity of $\Pom$ is the well-known cardinal characteristic $\h$ 
(the 
\emph{distributivity number}) 
which has been
defined in~\cite{basematrix}, where the famous base matrix theorem is proved, 
and
 is tightly connected to many other structural properties of $\Pom$ involving towers and mad families.

Recall that $\h$ 
is the smallest number of mad families such that there is no single mad family refining\footnote{For basic definitions, see Section~\ref{sec:preliminaries}.} all of them 
(or, equivalently, the least cardinal on which $\Pom$ adds a new function into the ordinals). 
It is easy to see and well-known that such a system of $\h$ many mad families can be chosen to be refining (i.e., having property~(2) from Definition~\ref{defi:com}).

In this paper, we consider distributivity matrices of arbitrary height:

\begin{defi}\label{defi:com}
We say that
$\mathcal{A}=\{A_\xi \with \xi<\lambda\}$
is a \emph{distributivity matrix of height~$\lambda$} 
if
\begin{enumerate}

\item[(1)] $A_\xi$ is a mad family, for each $\xi < \lambda$,

\item[(2)] $A_\eta$ refines $A_\xi$ whenever $\eta \geq \xi$, and

\item[(3)] there is no common refinement, i.e., there is no 
mad family~$B$ 
which refines every $A_\xi$.

\end{enumerate}
\end{defi}

It is
straightforward to check
that the existence of
distributivity matrices
is only a matter of cofinality:
if $\delta$ is a singular cardinal with
$\cf(\delta) = \lambda$,
then there exists a distributivity matrix of height~$\delta$ 
if and only if
there exists one of height~$\lambda$.
Therefore, we are only interested in distributivity matrices of regular height.

Note that 
$\h$ is the minimal height of a distributivity matrix. 
On the other hand, 
it is easy to check that there can never be 
a distributivity matrix of regular height larger than~$\cc$. 
Distributivity matrices and similar objects
have been studied e.g.\ in \cite{basematrix}, 
 \cite{Dordal_model}, 
 \cite{Dow_tree_pi_bases}, 
 \cite{JossenSpinas}, 
\cite{MR2899694}, 
\cite{basematrix_general}, 
and 
\cite{SpinasWohofsky}.
However, to the best of our knowledge, all considered distributivity matrices are of height $\h$. The natural question arises whether a distributivity matrix is necessarily of height~$\h$. 
The main result of this paper
shows that it is consistent that there exists a distributivity matrix of
regular
height~$\lambda$
larger than~$\h$ 
(so, in particular, the existence of distributivity matrices of two different regular heights is consistent):

\begin{maintheo}\label{maintheo:maintheo_general}
Let $V_0$ be a model of ZFC
which
satisfies
GCH.
In~$V_0$, let
$\om_1 < \lambda \leq \precc$
be cardinals such that $\lambda$ is regular and
$\cf(\precc) > \om$.
Then there is a c.c.c.\ (and hence
cofinality
preserving)
extension~$W$ of $V_0$ 
in which 
there exists a distributivity matrix of height $\lambda$, and $\om_1 = \h = \bfrak <\cc = \precc$.
\end{maintheo}

We construct our model~$W$ as follows.
We start with~$V_0$
and first go to the
Cohen
extension
in which
$\cc = \precc$.
In this model~$V$, we define a forcing iteration
(see Section~\ref{subsec:forcing_definition})
which adds a distributivity matrix of height~$\lambda$. 
Building on ideas from~\cite{Hechler}, we use c.c.c.\ iterands which approximate the distributivity matrix by finite
conditions;
we have to use an iteration, because after a single step of the forcing, new reals are added, which prevents the
generically added
almost disjoint families
from being maximal.
We show that the generic object
is actually a distributivity matrix:
in particular,
the branches are towers
(see Section~\ref{subsection:branches}) and
the levels are mad families
(see Section~\ref{subsection:rows});
for that, 
we 
use 
complete subforcings 
(which are based on 
the notion of eligible set;
see Section~\ref{subsec:Q^c_complete})
to capture
new subsets of~$\om$
(see 
Section~\ref{subsection:auxiliary}).

To show that $\om_1=\h=\bfrak$, 
we show that $\bfrak=\omega_1$, and use the fact that $\h\leq \bfrak$ holds in ZFC.
In fact, we
show that the ground model reals~$\B = \om^\om \cap V_0$ remain unbounded.
For that, we
represent our iteration
as a finer iteration of Mathias forcings with
respect to filters (see Section~\ref{subsec:layering_and_F_beta}).
We use a characterization from~\cite{Osvaldo_B_Canjar} to show that
these filters
are $\B$-Canjar (see Section~\ref{sec:B_Canjar_filters}
and
Section~\ref{subsec:The_filters_are_B_Canjar}), 
i.e., that
the corresponding Mathias forcings
preserve the unboundedness of~$\mathcal{B}$. 
In~\cite{our_TOW_MAD},
the same is done
for
Hechler's original forcings~\cite{Hechler}
to add a tower or to add a mad family.

More precisely, 
we can use a genericity argument
to show that the filters are $\B$-Canjar at the stage where they appear, but we need the $\B$-Canjarness in
later stages
of the iteration.
Since the notion of $\B$-Canjarness of a filter is not absolute (see\footnote{We thank  Osvaldo Guzm\'{a}n~\cite{Osvaldo_personal} for providing an example of non-absoluteness.} Example~\ref{exam:not_absolute}), we have to develop
a method
how to guarantee that
the $\B$-Canjarness of a filter is
not destroyed by Mathias forcings with respect to certain other filters. 
One 
basic ingredient is defining 
a ``sum'' $\F_0 \fsum \F_1$ of two (or finitely many) filters $\F_0$ and $\F_1$ 
for which the following holds true 
(see Lemma~\ref{lemm:sum_preserves_Canjar}): 

\begin{prop}
If $\B \subseteq \omega^\omega$ is unbounded and $\F_0 \fsum \F_1$ is $\B$-Canjar, then
Mathias forcing with respect to~$\F_1$ 
forces ``$\F_0$ is $\B$-Canjar''.
\end{prop}

We conclude the paper with further discussion and some open questions. 
In Section~\ref{sec:Dordal_branches}, 
we consider  
the nature of 
maximal branches through distributivity matrices.   
There are two possibilities for a maximal branch: either it is cofinal 
or not.
Consistently, there are 
distributivity matrices of height~$\h$ without cofinal branches (this was shown in \cite{Dordal_model} and~\cite{Dow_tree_pi_bases}). 
In the model of Main Theorem~\ref{maintheo:maintheo_general}, all maximal branches of the (generic) distributivity matrix of height $\lambda > \h$ 
are cofinal. In the Cohen model, however, there are no distributivity matrices of this type of height larger than $\h$. 
In Section~\ref{sec:h_spectrum}, we 
discuss the notion of a distributivity spectrum.

\section{Preliminaries}\label{sec:preliminaries}

In this section, we 
recall 
some very basic and well-known definitions and facts. 
The reader should feel free to skip this section and only come back if necessary.

Let $[\omega]^\omega$ denote the collection of infinite subsets of~$\om$, 
and let 
$\subseteq^*$
denote 
the pre-order of 
almost-inclusion:
$b \subs^* a$ if $b \setminus a$ is finite.
We write $a =^* b$ if $a \subs^* b$ and $b \subs^* a$.
We say that $a$ and $b$ are \emph{almost disjoint} if $a \cap b$ is finite. 
Moreover,
we say that $A \subs [\om]^\om$ is an
\emph{almost disjoint family} (or \emph{ad family}) if
$a$ and $a'$ are almost disjoint whenever $a, a' \in A$ with $a \neq a'$. 
An almost disjoint family~$A$ is \emph{maximal} (called \emph{mad family})
if
for each
$b \in [\om]^\om$ there exists $a \in A$ such that $|b \cap a| = \aleph_0$ (i.e., if $A$ is a maximal antichain in $([\omega]^\omega, \subseteq^*)$). 
For two almost disjoint families $A$ and $B$, we say that 
$B$ \emph{refines}~$A$ if for each $b \in B$ there exists an~$a \in A$ with $b \subs^* a$. 
Let 
\[
\aomspec := \{ \mu \with \mu
\textrm{ is an infinite cardinal and there is
a
mad family of size } \mu \}
\]
be the \emph{mad spectrum on~$\omega$}, and let
$\afrak := \min(\aomspec)$
be the \emph{almost disjointness number}.
It is well-known and easy to see that there are always mad families of size $\cc$, i.e., $\cc\in \aomspec$. Indeed,
  by identifying $2^{<\omega}$ with $\omega$ and taking the set of branches through the tree $2^{<\omega}$, we get an almost disjoint family of size~$\cc$, which can be extended to a mad family (using the axiom of choice).

Recall from Definition~\ref{defi:com} that 
a distributivity matrix $\{A_\xi \with \xi<\lambda\}$ 
is a 
refining system of mad families without common refinement. 
Such a system 
can be viewed as a tree (which we think of growing downwards): for each $\xi < \lambda$, the elements of the mad family $A_\xi$
form the
level~$\xi$ of the tree, and for $b \in A_\eta$ and $a \in A_\xi$ with $\eta > \xi$, the element $b$ is below the element $a$ in the tree if and only if $b \subs^*a$.
Due to the refining structure of the distributivity matrix, 
each element of $A_\eta$ is below exactly one element of $A_\xi$. Note that this tree is
necessarily splitting\footnotemark{} at some limit levels:
this is because there
always
appear
$\subs^*$-decreasing sequences of limit length
which have no weakest lower bound,
and so no single element below such a sequence can be enough to get maximality of the next level.

\footnotetext{See also the discussion in Section~\ref{subsec:forcing_definition}
about the generic distributivity matrix of
Main Theorem~
\ref{maintheo:maintheo_general},
whose underlying tree is splitting everywhere.
}

We say that~$\seqlangle a_\xi \with \xi < \delta \seqrangle$ is a
\emph{branch through the
distributivity matrix~$\mathcal{A} = \{A_\xi \with \xi<\lambda\}$} if $a_\xi \in A_\xi$ for each $\xi < \delta$, and
$a_\eta \subs^* a_\xi$ for each $\xi \leq \eta < \delta$.
We say that the branch is \emph{maximal} if there is no branch through~$\mathcal{A}$
strictly extending it. 
As a matter of fact, a maximal branch 
through a distributivity matrix 
can be cofinal 
or not; 
for a discussion of different types of distributivity matrices (in particular such without cofinal branches), 
see Section~\ref{sec:Dordal_branches}.

We say that $b \in [\om]^\om$
\emph{intersects} a distributivity matrix~$\mathcal{A} = \{A_\xi \with \xi<\lambda\}$ if for each $\xi<\lambda$ there is an $a \in A_\xi$ with $b \subs^* a$.
Note that Definition~\ref{defi:com}(3) is equivalent
to
\begin{enumerate}

\item[(3')] $\{ b \in [\om]^\om \with b \textrm{ intersects } \mathcal{A} \}$
is not dense in~$([\omega]^\omega, \subseteq^*)$, i.e., 
there is 
an 
$\bar{a} \in [\om]^\om$ such that 
no $b \subs^* \bar{a}$ intersects~$\mathcal{A}$;

\end{enumerate}
in particular, (3') holds if there is \emph{no}~$b$ intersecting~$\mathcal{A}$.
If this is the case, we call~$\mathcal{A}$ \emph{normal}. 
In fact, a distributivity matrix can always be turned into a 
normal 
distributivity matrix of the same height 
(basically by ``restricting'' the matrix to a witness $\bar{a}$ for~(3')).

We say that a distributivity matrix~$\mathcal{A} = \{A_\xi \with \xi<\lambda\}$ is a 
\emph{base matrix} if $\bigcup_{\xi<\lambda} A_\xi$ is 
dense in~$([\omega]^\omega, \subseteq^*)$. 
It is straightforward to check that 
a base matrix is always normal.

For a sequence $\seqlangle a_\xi \with \xi < \delta \seqrangle\subs [\om]^\om$,
we say that $b \in [\om]^\om$ is a \emph{pseudo-intersection} of $\seqlangle a_\xi \with \xi < \delta \seqrangle$ if $b \subs^* a_\xi$ for each $\xi < \delta$.
We say that
$\seqlangle a_\xi \with \xi < \delta \seqrangle$
is a \emph{tower of length~$\delta$}
if
$a_\eta \subs^* a_\xi$ for
any $\eta > \xi$, and
it does not have an infinite pseudo-intersection.
Let
\[
\tomspec := \{ \delta \with \delta \textrm{ is regular and there is a tower of length } \delta \}
\]
be the \emph{tower spectrum},
and let
$\tfrak := \min(\tomspec)$
be the \emph{tower number}.
Note that whenever $\seqlangle a_\xi \with \xi < \delta \seqrangle$ is a tower, then there is a (sub)tower of length~$\cf(\delta)$. On the other hand, each tower of length~$\cf(\delta)$ can be expanded to one of length~$\delta$ (by repeating elements). Therefore
the restriction to regular cardinals in the definition of the tower spectrum makes sense.

For $f,g\in \omega^\omega$, we write $f\leq^*g$ if $f(n)\leq g(n)$ for all but finitely many $n\in \omega$.
We say that $\B\subseteq \omega^\omega$ is an \emph{unbounded family}, if there exists no $g\in \omega^\omega$ with $f\leq^* g$ for all  $f\in \B$.
The \emph{(un)bounding number~$\bfrak$} is the smallest size of an unbounded family in $\omega^\omega$.
The following inequalities between the cardinal characteristics are well-known and not too hard to prove 
(see, e.g., \cite{Blass_cardinal_characteristics} for more details):
\begin{equation}\label{eq:ZFC_inequ}
\om_1 \leq \tfrak \leq \h \leq \bfrak \leq \afrak \leq \cc.
\end{equation}

\section{Forcing a distributivity matrix}\label{section:4}

In this section, we 
start 
with the proof of our main result, i.e., Main Theorem~\ref{maintheo:maintheo_general}: 
we define a forcing 
(see Section~\ref{subsec:forcing_definition}), 
show 
basic properties of the forcing and of the generic object (see Sections~\ref{subsec:ccc_etc} and~\ref{subsec:generic_matrix}), 
and 
prove a crucial lemma about 
complete subforcings 
(see Section~\ref{subsec:Q^c_complete}). 
In Section~\ref{section:om_2_in_COM}, we will finish the proof that the generic object is indeed a distributivity matrix. 
Section~\ref{sec:B_Canjar_filters} and Section~\ref{section:h_b_om_1} are devoted to the remaining part of the proof of Main Theorem~\ref{maintheo:maintheo_general}, i.e., to showing 
that 
$\om_1 = \h = \bfrak$ in the final model.

\subsection{Definition of the forcing iteration}\label{subsec:forcing_definition}

We will define a forcing for adding a distributivity matrix. The definition has been 
motivated by the forcing for adding towers and mad families from Hechler's paper~\cite{Hechler}.
The presentation of our forcing 
will be 
somewhat different from the presentation of Hechler's forcings 
in~\cite{Hechler}. In~\cite{our_TOW_MAD}, we represent
these forcings in a form which is
analogous
to our
definition of $\QQ_\alpha$ 
below.

We proceed as follows.
In $V_0$, let
$\CC_\mu$ be the usual forcing for adding $\mu$ many Cohen reals, and let $V$ be the extension by~$\CC_\mu$.
In~$V$, we perform our main forcing iteration of length~$\lambda$ which is going to add a distributivity matrix of height~$\lambda$. 
The iteration is going to be a finite support iteration whose iterands have the countable chain condition (see Lemma~\ref{lemm:FKW_ccc_Knaster}) and are of size continuum; in particular, the size of the continuum stays the same during the whole iteration, and hence $\cc = \mu$ holds true in the final model (see Lemma~\ref{lemm:size_of_c}).

As
discussed in 
Section~\ref{sec:preliminaries},
a distributivity matrix can be viewed as
a tree,
where each node is equipped with an element of~$[\om]^\om$. 
In fact,
our generic distributivity matrix
$\{A_{\xi+1} \with \xi<\lambda\}$
will be based on the tree~$\lambda^{<\lambda}$:
 each node $\sigma \in \lambda^{<\lambda}$ of successor length
will carry
an infinite set~$a_\sigma \subs \om$ such that for each
$\xi < \lambda$,
\[
A_{\xi+1} = \{ a_\sigma \with \sigma \in \lambda^{\xi+1} \}
\]
is a mad family, and
$a_\sigma \subs^* a_\tau$ if
$\sigma$ extends $\tau$.
In particular, all maximal branches
of our distributivity matrix will be cofinal.

We write $\tau\unlhd \sigma$ if $\tau \subseteq \sigma$ (i.e., if $\sigma$ extends $\tau$);
we write
$\tau \lhd \sigma$ if $\tau \unlhd \sigma$ and $\tau \neq \sigma$. The length of $\sigma$ is denoted by~$|\sigma|$.
We
think of
the tree~$\lambda^{<\lambda}$
as
``growing downwards'',
i.e., we
say that $\sigma$ is below $\tau$ if
$\tau \unlhd \sigma$;
moreover,
we say that
$\sigma\concat \newbeta$ is
to the
left of $\sigma\concat \newalpha$ whenever $\newbeta<\newalpha$.

Note that
our mad families~$A_{\xi+1}$ are indexed by successor ordinals only, for the following reason.
Since there are
$\subs^*$-decreasing sequences of limit length
which
do not have weakest lower bounds, and the
mad family
on the level directly below such a sequence has to be ``maximal below the sequence''
(i.e., each pseudo-intersection of the sequence is compatible with some element of the mad family),
it is necessary
that the underlying tree
``splits''
at such limit levels.
However,
$\lambda^{<\lambda}$ does not split at limit levels,
so it is convenient to
equip only
nodes~$\sigma$ of successor length
with infinite sets~$a_\sigma$, and use
nodes~$\rho$ of limit length to talk about the branch
$\seqlangle a_\sigma \with \sigma \init \rho \seqrangle$.

Before giving the precise definition of our forcing iteration, let us describe the idea informally.
We
start with the tree $\lala$ in~$V$, and generically add a set $a_\sigma \subs\om$ for every $\sigma\in \lala$ (of successor length) in such a way
that
$a_\tau \supseteq^* a_\sigma$ if
$\tau \unlhd \sigma$,
and $a_\sigma\cap a_\tau =^* \emptyset$ if
$|\sigma| = |\tau|$.
This results in a refining system of 
almost disjoint families. 
But these antichains are not maximal, which can be seen as follows.
The forcing adds new reals (any
$a_\tau$
is a new infinite subsets of $\omega$), so there are new branches
through
$\lambda^{<\omega}$. Let $\rho$ be such a new branch of length $\omega$; then $\seqlangle a_{\rho\restr n} \with n<\om\seqrangle$ is a $\subseteq^*$-decreasing sequence of length $\omega$ (in the extension), so (since $\omega<\tfrak$) it has an infinite pseudo-intersection~$b$. It is easy to see that $b$ is incompatible
with
all $a_\sigma$ with
$\sigma\in \lambda^{\omega+1}\cap V$,
 so the antichain
$\{ a_\sigma \with |\sigma|= \om+1\}$
is not maximal.

To solve this problem, we use a finite support iteration
\[
\{ \PP_\alpha, \name{\mtx}_\alpha \with \alpha < \lambda \}
\]
of length~$\lambda$.
At each step $\mtx_\alpha$ of the iteration,
a set $a_\sigma$ is added for every new node~$\sigma$ (of successor length)
of the tree $\lala$.
In the definition below,
we will use $T_\alpha$ to denote these new nodes
(the
nodes $\sigma$ for which no sets $a_\sigma$ have been added yet),
and
we will use  $T'_\alpha$ to denote the old nodes
(the nodes $\sigma$ for which there has already been added a set $a_\sigma$ at an earlier stage $\untenbeta < \alpha$ of the iteration).
The sets~$a_\sigma$ with $\sigma \in T'_\alpha$
will be used in the definition of the iterand~$\mtx_\alpha$.
In the definition of the very first forcing~$\mtx_0$ of the iteration, the set $T_0$ will be the collection of all nodes in $\lala$ 
(of successor length), 
and $T'_0$ will be empty (since no sets $a_\sigma$ have been defined yet).
After $\lambda$ many steps, we are finished, because no new nodes appear at stage~$\lambda$
(see Lemma~\ref{lemm:sigma_appears_somewhere}).

As usual, we 
abuse notation and 
identify $a_\sigma \subs \om$ with its characteristic function in~$2^\om$.

\begin{defi}\label{defi:main_definition}
Let $\alpha < \lambda$, and
assume that $\PP_\alpha$ has already been defined. For every $\untenbeta \leq \alpha$, let $G_\untenbeta$ be generic for~$\PP_\untenbeta$. We work in $V[G_\alpha]$ to define 
our iterand~$\mtx_\alpha$.
First
(letting
$\mathrm{succ}$
denote the sequences of ordinals of
successor length),
let
$$
T'_\alpha := \bigcup_{\untenbeta < \alpha} (\lala \cap \mathrm{succ})^{V[G_{\untenbeta}]},
$$
and let
$$
T_\alpha:= (\lala \cap \mathrm{succ})^{V[G_{\alpha}]} \setminus T'_\alpha.
$$

Note that for each $\sigma \in T'_\alpha$, there exists a minimal $\untenbeta < \alpha$ such that $\sigma \in V[G_\untenbeta]$, and hence, by induction, $a_\sigma$ has been added by $\mtx_\untenbeta$.
For each $\sigma \in T_\alpha$, the set $a_\sigma$ is not defined yet, and will be added by $\mtx_\alpha$
(see below, at the end of the definition).

Now $\QQ_\alpha$ is defined 
as follows:
$p \in \mtx_\alpha$ if
$p$ is a function
with finite domain,
$\dom(p) \subs T_\alpha$,
and for each $\sigma \in \dom(p)$,
we have
$$p(\sigma) = (s^p_\sigma, f^p_\sigma, h^p_\sigma) =
(s_\sigma, f_\sigma, h_\sigma),$$
where\footnote{The paragraph
after the definition
gives
a short
intuitive explanation
of the roles of $s_\sigma$, $f_\sigma$,
and~$h_\sigma$.}
\begin{enumerate}

\item $s_\sigma \in 2^{<\om}$,

 \item for each $\tau \in \dom(p)$ with $\tau \init \sigma$, $|s_\tau| \geq |s_\sigma|$,

 \item $\dom(f_\sigma) \subs (\dom(p)\cup T'_\alpha) \cap \{ \tau \in
T_\alpha \cup T'_\alpha \with \tau \init \sigma \}$, finite,\footnote{Note that in~(6), it automatically follows that $\dom(h_\sigma)$ is finite because $\dom(h_\sigma) \subs \dom(p)$, but not here
because $\dom(f_\sigma) \subs \dom(p)\cup T'_\alpha$.}

 \item $f_\sigma{:}\ \dom(f_\sigma) \rightarrow \om$,

 \item whenever $\tau \in \dom(f_\sigma)\cap
T_\alpha$, and\footnote{By~(2), here it is actually sufficient to require $n \in \dom(s_\sigma)$.} $n \in \dom(s_\tau) \cap \dom(s_\sigma)$ 
with 
$n \geq f_\sigma(\tau)$, we have
 $$s_\tau(n) = 0 \rightarrow s_\sigma(n) = 0,$$
 and whenever $\tau \in \dom(f_\sigma)\cap T'_\alpha$ and 
 $n \in \dom(s_\sigma)$ 
 with 
 $n\geq f_\sigma(\tau)$,  we
 have 
 $$a_\tau(n) = 0 \rightarrow s_\sigma(n) = 0,$$

 \item
 $\dom(h_\sigma) \subs \dom(p) \cap \{ \rho \concat \newbeta \with \newbeta < \newalpha \}$
(where
$\rho \in \lala$ and $\newalpha \in \lambda$ such that $\sigma = \rho \concat \newalpha$),

  \item $h_\sigma{:}\ \dom(h_\sigma) \rightarrow \om$,

 \item whenever $\tau \in \dom(h_\sigma)$, and 
 $n \in \dom(s_\tau) \cap \dom(s_\sigma)$ 
 with 
 $n \geq h_\sigma(\tau)$, we have
  $$s_\tau(n) = 0 \lor s_\sigma(n) = 0.$$

\end{enumerate}

The order
on $\mtx_\alpha$ is defined as follows:
$q \leq p$ (``$q$ is stronger than $p$'') if
\begin{enumerate}
 \item[(i)] $\dom(p) \subs \dom(q)$,

 \item[(ii)] and for each $\sigma \in \dom(p)$, we have
 \begin{enumerate}
 \item $s^p_\sigma \initeq s^q_\sigma$,

 \item $\dom(f^p_\sigma) \subs \dom(f^q_\sigma)$ and
 $f^p_\sigma(\tau) \geq f^q_\sigma(\tau)$
 for each $\tau\in \dom(f^p_\sigma)$,

 \item $\dom(h^p_\sigma) \subs \dom(h^q_\sigma)$ and
 $h^p_\sigma(\tau)\geq h^q_\sigma(\tau)$
 for each $\tau\in \dom(h^p_\sigma)$.

 \end{enumerate}

 \end{enumerate}

Given a generic filter $G$ for~$\QQ_\alpha$, we define, for each $\sigma \in
T_\alpha$,
$$a_\sigma := \bigcup \{ s^p_\sigma \with  p \in G \land \sigma \in \dom(p) \}.$$

This completes the definition of the forcing.
\end{defi}

Let us describe the role of the parts of a condition:
$s_\sigma$ is a finite approximation of the set $a_\sigma$
assigned to $\sigma$,
whereas the functions $f_\sigma$ and $h_\sigma$ are promises for guaranteeing
that the branches through the
 generic matrix are $\subs^*$-decreasing and the levels are almost disjoint families, respectively.
  More precisely, $f_\sigma$
  promises
  that $a_\sigma\setminus f_\sigma(\tau) \subseteq a_\tau$ for
  each $\tau\in \dom(f_\sigma)$ 
   and $h_\sigma$
  promises that
$a_\tau \cap a_\sigma \subseteq h_\sigma(\tau)$
for each $\tau\in \dom(h_\sigma)$ 
(see Lemma~\ref{lemm:forcing_almost_included_disjoint}).

\begin{rema}\label{rema:not_sep}
Note that $\QQ_\alpha$
is not separative. As an example, we can take $p$
and $q$ as follows: $\dom(p)=\dom(q)=\{\sigma,\tau\}$ (where $\sigma$ is to the left of $\tau$ within the same block), $p(\tau)=q(\tau)=(\seqlangle 1\seqrangle, \emptyset, h)$ where $h(\sigma)=0$ and $p(\sigma)=(\seqlangle \seqrangle, \emptyset, \emptyset)$ and $q(\sigma)=(\seqlangle 0 \seqrangle, \emptyset, \emptyset)$. It is easy to see  that $p\nleq q$, but $p\leq^*q$, i.e., any
condition stronger than~$p$
is compatible with~$q$.
Therefore, we later need to provide
certain iteration lemmas for the general case of non-separative forcings (see 
Lemma~\ref{lemm:general_iteration_reduction}). 
\end{rema}

\begin{rema}\label{rema:more_general}
Let us remark that it is possible to derive a bit more from the proof of Main Theorem~\ref{maintheo:maintheo_general} than what is stated in the theorem.
Our 
forcing construction
is based on the tree~$\lala$ (see Definition~\ref{defi:main_definition}) and therefore
results in a specific kind of
distributivity matrix of height~$\lambda$: first, all its maximal  branches are cofinal, and second, the underlying tree
has
$\lambda$-splitting, i.e., each node has exactly $\lambda$ many immediate successors. From the latter property, it immediately
follows that $\lambda \in \aomspec$ (in particular, $\afrak \leq \lambda$).

We can modify
the construction
(by changing the underlying tree) to obtain
different kinds of
distributivity matrices of height~$\lambda$.
In fact, the following generalization of
Main Theorem~\ref{maintheo:maintheo_general} holds true:
if
$\omega_1\leq \lambda\leq \cf(\theta)$ and $\theta\leq \precc$
with $\lambda$ regular and
$\cf(\precc) > \om$,
then (using $\theta^{<\lambda}$ as the underlying tree)
there is an extension such that
$\omega_1=\h=\bfrak$ and $\precc=\cc$, and there exists a distributivity matrix of height~$\lambda$ with $\theta$-splitting
(hence, in particular, $\theta \in \aomspec$).
The reason why we have to require $\lambda\leq \cf(\theta)$ is Lemma~\ref{lemm:finding_gamma_fresh} 
(see Remark~\ref{rema:why_cof_theta_large}). 
It would even be possible to have different splitting at different nodes, provided that all the
splitting
sizes
have cofinality at least~$\lambda$.
This way, we can get more values into $\aomspec$ (similar as in Hechler's paper \cite{Hechler},
where he constructs a model in which all uncountable regular
cardinals up to $\cc$ are in $\aomspec$).

Note that even $\lambda = \om_1$ is possible in our forcing construction. 
It is true that it does
not yield a 
distributivity matrix of regular height larger than~$\h$ (because $\h = \om_1$ holds true in our model), 
but
we
can obtain distributivity matrices of height~$\om_1$ with
additional features (e.g.,
by
choosing 
$\theta = \lambda = \om_1$, resulting in a matrix with $\om_1$-splitting). Observe that
it is always possible to turn a distributivity matrix with $\theta$-splitting into a distributivity matrix with $\cc$-splitting (of the same height), by just taking every $\om$th level (and deleting all other levels).
It is not clear
whether it is possible to do it the other way round, i.e.,
to get a distributivity matrix with
$\theta$-splitting (for $\theta < \cc$) from a
distributivity matrix with $\cc$-splitting (even if $\theta$ happens to be in~$\aomspec$). Therefore, we decided to state and prove Main Theorem~\ref{maintheo:maintheo_general} for
$\theta = \lambda$ (i.e., small splitting),
and not for $\theta = \cc$.
\end{rema}

As a matter of fact,
the Cohen model satisfies $\aomspec=\{\omega_1,\cc\}$ (see, e.g., \cite[Proposition~3.1]{Brendle_mad}).
Thus,
if
$\omega_1 < \theta < \cc$,
there are no mad families of size $\theta$,
and hence no distributivity matrix with $\theta$-splitting in the Cohen model.
If we choose, e.g., $\lambda=\omega_1$, $\theta=\omega_2$ and $\mu=\omega_3$ in the generalization of our main theorem described in the above remark, the generic matrix
cannot exist in the Cohen model
with
$\cc=\omega_3$.
On the other hand,
our forcing construction with
$\theta=\lambda=\omega_1$
and
$\omega_1 < \mu$
actually results in the Cohen model
with $\cc = \mu$:  
this can be seen by representing the iteration as an iteration of  Mathias forcings with respect to filters, as described in Section~\ref{subsec:layering_and_F_beta}; 
since 
$\theta=\lambda=\omega_1$, all the filters are countably generated, therefore the respective Mathias forcings are forcing equivalent to Cohen forcing.
Therefore, 
we can in particular derive the following from
our
proof of Main Theorem~\ref{maintheo:maintheo_general}:

\begin{observation}
Let $\mu>\omega_1$.
Then,
in the Cohen model
with $\cc = \mu$ (i.e., in the extension of a GCH model by $\CC_\mu$),
there exists a distributivity matrix of height~$\omega_1$ which is 
$\omega_1$-splitting 
everywhere.  
\end{observation}

In particular, 
the distributivity matrix 
is $\omega_1$-splitting at limit levels;  
note that being 
$\om_1$-splitting at successors is not so much of interest because it is known that $\om_1 \in \aomspec$, and so such splitting at successors can be accomplished by hand.

\subsection{Countable chain condition and some implications}\label{subsec:ccc_etc}

We are now going to show that our iterands $\QQ_\alpha$ have the c.c.c.; it immediately follows that their finite support iteration~$\PP_\lambda$ has the c.c.c.\ as well, and therefore it does not change cofinalities or cardinalities.

\begin{lemm}\label{lemm:FKW_ccc_Knaster}
$\mtx_\alpha$ 
is precaliber~$\om_1$
(hence in particular c.c.c.) 
for every $\alpha < \lambda$.
\end{lemm}

In fact, $\QQ_\alpha$ is even $\sigma$-centered: in Section~\ref{section:h_b_om_1}, we are going to show that
each $\mtx_\alpha$
can be represented as a finite support iteration
of length strictly less than~$\cc^+$
of Mathias forcings with respect to
certain filters;
since
filtered Mathias forcings are always $\sigma$-centered (see Definition~\ref{defi:Mathias_with_filter} and the subsequent remark), and
$\sigma$-centeredness is preserved under finite support iterations of length strictly less than~$\cc^+$,
it follows that
$\QQ_\alpha$
is
$\sigma$-centered
(see also Corollary~\ref{coro:alles_sigma_centered}).

\begin{proof}[Proof of Lemma~\ref{lemm:FKW_ccc_Knaster}]
Let $\{p_i \with i<\omega_1\}\subseteq \mtx_\alpha$.
We want to show that the set cannot be an antichain.
First note that it is possible to extend\footnote{In 
Lemma~\ref{lemm:full}, 
we will show a stronger fact. 
}
all $s_\sigma^p$ (with $\sigma \in \dom(p)$) of a condition $p \in \QQ_\alpha$ to the same length~$N_p \in \om$, by just adding $0$'s at the end. Therefore we can assume without loss of generality
that there exists $N$ such that 
$|s_\sigma^{p_i}| = N$
for each $i \in \om_1$ and each
$\sigma \in \dom(p_i)$ 
(the reason why we want the $s_\sigma^{p_i}$ to have the same length, is to avoid trouble with Definition~\ref{defi:main_definition}(2)).
Since
$\dom(p_i)\subseteq T_\alpha \subs \lala$ is finite for every $i$, we can apply the $\Delta$-system lemma
to find a subset $X\subseteq \omega_1$ of size $\omega_1$ such that $\{ \dom(p_i)\with i\in X\}$ is a $\Delta$-system
with root~$R \subs T_\alpha$.
Also,
$\dom(f_\sigma^{p_i}) \cap T'_\alpha$
is finite for each $i\in X$ and
each $\sigma\in\dom(p_i)$,
so we can repeatedly apply\footnote{In case $\alpha = 0$, this is not necessary, because $T'_0 = \emptyset$ (in the definition of~$\QQ_0$).} the $\Delta$-system lemma, for each $\sigma\in R$ (hence finitely many times), to find
a subset $Y \subseteq X$ of size $\omega_1$ such that $\{ \dom(f_\sigma^{p_i}) \cap T'_\alpha \with i \in Y \}$ is a $\Delta$-system with root~$A_\sigma$ for each $\sigma \in R$.
Moreover, we can assume without loss of generality that for each  $\sigma \in R$, there are $s_\sigma^*$, $f_\sigma^*$, and $h_\sigma^*$ such that for all $i \in Y$, we have $s_\sigma^{p_i} = s_\sigma^*$, $f_\sigma^{p_i} \restr (R \cup A_\sigma) = f_\sigma^*$, and
$h_\sigma^{p_i} \restr R = h_\sigma^*$.
Now it is straightforward to check that any two
conditions from $\{ p_i \with i \in Y \}$ are compatible;  
in fact, any finitely many of them have a common lower bound. 
\end{proof}

We now show that the size of the continuum in the final model
is as desired; in fact, the following holds:

\begin{lemm}\label{lemm:size_of_c}
Let $\alpha \leq \lambda$. Then, in $V[\PP_{\alpha}]$,
we have
$\cc=\precc$. 
\end{lemm}
\begin{proof}
First note that $V\models \cc=\precc \land \precc^{<\precc} = \precc$, because it is the extension after adding $\precc$ many 
Cohen 
reals over a model which satisfies GCH.
We show simultaneously by induction on $\alpha\leq \lambda$ that 
\begin{enumerate}
\item $|\PP_\alpha|\leq \precc$ and
\item $V[\PP_{\alpha}]\models \cc=\precc\land |T_\alpha| \leq \lambda^{<\lambda}\leq \precc$.
\end{enumerate}

Clearly 
(1) and (2) hold for $\PP_0$ since $\PP_0$ is the trivial forcing.
Now assume that we have shown (1) and (2) for each $\alpha'<\alpha$.

To show~(1), argue as follows.
If $\alpha$ is a limit, then
 $|\PP_\alpha|\leq \precc$, because we use finite support, each $\PP_{\alpha'}\leq \precc$, and $\alpha\leq\precc$.
If $\alpha=\alpha'+1$ is a successor,
$\PP_\alpha=\PP_{\alpha'}*\QQ_{\alpha'}$. By induction, $\PP_{\alpha'}\forces |T_{\alpha'}| \leq \lala\leq\precc$, and so it is easy to check that
$|\QQ_{\alpha'}|\leq
\precc$, hence $|\PP_{\alpha'}*\QQ_{\alpha'}|\leq\precc$.

To show (2), we count nice names.
 For every real $x$ in $V[\PP_\alpha]$, there exists a nice name. Such a nice name
 consists of
 antichains $X_n$ in $\PP_\alpha$ for each entry $x(n)$. By the c.c.c., each $X_n$ is countable, so the number of nice names for reals is $|\PP_\alpha|^{\leq\omega}\leq\precc$, so there are
 only
 $\precc$ many reals in~$V[\PP_\alpha]$. Similarly, a nice name for an element of $\lambda^{<\lambda}$
consists of
 less than $\precc$
 many countable antichains, and since $|\PP_\alpha|^{<\precc}\leq \precc^{<\precc}=\precc$, there are
 at most
 $\precc$ many elements of $\lambda^{<\lambda}$ in $V[\PP_\alpha]$.
\end{proof}

The following lemma guarantees that, by the end of the iteration of length~$\lambda$,
a set~$a_\sigma$ has been added for every $\sigma \in \lala$ 
of successor length 
(so $T_\lambda$ would be empty,
hence $\QQ_\lambda$ would be the trivial forcing -- if we would continue the iteration after~$\lambda$ many stages):

\begin{lemm}\label{lemm:sigma_appears_somewhere}
Every node $\sigma\in \lala$ from the final model $V[\PP_{\lambda}]$
already
appears in some intermediate model $V[\PP_\alpha]$ with $\alpha<\lambda$.
\end{lemm}
\begin{proof}
Let $\dot{\sigma}$ be a nice $\PP_\lambda$-name for $\sigma$; more precisely, $\dot{\sigma}$ has the following form. First, $\dot{\sigma}$ contains an antichain which decides the length of $\dot{\sigma}$. Since $\PP_\lambda$ has the c.c.c., this antichain is countable, so there are only countably many values possible for the length; let $\xi<\lambda$ be larger than all the possible values.
Now, for all $\xi'<\xi$, there is an antichain deciding the entry of $\dot{\sigma}(\xi')$ (if $\xi'$ is less than the length of $\dot{\sigma}$). Again, by c.c.c.\ all these antichains are countable.
So there are $\xi$ many countable antichains which are in $\dot{\sigma}$; the union of these antichains contains less than $\lambda$ many elements. Since we use finite support, there exists an $\alpha<\lambda$ such that $\dot{\sigma}$ is a $\PP_\alpha$-name, hence $\sigma\in V[\PP_\alpha]$.
\end{proof}

\subsection{The generic distributivity matrix}\label{subsec:generic_matrix}

Let $G$ be a generic filter for the iteration~$\PP_\lambda$.
In the final model~$V[G]$,
we
derive our ``intended generic object'' (which is going to be a distributivity matrix of height~$\lambda$) from the generic filter $G$ as follows.
For each $\sigma \in \lala \cap \mathrm{succ}$, we can fix
the minimal~$\alpha < \lambda$ such that $\sigma \in V[G_\alpha]$ (see Lemma~\ref{lemm:sigma_appears_somewhere}).
Then in $V[G_\alpha]$, the node $\sigma$ belongs to $T_\alpha$,
 and,
letting $G(\alpha)$ be the corresponding filter for~$\QQ_\alpha$,
the set
\begin{equation}\label{eq:a_sigma}
a_\sigma = \bigcup \{ s^p_\sigma \with  p \in G(\alpha) \land \sigma \in \dom(p) \}
\end{equation}

is added by~$\QQ_\alpha$.
Back in the final model $V[G]$,
we let,
for
each
$\xi<\lambda$,
$$A_{\xi+1}:=\{ a_\sigma \with |\sigma|=\xi+1\}$$
(which is going to be a mad family).
Here, we are going to show that
the generic object
$\{ A_{\xi+1} \with \xi<\lambda \}$
is a refining system of 
almost disjoint families.

Our first lemma guarantees
that each $a_\sigma$ is going to have infinitely many~$1$'s. 
Note that, 
whenever
we write
``$s(m)=1$'', 
we
actually mean
``$m \in \dom(s)$ and $s(m)=1$''.

\begin{lemm}\label{lemm:total_domain}
Let $\alpha < \lambda$. 
For each $\sigma \in
T_\omikron$ and each $n \in \om$, the set
$$D_{\sigma,n} 
= \{ q \in \mtx_\omikron \with \sigma \in \dom(q) 
\textrm{ and } 
 |s^q_\sigma| \geq n \}$$ 
is dense in $\mtx_\omikron$.
In particular, $\dom(a_\sigma) = \om$ (i.e., $a_\sigma$ can be viewed as a subset of~$\om$). 
\end{lemm}

\begin{proof}
Let $\sigma\in
T_\omikron$, $n\in\omega$ and $p\in \mtx_\omikron$.
If $\sigma\notin \dom(p)$, extend $p$ to $p\cup \{ (\sigma,(\seqlangle\seqrangle,\emptyset,\emptyset))\}$.
From now on, we assume that $\sigma\in\dom(p)$.

First note that 
(2),(5), and (8) 
in Definition~\ref{defi:main_definition}
give restrictions on how $s_\tau^p$'s can be extended; however, 
it is always possible to extend an $s_\tau^p$ by $0$'s (provided that all $s_{\tau'}^p$ with $\tau' \init \tau$ are at least as long, as demanded by~(2)): 
no matter what the $s_{\tau'}^p$ with $\tau' \init \tau$ are, 
extending $s_\tau^p$ by $0$
never makes (5) false (due to~(2), $s_{\bar{\tau}}^p$ for $\tau \init \bar{\tau}$ are shorter and therefore do not matter at all); 
similarly, no matter what the 
$s^p_\pi$ 
are with 
$\pi$ having the same predecessor as $\tau$, extending $s_\tau^p$ by $0$
never makes (8) false.

Now, for every $\tau\in \dom(p)$ with $\tau \initeq \sigma$, 
if $|s_\tau^p|<n$, 
extend $s_\tau^p$ with $0$'s 
to 
length~$n$. 
In particular, the resulting condition~$q$ satisfies 
$|s^q_\sigma| \geq n$, 
as desired. 
\end{proof}

The next lemma shows that we can always assume that the domain of $f^p_\sigma$ and the domain of $h^p_\sigma$ is as large as possible.
Parts of the lemma
will be essential
also later, for the notion of ``full condition'' (see Definition~\ref{defi:full}).

\begin{lemm}\label{lemm:promises_are_dense}
Let $p \in \mtx_\omikron$ and $\sigma \in \dom(p)$.
Then there exists a $q \leq p$ such that $\dom(q) = \dom(p)$, and the following holds:
\begin{enumerate}
 \item[(a)]
$\tau\in \dom(f_\sigma^q)$ for each $\tau  \in \dom(q)$ with $\tau\lhd \sigma$, and

 \item[(b)]
						(letting $\sigma = \rho \concat \newalpha$)
				$\rho \concat \newbeta \in \dom(h_{\sigma}^q)$ for each $\newbeta<\newalpha$ with  							$\rho\concat \newbeta \in \dom(q)$.

\end{enumerate}

In particular, the set
$$D := \{ q \in \mtx_\omikron \with \textrm{ (a) and (b) holds for each } \sigma \in \dom(q) \}$$
is dense in $\mtx_\omikron$.

Moreover, if
$\alpha > 0$, then the following holds:
whenever
$\tau'\init \sigma$ with
$\tau'\in T'_\alpha$ (i.e., $a_{\tau'}$ has already been added before),
there exists $q\leq p$ such that $\dom(q)=\dom(p)$,
$q \in D$,
and
 $\tau'\in \dom(f^q_\sigma)$.
\end{lemm}

\begin{proof}
For every $\tau \in \dom(p)\setminus \dom(f^p_\sigma)$ with $ \tau \init \sigma$, let
$f_\sigma^q(\tau):=|s_\sigma^p|$.
For every $\rho \concat \newbeta \in \dom(p)\setminus \dom(h^p_\sigma)$ with $\newbeta<\newalpha$, let $h_\sigma^q(\rho \concat \newbeta) :=
|s_{\sigma}^p|$.
For the moreover part, let
$f^q_{\sigma}(\tau'):=|s_\sigma^p|$.
If we (repeatedly) extend $p$ in this way to $q$, it is clear that $q$ is a condition with the properties we wanted.
\end{proof}

The next lemma will be used to show that, for $\tau \lhd \sigma$, the set of conditions which force $a_\sigma\subseteq^* a_\tau$ is dense, 
as well as 
to show
 that, for $\rho \concat \newalpha$ and $\rho \concat \newbeta$ with $\newalpha\neq \newbeta$, the set of conditions which force $a_{\rho \concat \newalpha}\cap a_{\rho \concat \newbeta}=^* \emptyset$ is dense.

\begin{lemm}\label{lemm:forcing_almost_included_disjoint}
Let $p \in \mtx_\omikron$ 
and 
$\sigma \in \dom(p)$. 

\begin{enumerate}

\item 
If $\tau \in \dom(f^p_\sigma)$, then
$p \forces a_\sigma \setminus f^p_\sigma(\tau) \subs a_\tau$
(in particular,
$p \forces a_\sigma \subs^* a_\tau$).

\item If $\tau \in \dom(h^p_\sigma)$, then
$p \forces a_\tau \cap a_\sigma \subs h^p_\sigma(\tau)$
(in particular,
$p \forces a_\tau \cap a_\sigma =^* \emptyset$).

\end{enumerate}
\end{lemm}

\begin{proof}
For (1), we use 
Definition~\ref{defi:main_definition}(5). 
There are two cases: 
if $\tau\in T_\alpha$ (i.e., $\tau$ is a new node), it follows that $\tau\in \dom(p)$, 
and that 
$s^p_\tau(n)=1$ whenever 
$n \geq f_\sigma^p(\tau)$ and 
$s^p_\sigma(n)=1$; 
if $\tau\in T'_\alpha$ 
(i.e., $a_\tau$ has already been added by~$\QQ_\beta$ for some $\beta < \alpha$), 
it follows that $a_\tau(n)=1$ whenever 
$n \geq f_\sigma^p(\tau)$ 
and 
$s^p_\sigma(n)=1$. 
In both cases, 
using~\eqref{eq:a_sigma}, 
we get $p\forces a_\sigma \setminus f_\sigma^p(\tau) \subseteq a_\tau$.

The proof of~(2) is similar, using 
Definition~\ref{defi:main_definition}(8). 
\end{proof}

Finally, we can show that in
the final model~$V[\PP_\lambda]$, 
the sets along branches of~$\lala$ are $\subseteq^*$-decreasing, 
and 
the sets on 
any level of 
$\lala \cap \mathrm{succ}$
are pairwise almost disjoint.

\begin{coro}\label{coro:almost_included_disjoint}
In~$V[\PP_\lambda]$, the following hold:

\begin{enumerate}

\item 
If $\tau, \sigma \in
\lala \cap \mathrm{succ}$
such that $\tau \init \sigma$, then
$a_\sigma \subs^* a_\tau$.

\item 
If
$\rho \in
\lala$,
and $\newbeta < \newalpha < \lambda$, then $a_{\rho \concat \newbeta} \cap a_{\rho \concat \newalpha} =^* \emptyset$.
Indeed, the following holds. For each
$\sigma, \sigma' \in
\lala \cap \mathrm{succ}$
satisfying $|\sigma| = |\sigma'|$
and $\sigma \neq \sigma'$,
we have $a_{\sigma} \cap a_{\sigma'} =^* \emptyset$;
in other words,
for each $\xi < \lambda$,
$$A_{\xi+1}
=
\{ a_\sigma \with \sigma \in \lambda^{\xi+1} \}$$
is an almost disjoint family.

\end{enumerate}
\end{coro}

\begin{proof}
To show~(1), 
let $\eta<\lambda$ be minimal such that $\sigma \in (\lala)^{V[\PP_\eta]}$.
Lemma~\ref{lemm:total_domain}, 
Lemma~\ref{lemm:promises_are_dense}, and Lemma~\ref{lemm:forcing_almost_included_disjoint}(1) in particular imply that the set
$$\{ q \in \mtx_\eta \with q \forces a_\sigma \subs^* a_\tau \}$$
is dense. Hence $V[\PP_{\eta+1}]\models  a_\sigma \subs^* a_\tau$, and
this remains true in the final model.

To show~(2), let $\eta<\lambda$ be minimal such that $\rho \in (\lala)^{V[\PP_\eta]}$.
Lemma~\ref{lemm:total_domain}, 
Lemma~\ref{lemm:promises_are_dense}, and Lemma~\ref{lemm:forcing_almost_included_disjoint}(2) in particular imply that the set
$$\{ q \in \mtx_{\eta} \with q \forces a_{\rho \concat \newbeta} \cap a_{\rho \concat \newalpha} =^* \emptyset \}$$
is dense; this proves the first assertion of~(2).
To prove the second assertion of~(2), find
$\rho \in \lala$
with $\rho \init \sigma, \sigma'$ and
$\newalpha, \newbeta < \lambda$
with $\newbeta \neq \newalpha$
such that
$\rho \concat \newbeta \initeq \sigma$ and
$\rho \concat \newalpha \initeq \sigma'$, and apply the first assertion of~(2) 
as well as~(1).
\end{proof}

Finally, 
we show that each $a_\sigma$ 
is an infinite subset of~$\om$:

\begin{lemm}\label{lemm:infinite_set}
Let $\alpha < \lambda$. 
For each $\sigma \in
T_\omikron$ and each $n \in \om$, the set
$$D_{\sigma,n} := \{ q \in \mtx_\omikron \with \sigma \in \dom(q) 
\textrm{ and } 
\exists m \geq n (s^q_\sigma(m) = 1) \}$$
is dense in $\mtx_\omikron$.
In particular, 
$a_\sigma \in [\om]^{\om}$ 
(when viewed as a subset of~$\om$).
\end{lemm}

\begin{proof}
The proof proceeds by induction on~$\alpha < \lambda$.

Let $\sigma\in
T_\omikron$, $n\in\omega$ and $p\in \mtx_\omikron$.
By Lemma~\ref{lemm:total_domain}, 
we can assume that $\sigma\in\dom(p)$.
Let $N_0 \in \om$ be bigger 
than the maximal length of all the $s_\tau^p$ with $\tau\in \dom(p)$ and $\tau \initeq \sigma$, and bigger than~$n$. 
Let 
$$A:=\bigcup\{ \dom(f_\tau^p)\cap T'_\alpha \with \tau\in \dom(p) \land \tau \initeq \sigma \}.$$ 
If $A$ is empty (which in particular 
holds in case $\alpha = 0$, due to 
$T'_0 
= \emptyset$), let $m \in \om$ be arbitrary with $m \geq N_0$. 
Otherwise, 
let 
$N_1 \geq N_0$ 
be large enough such that 
$a_{\psi'} \setminus N_1 \subseteq a_{\psi}$  
for all 
$\psi,\psi'\in A$ 
with $\psi \initeq \psi'$ 
(see Corollary~\ref{coro:almost_included_disjoint}). 
Moreover, 
let $\psi^*$ be the longest element 
of the finite set~$A$, and let 
$m \geq N_1$ such that 
$a_{\psi^*}(m)=1$ 
(this is possible, since $a_{\psi^*}$ is infinite 
by induction, 
due to the fact that $\psi^* \in T'_\alpha$, and hence  
$\psi^* \in T_\beta$ for some $\beta < \alpha$).
Therefore $a_{\psi}(m) = 1$ for each $\psi \in A$.

Now, for every $\tau\in \dom(p)$ with $\tau \initeq \sigma$, extend $s_\tau^p$ with $0$'s 
to length~$m$. 
Finally, we 
extend
$s_\tau^p$ to
${s_\tau^p}  \concat 1$
for every 
$\tau\in \dom(p)$ with $\tau \initeq \sigma$ 
(in particular, for $\tau = \sigma$). 
It is easy to check that the resulting $q \leq p$ is indeed a condition, 
and $s^q_\sigma(m) = 1$, 
as desired. 
\end{proof}

Altogether, we have proved that
$\{ A_{\xi+1} \with \xi < \lambda \}$
is a refining system of ad families,
i.e., for each $\xi < \lambda$, $A_{\xi+1}$ is
an almost disjoint family,
and
for all $\xi < \xi' < \lambda$, $A_{\xi'+1}$ refines $A_{\xi+1}$.

To show that $\{ A_{\xi+1} \with \xi < \lambda \}$ is actually
a distributivity matrix requires much more work. The proof will be completed in Section~\ref{section:om_2_in_COM}. After
a lot of preparatory
work,
it will be shown in Section~\ref{subsection:rows} that the sets $A_{\xi+1}$ are indeed maximal, and in Section~\ref{subsection:branches} that the
sets along branches are indeed towers, which implies that there is no set intersecting the whole family~$\{ A_{\xi+1} \with \xi < \lambda \}$ (and hence there is no common refinement).

\subsection{Eligible sets and complete subforcings}\label{subsec:Q^c_complete}

The goal of this section is to show that our forcing~$\mtx_\omikron$ 
has complete subforcings which use only part of 
$T_\alpha$ 
(see Lemma~\ref{lemm:C_complete_subforcing}).
In Section~\ref{subsection:auxiliary},
this will be extended to the whole iteration (see Lemma~\ref{lemm:hereditarily_is_complete}),
which will be an important ingredient of the proof that the generic object is a distributivity matrix (see Section~\ref{subsection:branches} and Section~\ref{subsection:rows}).
Moreover, we will show in Section~\ref{section:h_b_om_1} that each $\mtx_\alpha$ (and hence our whole iteration) can be seen as an iteration of Mathias forcings with respect to
certain filters; to show that these filters are
$\mathcal{B}$-Canjar, we will again use  Lemma~\ref{lemm:C_complete_subforcing}.
Let us start with a concept which is going to be very useful:

\begin{defi}\label{defi:full}
A condition $p\in \mtx_\omikron$ is called \emph{full} if there exists an $N\in \omega$ such that for all $\sigma\in \dom(p)$
\begin{enumerate}
	\item $ |s_\sigma^p|=N$,
\item
$N>\max(\rang(f_\sigma^p))$ and $N> \max(\rang(h_\sigma^p) )$,

			\item $\tau\in \dom(f_\sigma^p)$ for each $\tau  \in \dom(p)$ with $\tau\lhd \sigma$, and

						\item
						(letting $\sigma = \rho \concat \newalpha$)
				$\rho \concat \newbeta \in \dom(h_{\sigma}^p)$ for each $\newbeta<\newalpha$ with  							$\rho\concat \newbeta \in \dom(p)$.

\end{enumerate}
Moreover,
$p \in \PP_\lambda$ is
full if 
$p(0)$ is full.
\end{defi}

Later, we will consider quotients $\PP_\lambda/\PP_\eta$ and therefore use a modification, where $0$ is replaced by $\eta$, i.e., $p(\eta)$ is full; see Remark~\ref{rema:eta_tail}.

The set of full conditions is dense:

\begin{lemm}\label{lemm:full}
 For every condition $p\in\mtx_\omikron$ there exists a full condition $q$ with $q\leq p$ and $\dom(q)=\dom(p)$. 
Hence 
the set of full conditions in $\PP_\lambda$ is dense in $\PP_\lambda$.

\end{lemm}
\begin{proof}
We can assume that $p$ belongs to the dense set $D$ from Lemma~\ref{lemm:promises_are_dense}, i.e.,
$p$ fulfills (3) and (4) for each $\sigma\in \dom(p)$.
Now let
\[
N>\max(\rang(f_\sigma^p)), \max(\rang(h_\sigma^p) ),  |s_\sigma^p|
\]
for every $\sigma\in\dom(p)$.
Finally, for every
$\sigma \in \dom(p)$, extend $s_\sigma^p$ with $0$'s to length~$N$.
It is easy to see that this
results in
a condition~$q$
which is full.
\end{proof}

We
now introduce
a notation for the
collection of conditions in~$\QQ_\omikron$
whose
domain is
contained in
a prescribed set of nodes:

\begin{defi}
Let $C \subs \lala$.
Define
\[
\FKWW_\omikron^C := \{ p \in \FKWW_\omikron \with \dom(p) \subs C \}.
\]
\end{defi}

In our completeness lemma below we are going to show that $\QQ_\omikron^C$ is a complete subforcing of~$\QQ_\omikron$ provided that $C$ has a
certain form.

\begin{defi}
Let $E\subseteq \lala$.
We call $E$ \emph{$\alpha$-$\leftup$-closed}
if
\begin{itemize}
\item for each $\sigma\in E$ and each $ \tau\lhd\sigma$ with $\tau\in T_\omikron$, we have
 $\tau \in E$,

	\item for each $\rho$ and $\newalpha$ with $\rho\concat \newalpha \in E$ and each  $ \newbeta<\newalpha$, we have $\rho\concat \newbeta \in E$.
\end{itemize}

We say that $C$ is \emph{$\alpha$-eligible} if 
$C=E \cup \bar{C}$,
where
$E \subs \lala$ is $\alpha$-$\leftup$-closed,
and either $\bar{C}$ is empty, or the following holds:
$E\subseteq \lambda^{<\gamma}$ for some $\gamma<\lambda$ and $\bar{C}\subseteq \lambda^{\gamma'}$ for some $\gamma' \geq \gamma$ (with~$\gamma'$ successor), and, 
for
$\sigma, \sigma' \in \bar{C}$,

\begin{enumerate}
\item either there
exist
$\rho$,
$i$ and $j$
such that $\rho\concat i=\sigma$ and $\rho\concat j = \sigma'$ 
(i.e., $\sigma$ and $\sigma'$ are in the same ``block''),

\item       or there exist
two incomparable nodes $\tau, \tau' \in E$
with
$\tau \lhd \sigma$ and $\tau' \lhd \sigma'$
(i.e., $\sigma$ and $\sigma'$ split within~$E$).

\end{enumerate}
\end{defi}

So 
an 
$\alpha$-eligible 
set consists of an $\alpha$-$\leftup$-closed part
together with nodes from 
one single later level.
Clearly, each 
$\alpha$-$\leftup$-closed set is 
$\alpha$-eligible.

The purpose of items (1) and (2) in the definition above is 
to ensure that after forcing with $\QQ_\omikron^C$, the two sets $a_{\sigma}$ and $a_{\sigma'}$ are almost disjoint (which is necessary for $\QQ_\omikron^C$ being a complete subforcing); this is guaranteed by either
item:
the two sets are forced to be almost disjoint either because it happens in the same block,
or because they are almost contained in almost disjoint sets which are already added by~$\QQ_\omikron^E$.

For $p\in \FKWW_\omikron$, let $p\doppelrestr C$ be the condition $p'$ with $\dom(p')=\dom(p)\cap C$, and $s_\sigma^{p'}=s_\sigma^{p}$, 
$f_\sigma^{p'}=f_\sigma^{p}\restr (C \cup T'_\alpha)$ 
and $h_\sigma^{p'}=h_\sigma^{p}\restr C$
for each $\sigma\in\dom(p')$.
Clearly, $p'$
is a condition
in~$\FKWW_\omikron^C$.
Note that if $C$ is $\alpha$-$\leftup$-closed, then $p\doppelrestr C=p\restr C$, because for every $\sigma\in\dom(p)\cap C$ clearly 
$f_\sigma^{p}\restr (C \cup T'_\alpha) = f_\sigma^{p}$ 
and $h_\sigma^{p}\restr C=h_\sigma^{p}$.

The
following crucial
completeness
lemma is given in a quite general form. This way, it can be used in Section~\ref{subsection:auxiliary}
as well as in Section~\ref{subsec:The_filters_are_B_Canjar}.
For Section \ref{subsection:auxiliary},
a somewhat
easier version would be enough
(see the proof of Lemma~\ref{lemm:hereditarily_is_complete}).

\begin{lemm}\label{lemm:C_complete_subforcing}
Let $C$
be 
$\alpha$-eligible. 
 Then
$\FKWW_\omikron^C$ is a complete subforcing
of~$\FKWW_\omikron$.
Moreover,
if $p\in \FKWW_\omikron$
is a full condition, then 
$p\doppelrestr C$ is a reduction of $p$ to $\FKWW_\omikron^C$.
\end{lemm}

Note that
the sets
$\lambda^1=\{ \sigma \in \lala\with |\sigma|=1\}$
and
$1^{<\lambda}= \{ \sigma \in \lala \with \sigma(\xi)=0$ $\text{for every } \xi\}$
are $0$-$\leftup$-closed, hence the forcings $\QQ_0^{(\lambda^1)}$ and $\QQ_0^{(1^{<\lambda})}$ are complete subforcings of $\QQ_0$ by the lemma. These forcings are
isomorphic
to the forcings introduced by Hechler \cite{Hechler} to
add a mad family and a tower, respectively
(compare with the respective definitions in~\cite{our_TOW_MAD}).

\begin{proof}[Proof of Lemma~\ref{lemm:C_complete_subforcing}]
We give the proof only for the case $\alpha=0$, and leave the (only slightly different) general case to the reader.

We first show that $\FKWW_\zeromikron^C\subseteq_{ic} \FKWW_\zeromikron$.
Let $p_0, p_1\in \FKWW_\zeromikron^C$ and $q\in \FKWW_\zeromikron$ with $q\leq p_0, p_1$. We have to show that there exists a condition $q'\in \FKWW_\zeromikron^C$ with $q'\leq p_0,p_1$.
Let $q':= q\doppelrestr C$. It is very easy to check that $q'$ is as we wanted.

Let $p\in \FKWW_\zeromikron$. 
To find a reduction, 
let $p'\leq p$ be a full condition
 with $\dom(p') = \dom(p)$
 (see Lemma~\ref{lemm:full}).
We will show that $p'\doppelrestr C$ is a reduction of $p$ to $\QQ_\zeromikron^C$.
Let 
$q\leq p'\doppelrestr C$ 
with $q\in \FKWW_\zeromikron^C$. We have to show that $q$ is compatible with $p$. To show this, we define a witness $r$ as follows.
Let $\dom(r):= \dom(p')\cup \dom(q)$.
For $\sigma\in \dom(q)$, let $s_\sigma^r:= s_\sigma^q$, and
for $\sigma\in \dom(q)\setminus \dom(p')$, let $f_\sigma^r:= f_\sigma^q$ and $h_\sigma^r:=h_\sigma^q$.

 For $\sigma\in\dom(q)\cap\dom(p')$, let $\dom(f_\sigma^r):= \dom(f_\sigma^q)\cup \dom(f_\sigma^{p'})$ and let
$f_\sigma^r(\sigma'):=\min(f_\sigma^q(\sigma'),f_\sigma^{p'}(\sigma'))$ for every $\sigma'\in \dom(f_\sigma^q)\cap  \dom(f_\sigma^{p'})$ and $f_\sigma^r(\sigma'):=f_\sigma^{p'}(\sigma')$ for $\sigma'\in \dom(f_\sigma^{p'})\setminus \dom(f_\sigma^q)$. Similarly, let
$\dom(h_\sigma^r):= \dom(h_\sigma^q)\cup \dom(h_\sigma^{p'})$ and let
$h_\sigma^r(\sigma'):=\min(h_\sigma^q(\sigma'),h_\sigma^{p'}(\sigma'))$ for every $\sigma'\in \dom(h_\sigma^q)\cap  \dom(h_\sigma^{p'})$ and $h_\sigma^r(\sigma'):=h_\sigma^{p'}(\sigma')$ for $\sigma'\in \dom(h_\sigma^{p'})\setminus \dom(h_\sigma^q)$.

 For $\sigma\in\dom(p')\setminus \dom(q)$, make the following definition. Let $f_\sigma^r:= f_\sigma^{p'}$ and $h_\sigma^r:= h_\sigma^{p'}$. If there is no $\tau\in \dom(q)$ with $\sigma\unlhd \tau$, let $s_\sigma^r:= s_\sigma^{p'}$. 
  If there exists $\tau\in \dom(q)$ with $\sigma\unlhd \tau$, extend $s_\sigma^{p'}$ to the maximal length of the $s_{\tau}^q$ for $\tau\in\dom(q)$ with $\sigma\lhd \tau$ in the following way: if $n\geq |s_\sigma^{p'}|$ and there exists $\tau\in \dom(p')$ which extends $\sigma$
  with $s_{\tau}^q(n)=1$,
let $s_\sigma^r(n)=1$, and let $s_\sigma^r(n)=0$
 otherwise. This makes sure that $s_\sigma^{r}(n)=1$ whenever $s_\tau^{r}(n)=1$ for $\sigma \init \tau$ and $\sigma\in \dom(f_\tau^{p'})$.

\begin{clai*}
$r$ is a condition.
\end{clai*}

\begin{proof}
It is very easy to check that $s_\sigma^r$, $f_\sigma^r$ and $h_\sigma^r$ are well-defined with the right domains and ranges for all $\sigma\in \dom(r)$.

If $\sigma \unlhd \tau$, then $|s_\sigma^r|\geq |s_\tau^r|$: if $\sigma$ and $\tau$ are both in $\dom(q)$, so $s_\sigma^r=s_\sigma^q$ and $s_\tau^r=s_\tau^q$, so the length is ok, because they are both from $q$; if $\sigma\notin \dom(q)$, we lengthened $s_\sigma^{r}$ to make it as long as all the $s$'s of $\tau$'s which extend it.

Let $\sigma, \tau\in \dom(r)$ with 
$\tau\in \dom(f_\sigma^r)$ and $m\geq f_\sigma^r(\tau)$ and $s_\sigma^r(m)=1$; we have to show that $s_\tau^r(m)=1$ 
(note that, in the general case, i.e., 
for $\QQ_\alpha$ with $\alpha > 0$, one also has to deal with the case where
$\tau\in \dom(f_\sigma^r)$, but
$\tau \notin \dom(r)$, which is analogous, but concerned 
with $a_\tau^r(m)$ in place of $s_\tau^r(m)$). 
\emph{Case 1:} $\sigma$ and $\tau$ are both in $\dom(q)$. It follows that $\tau\in C\cap \dom(f_\sigma^{r})= \dom(f_\sigma^q)$, $f_\sigma^r(\tau)=f_\sigma^q(\tau)$ and $s_\sigma^r=s_\sigma^q$ and $s_\tau^r=s_\tau^q$, so they fit together, because $q$ is a condition.
\emph{Case 2:} $\sigma\in\dom(q)$, $\tau\notin\dom(q)$. Since $\tau\notin\dom(q)$ and $\tau\in \dom(f_\sigma^r)$, it follows that $\tau\in \dom(f_\sigma^{p'})$. In particular $f_\sigma^{p'}$ is defined, so $\sigma\in \dom(p')$.
 If $m<|s_\tau^{p'}|$, it follows that $s_\tau^r(m)=s_\tau^{p'}(m)$ and $s_\sigma^r(m)=s_\sigma^{p'}(m)$, and $f_\sigma^r(\tau)=f_\sigma^{p'}(\tau)$. So $s_\tau^r(m)=1$, because $p'$ is a condition.
If $m\geq |s_\tau^{p'}|=|s_\sigma^{p'}|$, then $s_\sigma^{r}(m)=1$ implies that $s_{\rho}^q(m)=1$ for some $\rho\unrhd \sigma$, but then $\rho \unrhd \tau$, and therefore also $s_\tau^{r}(m)=1$.
\emph{Case 3:} $\sigma\notin\dom(q)$. So $f_\sigma^r=f_\sigma^{p'}$, and it follows that $\tau\in\dom(f_\sigma^{p'})\subseteq \dom(p')$.
If $m<|s_\sigma^{p'}|$, it follows that $s_\tau^{p'}(m)=s_\tau^{r}(m)=1$, because $p'$ is a condition and $\tau\in \dom(f_\sigma^{p'})$.
 If $m\geq|s_\sigma^{p'}|$, then our definition implies that there exists a $\rho$ with $\rho\rhd \sigma$, $\rho\in\dom(p')$ and $s_{\rho}^{q}(m)=1$. So both $\rho$ and $\tau$ are in $\dom(p')$, $\tau\in \dom(f_{\rho}^{p'})$ and $m\geq f_{\rho}^{p'}(\tau)$, hence $s_{\rho}^{r}(m)=1$ implies that $s_{\tau}^{r}(m)=1$ 
 by definition of $s_\tau^r$.
This finishes the proof that $s_\sigma^r$ and $s_\tau^r$ fit together (with respect to $f_\sigma^r$).

Assume $\rho, \rho' \in \dom(r)$ and
$\rho'\in \dom(h_\rho^r)$, $m\geq h_\rho^r(\rho')$ and $s_\rho^r(m)=1$; we have to show that $s_{\rho'}^r(m)=0$, if it is defined.
\emph{Case 1:}
$\rho, \rho'\in \dom(q)$.
The requirement follows, because $q$ is a condition.
\emph{Case 2:} $\rho, \rho'\in \dom(p')\setminus \dom(q)$. In this case the requirement holds, because $p'$ is a condition.
\emph{Case 3:} One of them is in $\dom(q)$, the other one not. 
Since $\rho'\in \dom(h_\rho^r)$, both are in $\dom(p')$ and $h_\rho^r(\rho')=h_\rho^{p'}(\rho')$ (because $q$ cannot
provide an $h$-value for a pair of two nodes
if not both of them are in $\dom(q)$). So for $m<|s_\rho^{p'}|$, the requirement holds, because it depends only on $p'$.
The form of $C$ implies that for at most one of the two nodes $\rho$ and $\rho'$ there exists a node in $C$  extending it. Therefore, for $m\geq |s_\rho^{p'}|$, only one of $s_\rho^{r}(m)$ and $s_{\rho'}^r(m)$
is
defined, and we have nothing to show.
This finishes the proof that $s_\rho^r$ and $s_{\rho'}^r$ fit together (with respect to $h_\rho^r$).
\end{proof}

It is straightforward to check that
$r$ 
extends both $q$ and $p'$ (and therefore~$p$).
\end{proof}

\section{No refinement, and madness of levels}\label{section:om_2_in_COM}

This section is dedicated to the central
part of the proof
that the generic object added by our forcing iteration is a distributivity matrix of height~$\lambda$: we will show that the levels are mad families and that there is no further refinement. This will be done in Section~\ref{subsection:rows} and Section~\ref{subsection:branches}, respectively.
Before that, we provide several preliminary lemmas and concepts.

\subsection{On forcing iterations and correct systems}
\label{subsec:general_forcing}

In this section, we give some lemmas about forcing
iterations (and completeness) in general,
i.e., they are not specific for our forcing from Definition~\ref{defi:main_definition}. We will need
them
for our proofs. For a good source about forcing iteration, see~\cite{Tools}.
Here $\PP$, $\QQ$, etc.\
are arbitrary forcing notions.

For two
forcing notions $\PP$ and $\PP'$, let $\PP' \compl \PP$ denote that $\PP'$ is a complete subforcing of~$\PP$.
Recall that $\PP' \compl \PP$ if and only if
\begin{enumerate}

\item $\PP' \subseteq_{ic} \PP$, i.e.,
for each $q, q' \in \PP'$, we have
$q \incomp_{\PP'} q' \implies q \incomp_{\PP} q'$, and

\item
for each condition $p \in \PP$, there is
$q \in \PP'$ such that $q$ is a
\emph{reduction of~$p$ to~$\PP'$},
i.e., for each $r \in \PP'$ with $r \leq q$, we have $r \comp_{\PP} p$.

\end{enumerate}

Let us first recall two easy facts:

\begin{lemm}\label{lemm:compl_compl_is_compl}
Suppose that $\PP_0 \compl \PP$ and\footnote{In fact,
$\PP_1 \subseteq_{ic} \PP$ is
sufficient for the proof to go through.} $\PP_1 \compl \PP$ satisfying $\PP_0 \subs \PP_1$.
Then $\PP_0 \compl \PP_1$.
Moreover, if $q \in \PP_0$
is a reduction of $p \in \PP_1$ from~$\PP$ to $\PP_0$, then $q$ is also a reduction of~$p$ from $\PP_1$
to~$\PP_0$.
\end{lemm}

\begin{lemm}
\label{lemm:Resch_lemma}
Suppose that
$\PP' \compl \PP$.
Let $\varphi$ be some formula, let $\name{x}$, $\name{y}$, etc.\ be
$\PP'$-names,
and let $p \in \PP$ such that
$p \forces_{\PP} \varphi(\name{x}, \name{y}, \ldots)$.
Then
for each $p' \in \PP'$ which is a reduction of~$p$,
we have 
$p'
\forces_{\PP'}
\varphi(\name{x}, \name{y}, \ldots)$.
\end{lemm}

Let us now recall the following well-known fact (see, e.g., \cite{Brendle_Vera}):
If $\{ \PP_\alpha, \name{\mtx}_\alpha \with \alpha < \delta \}$ and  $\{ \PP'_\alpha, \name{\mtx}'_\alpha \with \alpha < \delta \}$
are finite support iterations
such that
$\forces_{\PP_\alpha} \name{\mtx}'_\alpha \compl \name{\mtx}_\alpha$ for each $\alpha < \delta$, then $ \PP'_\delta$ is complete in $ \PP_\delta$. 
We will need the following 
technical strengthening 
of this fact.

\begin{lemm}\label{lemm:general_iteration_reduction}
Let $\{ \PP_\alpha, \name{\mtx}_\alpha \with \alpha < \delta \}$ and  $\{ \PP'_\alpha, \name{\mtx}'_\alpha \with \alpha < \delta \}$
be\footnotemark{} finite support iterations
such that for each $\alpha < \delta$,
\[
\forces_{\PP_\alpha} \name{\mtx}'_\alpha \compl \name{\mtx}_\alpha.
\]
Then
$\PP'_{\delta}$ is a complete subforcing
of $\PP_{\delta}$.

\footnotetext{We do not seem to need here that our finite support iterations are c.c.c.; however,
finite support iterations
of non-c.c.c.\ iterands
collapse cardinals. We will use the lemma for our forcings from Definition~\ref{defi:main_definition}, so,  in our application,
everything will have the c.c.c.\ anyway.}

Moreover, if
$\Resch{:}\ \mtx_0 \rightarrow \mtx'_0$
is a map such that $\Resch(q)$ is a reduction of~$q$ for each $q \in \mtx_0$,
then
for each $p \in \PP_\delta$,
there is a $p' \in \PP'_\delta$ such that $p'$ is a reduction of $p$,
and
$p'(0) = \Resch(p(0))$, and,
if
$\alpha \geq 1$ and
$p(\alpha)$
is a $\PP'_\alpha$-name with
$p \restr \alpha \forces p(\alpha) \in \name{\QQ}'_\alpha$,
then $p'(\alpha) = p(\alpha)$.
\end{lemm}

In fact, 
the iterands 
in the above lemma 
need not be 
separative, which is 
essential, 
because 
we are going to apply 
it 
to 
our forcings~$\QQ_\alpha$ 
from Definition~\ref{defi:main_definition}, which are not separative (see Remark~\ref{rema:not_sep}).

The following concept has been introduced by Brendle (see, e.g., \cite{Luminy_lecture_notes} and
\cite{brendle2021forcing}):

\begin{defi}
A system of forcings $\RR_0,\RR_1 \compl \RR$ with $\RR_0\cap \RR_1\compl \RR_0, \RR_1$ is \emph{correct} if
any two conditions $p_0\in \RR_0$ and $p_1\in \RR_1$ which have a common reduction in $\RR_0\cap \RR_1$ are compatible in $\RR$.
\end{defi}

 In the following lemma, we are considering a system where $\RR=\PP * \name{\QQ}$,
$\RR_0=\PP$, and $\RR_1=\PP' * \name{\QQ}'$.
It is easy to check that, under the assumptions of the lemma, this is a correct system.
We do not know, however, whether the conclusion of the lemma holds for every correct system.

\begin{lemm}\label{lemm:wotan}
Let $\PP * \name{\QQ}$ and
$\PP' * \name{\QQ}'$ be two-step iterations satisfying
$\PP' \compl \PP$ and
$\forces_{\PP} \name{\QQ}' \compl \name{\QQ}$.
Then
\[
V[\PP' * \name{\QQ}']
\cap V[\PP] = V[\PP'].
\]
\end{lemm}

\begin{proof}
We will only show the special case which we will need later (it is straightforward to extend the proof
to the general case): for any $\delta, \varepsilon \in \ord$,
\[
\delta^\varepsilon
\cap
V[\PP' * \name{\QQ}']
\cap V[\PP] \subs V[\PP'].
\]

Let $G$ be a generic filter for $\PP'$, and let $\name{f}_0$ be a $\PP$-name, and let $\name{f}_1$ be a $\PP' * \name{\QQ}'$-name.
Work in $V[G]$.
Assume towards contradiction that there is a condition
$(p,\name{q}) \in \PP * \name{\QQ}$ with $p \in \PP/G$ such that
\begin{equation}\label{eq:names_are_equal}
(p,\name{q}) \forces \name{f}_0 = \name{f}_1 \in \ord^{<\ord}  \;\land\;
\name{f}_0 \notin V[G].
\end{equation}

Let $p' \in G$ be a reduction of $p$ to~$\PP'$.
By standard arguments, we can fix 
a 
$\PP'$-name $\name{q}'$
such that $p \forces``\name{q}'$ is a reduction of $\name{q}$'' and $(p',\name{q}')\in \PP'*\name{\QQ}'$.

Since $p$ is reduction of $(p,\name{q})$ to $\PP$,
it follows from~\eqref{eq:names_are_equal} and Lemma~\ref{lemm:Resch_lemma} that
$p \forces \name{f}_0 \notin V[G]$.
Therefore, we can fix $\gamma \in \varepsilon$ such that  $p$ does not decide $\name{f}_0(\gamma)$ in $\PP/G$.
Let $(p_1,\name{q}_1)\leq (p',\name{q}')$ and $\xi_1\in \delta$ such that $p_1\in G$ and $(p_1,\name{q}_1)\forces \name{f}_1(\gamma)=\xi_1$. Since $p$ does not decide $\name{f}_0$ at $\gamma$, we can fix $p_0\in \PP/G$ with $p_0\leq p$ and $ \xi_0 \in \delta$ with $\xi_0\neq \xi_1$ such that $p_0\forces \name{f}_0(\gamma)=\xi_0$.
Now we want to find a condition $(p^*,\name{q}^*)$ which is stronger than $(p,\name{q}), (p_1,\name{q}_1)$ and $(p_0,\weakest)$.

First note that $p_0$ and $p_1$ are compatible, because $p_0\in \PP/G$ and $p_1\in G$, and fix $p^*\leq p_0,p_1$.
Since $p^*\leq p, p_1$ it follows that $p^*\forces ``\name{q}'$ is a reduction of $\name{q}$ and $\name{q}_1\leq \name{q}'$'' hence $p^*\forces \name{q}_1 \comp \name{q}$. Let $\name{q}^*$ be a $\PP$-name such that $p^*\forces \name{q}^*\leq \name{q}_1, \name{q}$. It is easy to check that $(p^*,\name{q}^*)\leq (p,\name{q}), (p_1,\name{q}_1), (p_0,\weakest)$.
Now $(p^*,\name{q}^*)\forces \name{f}_0=\name{f}_1\land \name{f}_0(\gamma)=\xi_0 \land \name{f}_1(\gamma)=\xi_1$, but $\xi_0\neq \xi_1$, a contradiction.
\end{proof}

We conclude
with
an
easy observation we will need later on:

\begin{lemm}\label{lemm:finde_r_und_m}
Suppose that $\PP' \compl \PP$, and $\name{b}$ is a $\PP'$-name, and $p \in \PP$ is such that
$p \forces \name{b} \in [\om]^{\om}$.
Then for each $N \in \om$ there exists $r \in \PP'$ and $m > N$ such that
$r \forces m \in \name{b}$
and $r$ is compatible with $p$.
\end{lemm}

\subsection{Complete subforcings: hereditarily below $\gamma$}\label{subsection:auxiliary}

In this section, we give some technical definitions and lemmas as a preparation for the main proofs in Section~\ref{subsection:branches} and Section~\ref{subsection:rows}.
More precisely, we define, for each $\gamma<\lambda$, the subforcings of ``hereditarily below~$\gamma$'' conditions of our iteration and show that they form complete subforcings (see Lemma~\ref{lemm:hereditarily_is_complete}). Furthermore, we show that each condition
is hereditarily below~$\gamma$ for some $\gamma<\lambda$ (see Lemma~\ref{lemm:finding_gamma_fresh}).

Let us now provide the following recursive definition (we give the definition for the entire iteration but we will actually need it for tails of the iteration; see Remark~\ref{rema:eta_tail}):

\begin{defi}
Let $\gamma<\lambda$. By recursion on $\alpha\leq \lambda$ we define when a condition $p\in \PP_\alpha$ is 
\emph{hereditarily below~$\gamma$} (and introduce the notation $\below{\gamma}\PP_\alpha$):
\begin{enumerate}
\item $p\in \FKWW_0$ is hereditarily below $\gamma$, if $\dom(p)\subseteq  \gamma^{<\gamma}$.
\item Let $\below{\gamma}\PP_{\alpha}:= \{ p\in \PP_\alpha \with p \text{ hereditarily below } \gamma\}$.

	\item $p\in \PP_{\alpha+1}$ is hereditarily below $\gamma$, if $p\restr \alpha$ is hereditarily below $\gamma$ and $p(\alpha)$ is a $\below{\gamma}\PP_{\alpha}$-name, and $p\restr \alpha \forces \dom(p(\alpha))\subseteq \gamma^{{<}\gamma}.$
	\item For $\alpha$ limit, $p\in \PP_{\alpha}$ is hereditarily below $\gamma$, if $p\restr \untenbeta$ is hereditarily below $\gamma$, for every $\untenbeta<\alpha$.

\end{enumerate}
For $\alpha\leq \lambda$
a $\PP_{\alpha}$-name $\dot{b}$ is \emph{hereditarily below $\gamma$}, if for all $(\dot{x},p)\in \dot{b}$,  $p\in \below{\gamma}\PP_\alpha$ and $\dot{x}$ is hereditarily below $\gamma$ (this is by recursion).
\end{defi}

Clearly, if $p\in\PP_\alpha$ is hereditarily below $\gamma$ and $\gamma'>\gamma$, then $p$ is also hereditarily below $\gamma'$. The same holds for a $\PP_\alpha$-name $\dot{b}$.

\begin{rema}\label{rema:epsilon_delta_neu}
In the more general situation described in Remark~\ref{rema:more_general}, i.e., if we work with the tree $\theta^{<\lambda}$ in place of $\lala$, we would rather need a pair of ordinals $(\varepsilon,\delta)$ in place of $\gamma$, where $\varepsilon<\theta$ and $\delta<\lambda$, and $\gaga$ in the definition 
above (as well as in Definition~\ref{defi:almost_hereditarily_below_gamma} below)
would be replaced by $\varepsilon^{<\delta}$ (see also Remark~\ref{rema:why_cof_theta_large}).
\end{rema}

\begin{defi}\label{defi:almost_hereditarily_below_gamma}
Let $\gamma<\lambda$ and $\tau\in \lala$. By recursion on $\alpha\leq \lambda$ we define when a condition $p\in \PP_\alpha$ is \emph{almost hereditarily below $\gamma$ except for $\tau$} (and introduce the notation~$\belowt{\gamma}{\tau}\PP_\alpha$):
\begin{enumerate}
	\item $p\in \FKWW_0$ is almost  hereditarily below $\gamma$ except for $\tau$, if $\dom(p)\subseteq\gamma^{<\gamma}\cup \{\tau\}$.
\item Let $\belowt{\gamma}{\tau}\PP_{\alpha}:= \{ p\in \PP_\alpha \with p \text{ almost hereditarily below } \gamma \text{ except for }\tau\}$.

	\item $p\in \PP_{\alpha+1}$ is almost hereditarily below $\gamma$ except for $\tau$, if $p\restr \alpha$ is almost hereditarily below $\gamma$ except for $\tau$ and $p(\alpha)$ is a\footnote{\label{footnote:not_a_typo}This is not a typo:
we really require $p(\alpha)$ to be a $\below{\gamma}\PP_{\alpha}$-name,
not just a $\belowt{\gamma}{\tau}\PP_{\alpha}$-name.} $\below{\gamma}\PP_{\alpha}$-name, and $p\restr \alpha \forces \dom(p(\alpha))\subseteq \gamma^{{<}\gamma}.$
	\item For $\alpha$ limit, $p\in \PP_{\alpha}$ is almost hereditarily below $\gamma$ except for $\tau$, if $p\restr \untenbeta$ is almost hereditarily below $\gamma$ except for $\tau$, for every $\untenbeta<\alpha$.
\end{enumerate}
For $\alpha\leq \lambda$
a $\PP_{\alpha}$-name $\dot{b}$ is \emph{almost hereditarily below $\gamma$ except for $\tau$}, if for all $(\dot{x},p)\in \dot{b}$, both $p$ and $\dot{x}$ are almost hereditarily below $\gamma$ except for $\tau$.
We will write \emph{almost hereditarily below $\gamma$} and omit the $\tau$ if it is clear from the context which $\tau$ is meant.
\end{defi}

Clearly, if $p\in\PP_\alpha$ is almost hereditarily below $\gamma$ and $\gamma'>\gamma$, then $p$ is also almost hereditarily below $\gamma'$ and if $p\in\PP_\alpha$ is hereditarily below $\gamma$, then it is almost hereditarily below $\gamma$ except for $\tau$ for every $\tau$. The same holds for a $\PP_\alpha$-name $\dot{b}$.

\begin{rema}\label{rema:eta_tail}
As mentioned above, we will need several of our concepts for tails of the iteration instead of the whole iteration.
We will later have the following situation:
$\eta < \lambda$ will be fixed, and we will work in $V[G_\eta]$ for a fixed generic filter $G_\eta \subs
\PP_\eta$.
We will use variants of the above definitions and the subsequent lemmas for
the tail iteration
$\{ \PP_\alpha/G_\eta, \name{\mtx}_\alpha \with \eta \leq \alpha < \lambda \}$.
In the definitions and lemmas,
$\QQ_\eta$ plays the role of~$\QQ_0$
(see for example Lemma~\ref{lemm:statt_oben}(3)).
So, e.g., in the definition of almost hereditarily below $\gamma$ except for $\tau$
(with  $\tau \in  (\lala)^{V[G_\eta]}$),
we want $\dom(p(\eta))\subseteq (\gaga)^{V[G_\eta]}\cup \{\tau\}$.
\end{rema}

Before proving completeness,
let us recall that
$\gaga$ is 
$\alpha$-$\leftup$-closed; actually, we will
need a bit more:

\begin{lemm}\label{lemm:lala_is_left_up}
Assume $\PP'_\alpha$ is a complete subforcing of $\PP_\alpha$, and $G$ is generic for~$\PP_\alpha$.
Then in $V[G]$, the set $(\gaga)^{V[G \cap \PP'_\alpha]}$ is $\alpha$-$\leftup$-closed.
\end{lemm}
\begin{proof}
Suppose $\sigma$ and $ \rho\concat i $ belong to $ (\gaga)^{V[G \cap \PP'_\alpha]}$.
Note that the following holds in~$V[G]$:
$\sigma\restr (\xi+1) \in V[G \cap \PP'_\alpha]$
for each $\xi<|\sigma|$,
and $\rho \concat j \in V[G \cap \PP'_\alpha]$  for each $j < i$.
Therefore $(\gaga)^{V[G \cap \PP'_\alpha]}$ is $\alpha$-$\leftup$-closed.
\end{proof}

We can now show that the subforcing of conditions which are (almost) hereditarily below $\gamma$ is a complete subforcing:

\begin{lemm}\label{lemm:hereditarily_is_complete}
Let $\gamma < \lambda$.
Then
$\below{\gamma}\PP_{\lambda}$ is a complete subforcing of $\PP_{\lambda}$.

Also, if $\tau\in\lala$ is such that either
\begin{enumerate}

\item[(1)] $|\tau| \geq \gamma$, or
\item[(2)]
$\tau$ is such that $\gamma^{<\gamma} \cup \{\tau\}$ is $0$-$\leftup$-closed,
\end{enumerate}
then $\belowt{\gamma}{\tau}\PP_{\lambda}$ is a complete subforcing of $\PP_{\lambda}$.

Moreover, if $p$ is full and almost hereditarily below $\gamma$ except for $\tau$, then
$$
(p(0)\restr \gamma^{<\gamma}, p(1), p(2),\dots)
$$
is a reduction of $p$ to $\below{\gamma}\PP_\lambda$.
\end{lemm}

\begin{proof}
We show by induction on $\alpha$ that $\below{\gamma}\PP_{\alpha}$ (as well as $\belowt{\gamma}{\tau}\PP_{\alpha}$) is a complete subforcing of $\PP_{\alpha}$, for $1\leq\alpha\leq \lambda$.
In fact, we will define
$\below{\gamma}{\PP_\alpha}$-names $\dot{\QQ}'_\alpha$ such that $\below{\gamma}{\PP_\delta}$ (or $\belowt{\gamma}{\tau}{\PP_\delta}$, respectively) is the finite support iteration of the $\dot{\QQ}'_\alpha$'s; the only difference of the two iterations will be the first iterand~$\QQ'_0$.

(Initial step $\alpha=1$)
Note that $\below{\gamma}\PP_{1} = \QQ_0^{\gaga}$ is a complete subforcing of $\PP_{1} = \QQ_0$:
this is an easy
instance
of Lemma~\ref{lemm:C_complete_subforcing}, 
because $\gaga$ is
$0$-$\leftup$-closed.
Similarly,
$\belowt{\gamma}{\tau}\PP_{1} = \QQ_0^{\gaga \cup \{\tau\}}$ is a complete subforcing of $\QQ_0$:
in case~(2) holds, 
$\gaga \cup \{\tau\}$ is $0$-$\leftup$-closed by assumption;
in case~(1) holds, 
$\gaga \cup \{\tau\}$ 
is easily seen to be 
$0$-eligible. 
Take $\QQ'_0= \QQ_0^{\gaga}$ in the iteration representing $\below{\gamma}\PP_\lambda$, and take $\prescript{}{}{\QQ'_0}= \QQ_0^{\gaga\cup\{\tau\}}$  in the iteration representing $\belowt{\gamma}{\tau}\PP_\lambda$.

(Successor step $\alpha+1$)
Assume that $\below{\gamma}{\PP_\alpha}$ and $\belowt{\gamma}{\tau}{\PP_\alpha}$ are complete subforcings of $\PP_\alpha$. We show that  $\below{\gamma}{\PP_{\alpha+1}}$ and $\belowt{\gamma}{\tau}{\PP_{\alpha+1}}$ are complete subforcings of $\PP_{\alpha+1}$. In $V[G]$, for $G$ generic for~$\PP_\alpha$, let\footnote{Note that $E$ is really defined this way for both cases (see also footnote~\ref{footnote:not_a_typo}).
} $E:=(\gaga)^{V[G\cap\below{\gamma}{\PP_\alpha}]}$; by Lemma~\ref{lemm:lala_is_left_up}, $E$ is $\alpha$-$\leftup$-closed, so Lemma~\ref{lemm:C_complete_subforcing} implies 
that in $V[G]$,  $\QQ_\alpha^E$ is a complete subforcing of $\QQ_\alpha$.
We use the following, which we will prove
after finishing the proof of the lemma:

\begin{clai}\label{clai:s_f_h_basically_irrelevant}
$\QQ_\alpha^E$ is an element of $V[G\cap\below{\gamma}{\PP_\alpha}]$.
\end{clai}
\noindent
Using the claim,
we can fix a $\below{\gamma}{\PP_\alpha}$-name $\dot{\QQ}'_\alpha$ for  $\QQ_\alpha^{E}$.
Since $\below{\gamma}{\PP_\alpha} \subs \belowt{\gamma}{\tau}{\PP_\alpha}$ and
both are complete subforcings of $\PP_\alpha$,
Lemma~\ref{lemm:compl_compl_is_compl} implies that
$\below{\gamma}{\PP_\alpha}$ is a complete
subforcing of $\belowt{\gamma}{\tau}{\PP_\alpha}$,
so the $\below{\gamma}{\PP_\alpha}$-name $\dot{\QQ}'_\alpha$ is also a $\belowt{\gamma}{\tau}{\PP_\alpha}$-name.
So we can apply Lemma~\ref{lemm:general_iteration_reduction} to obtain that $\below{\gamma}{\PP_\alpha}*\dot{\QQ}'_\alpha$ and $\belowt{\gamma}{\tau}{\PP_\alpha}*\dot{\QQ}'_\alpha$ are complete subforcings of $\PP_{\alpha+1}$.
By definition, $\below{\gamma}{\PP_\alpha}*\dot{\QQ}'_\alpha$ is equivalent to  $\below{\gamma}\PP_{\alpha+1}$, and $\belowt{\gamma}{\tau}{\PP_\alpha}*\dot{\QQ}'_\alpha$ is equivalent to  $\belowt{\gamma}{\tau}\PP_{\alpha+1}$, so the successor step is finished.

(Limit step $\alpha$)
It follows by Lemma~\ref{lemm:general_iteration_reduction} that the limit of the finite support iteration of the $\dot{\QQ}'_{\alpha'}$ with $\alpha'<\alpha$ is a complete subforcing of $\PP_\alpha$, and by definition, $\below{\gamma}\PP_\alpha$ (or $\belowt{\gamma}{\tau}\PP_\alpha$ in the other case)  is equivalent to the limit of this finite support iteration.

 Now let us show the moreover part.
By the moreover part of Lemma~\ref{lemm:C_complete_subforcing}, $p(0)\restr \gaga$ is a reduction of $p(0)$ to $\below{\gamma} \PP_1$ (which is $\QQ'_0$ in the iteration representing $\below{\gamma}\PP_\lambda$). Since $p\in \belowt{\gamma}{\tau}\PP_\lambda$, which is the iteration of the $\dot{\QQ}'_\alpha$ (for $\alpha\geq 1$ the iterands of the two iterations coincide),
so $p \restr \alpha \forces p(\alpha) \in \QQ'_\alpha$ for $\alpha\geq 1$,
therefore Lemma~\ref{lemm:general_iteration_reduction} completes the proof.
\end{proof}

\begin{proof}[Proof of Claim~\ref{clai:s_f_h_basically_irrelevant}]
We work in $V[G]$.
Let $G_\beta := G \cap \PP_\beta$.
Let $T'_\alpha=  \bigcup_{\untenbeta < \alpha} (\lala \cap \mathrm{succ})^{V[G_{\untenbeta}]}$ and
$T_\alpha = (\lala \cap \mathrm{succ})^{V[G]} \setminus T'_\alpha$, as in the definition of $\mtx_\alpha$.

It is straightforward to check that
$\mtx_\alpha^E$ can be defined in $V[G\cap\below{\gamma}\PP_\alpha]$ provided that
$E\cap T'_\alpha$ (and hence also $E\cap T_\alpha$)  belongs to
$V[G\cap\below{\gamma}\PP_\alpha]$.
First note that
$$
E= (\gaga)^{V[G \cap \below{\gamma}\PP_\alpha]}=\gaga\cap V[G \cap \below{\gamma}\PP_\alpha]
$$
 and
$$
T'_\alpha=\bigcup_{\beta<\alpha}(\lala \cap \mathrm{succ})^{V[G_{\untenbeta}]}=\bigcup_{\beta<\alpha}(\lala \cap \mathrm{succ}\cap{V[G_{\untenbeta}]}).
$$
 So
$$
E\cap T'_\alpha=\bigcup_{\beta<\alpha}(\gaga\cap \mathrm{succ}\cap V[G \cap \below{\gamma}\PP_\alpha] \cap{V[G_{\untenbeta}]}).
$$
Apply Lemma~\ref{lemm:wotan} to
$\PP_\beta * \name{\QQ}$, where $\name{\QQ}$ is the quotient $\PP_\alpha/\PP_\beta$,  and
$\below{\gamma}\PP_\beta * \name{\QQ}'$,
where $\name{\QQ}'$ is the quotient $\below{\gamma}\PP_\alpha/\below{\gamma}\PP_\beta$ (which is possible since
$\below{\gamma}\PP_\beta \compl \PP_\beta$ by induction hypothesis, and
$\forces_{\PP_\beta} \name{\QQ}' \compl \name{\QQ}$
by Lemma~\ref{lemm:general_iteration_reduction} for the
tail iterations)
to obtain
\[
\gaga \cap V[G \cap \below{\gamma}\PP_\alpha] \cap{V[G_{\untenbeta}]} = \gaga\cap V[G_{\untenbeta}\cap \below{\gamma}\PP_\alpha].
\]
Therefore,
\[
E\cap T'_\alpha=\bigcup_{\beta<\alpha}(\gaga\cap \mathrm{succ})^{V[G_{\untenbeta}\cap \below{\gamma}\PP_\alpha]},
\]
which clearly belongs to $V[G\cap\below{\gamma}\PP_\alpha]$, as desired.
\end{proof}

The next lemma shows that every condition in $\PP_\lambda$ is (essentially) hereditarily below~$\gamma$ for some $\gamma<\lambda$.

\begin{lemm}\label{lemm:finding_gamma_fresh}
For every $p\in \PP_{\lambda}$, there exists a $\gamma<\lambda$ and a condition
 $p'\in \below{\gamma}\PP_{\lambda}$ which is 
forcing 
equivalent to~$p$.
\end{lemm}

\begin{proof}
We will actually show by induction on $\alpha$ that for every $p\in \PP_{\alpha}$, there exists a $\gamma<\lambda$ and a condition~$p'$ (forcing) equivalent to $p$ such that $p'\in \below{\gamma}\PP_{\alpha}$.

(Initial step $\alpha=1$) Given $p\in\PP_1=\QQ_0$, note that $\dom(p)\subseteq \lala$ is finite. So the maximum length of the nodes $\sigma$ in the domain as well as the maximal entry of the nodes are bounded, i.e., there is $\gamma<\lambda$ such that $\dom(p)\subseteq \gaga$. So $p\in \below{\gamma}\PP_{1}$.

(Limit step $\alpha$) Let $p\in \PP_\alpha$. By induction hypothesis, for each $\beta<\alpha$ there exists $\gamma_\beta$ such that $p\restr \beta \in \below{\gamma_\beta}\PP_{\beta}$. Since we are using finite support, there exists $\beta^*<\alpha$ which is an upper bound of the support of $p$. Then $p\in \below{\gamma_{\beta^*}}\PP_\alpha$.

(Successor step $\alpha+1$)
Let $(p,\dot{q})\in \PP_{\alpha}*\dot{\QQ}_{\alpha}$.
First, by the induction hypothesis, we can assume without loss of generality that there exists $\gamma_p<\lambda$ such that $p\in \below{\gamma_p}\PP_\alpha$.
We will describe a name~$\dot{q}'$ which is equivalent to $\dot{q}$ (more precisely, $p\forces \dot{q}=\dot{q}'$) and analyze it, to find a $\gamma<\lambda$ such that $\dot{q}'$ is a $\below{\gamma}\PP_\alpha$-name and $p\forces \dom(\dot{q}')\subseteq \gaga$.

\begin{clai}\label{clai:osterhase}
Let $\dot{\sigma}$ be a
$\PP_\alpha$-name of a
sequence of
ordinals of length
less than~$\lambda$;
then there exists a $\gamma < \lambda$
such that
there exists a
$\below{\gamma}\PP_\alpha$-name which is equivalent to~$\dot{\sigma}$.
The same holds true if $\name{\sigma}$ is a name for a finite sequence of such
sequences.
\end{clai}

\begin{proof}
Clearly, by c.c.c., there exists a $\delta<\lambda$ which is an upper bound for the length of~$\dot{\sigma}$.
Since $\PP_\alpha$ has the c.c.c.,
for each $\xi < \delta$, $\dot{\sigma}(\xi)$ is represented by a countable antichain.
So only $|\delta| \cdot \aleph_0$ many (hence less than $\lambda$ many) conditions appear in $\dot{\sigma}$.
By inductive hypothesis we can assume that each of these conditions belongs to $\below{\gamma}\PP_\alpha$ for some $\gamma<\lambda$, so we can fix a $\gamma_{\dot{\sigma}}<\lambda$ which is an upper bound of all the appearing $\gamma$.
So $\dot{\sigma}$ is actually a $\below{\gamma_{\dot{\sigma}}}\PP_\alpha$-name.
The statement about names for finite sequences of sequences follows easily.
\end{proof}

Now, let $\dot{N}$ be a
$\PP_\alpha$-name such that $p\forces |\dom(\dot{q})|=\dot{N}$; by
(a simple instance of)
Claim~\ref{clai:osterhase}, we can fix $\gamma_{\dot{N}}<\lambda$
and assume that
$\dot{N}$ is a
$\below{\gamma_{\dot{N}}}\PP_\alpha$-name.
To represent $\dot{q}$, we provide $\omega$-sequences $\seqlangle \dot{\sigma}_k \with k\in \omega\seqrangle$ (of names for potential members of $\dom(\dot{q})$)
and $\seqlangle (\dot{s}_k, \dot{f}_k, \dot{h}_k) \with k\in \omega\seqrangle$
such that
\[
p\forces \dom(\dot{q})= \{\dot{\sigma}_k \with k \in \dot{N}\}
\;\land\; \forall k \in \dot{N}
\;( \dot{q}(\dot{\sigma}_k) = (\dot{s}_k, \dot{f}_k, \dot{h}_k)),
\]
where $\dot{\sigma}_k$ is forced to be a sequence of ordinals of length less than $\lambda$,
and
$\dot{f}_k$ and $\dot{h}_k$
 can be represented
as finite sequences of such sequences, together with finite sequences of natural numbers,
 and
$\dot{s}_k$ is forced to be an element of $2^{<\om}$.
Using Claim~\ref{clai:osterhase}, we can find $\gamma' < \lambda$, larger than $\gamma_{\dot{N}}$,
such that
there exist
$\below{\gamma'}\PP_\alpha$-names $\dot{\sigma}'_k$, $\dot{s}'_k$, $\dot{f}'_k$, and $\dot{h}'_k$ which are equivalent
to $\dot{\sigma}_k$, $\dot{s}_k$, $\dot{f}_k$, and $\dot{h}_k$, respectively.
By replacing
all $\dot{\sigma}_k$, $\dot{s}_k$, $\dot{f}_k$, and $\dot{h}_k$
in $\dot{q}$
by their respective equivalent names,
we get a $\below{\gamma'}\PP_\alpha$-name $\dot{q}'$ such that $p\forces \dot{q}=\dot{q}'$.

Again by the c.c.c., there exist $\varepsilon,\delta<\lambda$ such that $p\forces \dot{\sigma}_k\in \varepsilon^{<\delta}$ for every $k<\omega$.
Let $\gamma:=\max(\gamma_p,\gamma',\varepsilon,\delta)$ $<\lambda$.
Then $(p,\dot{q}')\in \below{\gamma}\PP_{\alpha+1}$, and it is
 equivalent to $(p,\dot{q})$, which finishes the proof.
\end{proof}

\begin{rema}\label{rema:why_cof_theta_large} In the more general situation described in Remark~\ref{rema:more_general}, i.e., if we work with the tree $\theta^{<\lambda}$ in place of $\lala$ 
(see also Remark~\ref{rema:epsilon_delta_neu}), 
we have to require that $\cf(\theta)\geq \lambda$. The reason is that  $|\sigma_k|$ can be arbitrarily large below $\lambda$: if $\cf(\theta)<\lambda$, it could happen that there does not exist an $\varepsilon<\theta$ which is needed in the end of the generalization of the above proof.
\end{rema}

\begin{lemm}\label{lemm:finding_gamma}
Let $G$ be $\PP_{\lambda}$-generic and $V[G]\models b\subseteq \omega$. Then there exists a $\gamma<\lambda$ and a
$\PP_{\lambda}$-name $\dot{b}$
for~$b$
which is hereditarily below $\gamma$.
\end{lemm}

\begin{proof}
For every condition $p\in  \PP_{\lambda}$, let $\gamma_p<\lambda$ be such that there exists
a condition in $\below{\gamma_p}\PP_\lambda$ which is forcing equivalent to~$p$
(which is possible by Lemma~\ref{lemm:finding_gamma_fresh}).
Let $\dot{b}'$ be a nice name for $b$, and let $\dot{b}$ be
a name where every condition~$p$ appearing in~$\dot{b}'$ is replaced by an equivalent condition in
$\below{\gamma_p}\PP_{\lambda}$.
 Since $b$ is a countable set and $\PP_{\lambda}$ has the c.c.c., the set $B$ of conditions which appeared in $\dot{b}'$ is countable. Let $\gamma:=\sup\{ \gamma_p\with p\in B\}<\lambda$;
then
$\dot{b}$ is a $\below{\gamma}\PP_{\lambda}$-name.
\end{proof}

We conclude
with
a technical lemma which will be  crucial later on:

\begin{lemm}\label{lemm:statt_oben}
Suppose $\tau \in \lala \setminus \gamma^{<\gamma}$.
Let $p, r \in \PP_{\lambda}$ such that
$p$ is a full condition which is almost hereditarily below $\gamma$ except for $\tau$,
and $r$ is hereditarily below $\gamma$,
and
$p$ and $r$ are compatible (in~$\PP_{\lambda}$).
Then there exists
a $p^* \in \PP_{\lambda}$ such that
\begin{enumerate}

\item $p^*$ is almost hereditarily below~$\gamma$ except for~$\tau$,

\item $p^* \leq p, r$, and

\item $p^*(0)(\tau) = p(0)(\tau)$.

\end{enumerate}
\end{lemm}

\begin{proof}
Without loss of generality we can assume that
$\dom(p(0))\supsetneqq \{\tau\}$.
Since $p$ is full and almost hereditarily below~$\gamma$,
by (the ``moreover part'' of) Lemma~\ref{lemm:hereditarily_is_complete}
\[
\OldResch{p} := (p(0)\restr \gamma^{<\gamma}, p(1), p(2),\dots)
\]
is a reduction of $p$ to $\below{\gamma}\PP_\lambda$.
We show that $\OldResch{p}  \comp_{\below{\gamma}\PP_{\lambda}} r$.
Assume not. Since $\below{\gamma}\PP_{\lambda}$ is a complete subforcing of $\PP_{\lambda}$, it follows that $\OldResch{p} \incomp_{\PP_{\lambda}} r$. But $p\leq \OldResch{p}$, so $p\incomp_{\PP_{\lambda}} r$, which is a contradiction to the assumption of the lemma.

Let $q^*\in \below{\gamma}\PP_{\lambda}$ be such that $q^*\leq \OldResch{p}, r$; without loss of generality, we can assume that $q^*(0)$ is full.
Since $q^*(0)\leq \OldResch{p}(0)=p(0)\restr \gamma^{<\gamma}$
and $p(0)\restr \gamma^{<\gamma}$ is a reduction of $p(0)$ by Lemma~\ref{lemm:C_complete_subforcing} (recall that $p(0)\doppelrestr \gamma^{<\gamma} = p(0)\restr \gamma^{<\gamma}$ because $\gaga$ is $0$-$\leftup$-closed), it follows that $q^*(0)$ is compatible with $p(0)$. Let $\bar{q}(0)$ be a full witness for that. So $\bar{q}(0)\leq p(0), r(0), q^*(0)$.

Let $p^*(0):=\bar{q}(0)\restr \gamma^{<\gamma} \cup \{(\tau, p(0)(\tau))\}$, and for $\alpha>0$, let $p^*(\alpha):=q^*(\alpha)$.

\begin{clai*}
$p^*(0)$ is a condition.
\end{clai*}

\begin{proof}
For $\sigma, \sigma'\in \dom(\bar{q}(0)\restr \gamma^{<\gamma} )$, it is clear that the requirements for being a condition are fulfilled, because $\bar{q}(0)$ is a condition.

Let $\sigma \unlhd \tau$ and $\sigma \in \dom(p^*(0))$.
Let $\sigma'\in \dom(p(0))\setminus \{\tau\}$.
Clearly, $s_\sigma^{p^*(0)}=s_\sigma^{\bar{q}(0)}$ and $s_\tau^{p^*(0)}=s_\tau^{p(0)}$. Since $\bar{q}(0)$ and $p$ are full, it follows that
$|s_\sigma^{\bar{q}(0)}|=|s_{\sigma'}^{\bar{q}(0)}|$ and hence
 $|s_\sigma^{p^*(0)}|=|s_{\sigma'}^{p^*(0)}|\geq |s_{\sigma'}^{p(0)}|=|s_\tau^{p(0)}|=|s_\tau^{p^*(0)}|$.

Let $\sigma\in \dom(f_\tau^{p^*(0)})$
and assume that $s_\tau^{p^*(0)}(m)=1$ for some $m\geq f_\tau^{p^*(0)}(\sigma)$.
We have to show that $s_\sigma^{p^*(0)}(m)=1$. Since $\bar{q}(0)$ extends $p(0)$, we have $s_\tau^{\bar{q}(0)}(m)=1$ and $\dom(f_\tau^{p^*(0)})\subseteq \dom(f_\tau^{\bar{q}(0)})$, and for $\sigma\in\dom(f_\tau^{p^*(0)})$ it holds that $f_\tau^{p^*(0)}(\sigma)\geq f_\tau^{\bar{q}(0)}(\sigma)$, so $\sigma\in \dom(f_\tau^{\bar{q}(0)})$ and $m\geq f_\tau^{\bar{q}(0)}(\sigma)$. Since $\bar{q}(0)$ is a condition, it follows that $s_\sigma^{p^*(0)}(m)=s_\sigma^{\bar{q}(0)}(m)=1$.

Let $\sigma\in \dom(h_\tau^{p^*(0)})$
and assume that $s_\tau^{p^*(0)}(m)=1$ for some $m\geq h_\tau^{p^*(0)}(\sigma)$.
We have to show that $s_\sigma^{p^*(0)}(m)=0$. Since $\bar{q}(0)$ extends $p(0)$, we have $s_\tau^{\bar{q}(0)}(m)=1$ and $\dom(h_\tau^{p^*(0)})\subseteq \dom(h_\tau^{\bar{q}(0)})$, and for $\sigma\in\dom(h_\tau^{p^*(0)})$ it holds that $h_\tau^{p^*(0)}(\sigma)\geq h_\tau^{\bar{q}(0)}(\sigma)$, so $\sigma\in \dom(h_\tau^{\bar{q}(0)})$ and $m\geq h_\tau^{\bar{q}(0)}(\sigma)$. Since $\bar{q}(0)$ is a condition, it follows that $s_\sigma^{p^*(0)}(m)=s_\sigma^{\bar{q}(0)}(m)=0$.
\end{proof}

Moreover, $p^*(0)\leq q^*(0)$, because $\bar{q}(0)\leq q^*(0)$ and $q^*$ hereditarily below $\gamma$ except for $\tau$. So $p^*$ is a condition. Clearly $p^*$ is almost hereditarily below $\gamma$ and $p^*(0)(\tau)=p(0)(\tau)$.

Since $r(0)$ is hereditarily below $\gamma$ and $p(0)$ is almost hereditarily below $\gamma$, and $\bar{q}(0)\leq r(0), p(0)$, it is clear that $p^*(0)$
extends $r(0)$ and $p(0)$.
So clearly $p^*\leq r, p$.
\end{proof}

\subsection{No refinement: branches are towers}\label{subsection:branches}

Now we are ready to prove that the generic matrix has no refinement. More precisely, we show that the sets along any branch in our tree have no pseudo-intersection, i.e., they form a tower.

\begin{lemm}
In $V[\PP_\lambda]$,
the sequence $\seqlangle a_{\sigma\restr \xi} \with \xi < \lambda \seqrangle$
is a tower
for each~$\sigma\in \lambda^\lambda$.
\end{lemm}

\begin{proof} Let $G_\lambda$ be generic for $\PP_\lambda$ and work in $V[G_{\lambda}]$.
Fix $\sigma\in \lambda^\lambda$.
By Corollary~\ref{coro:almost_included_disjoint}(1), $\seqlangle a_{\sigma\restr \xi} \with \xi < \lambda \seqrangle$ is $\subseteq^*$-decreasing.
Let us show that  $\seqlangle a_{\sigma\restr \xi} \with \xi < \lambda \seqrangle$ is actually a tower. Let $b\subseteq \omega$ be infinite,
and assume towards a contradiction that
$b\subseteq^* a_{\sigma\restr \xi}$ for every $\xi<\lambda$.

Apply Lemma~\ref{lemm:finding_gamma} to get
$\gamma < \lambda$ and a
$\mathbb{P}_{\lambda}$-name $\dot{b}$ for $b$ which is
hereditarily below~$\gamma$. Without loss of generality we can assume that $\gamma$ is a successor ordinal.
Fix~$\eta < \lambda$ minimal such that~$\sigma \restr \gamma \in V[G_\eta]$
(such an $\eta$ exists by Lemma~\ref{lemm:sigma_appears_somewhere}).
From now on, we work in $V[G_\eta]$,
and we consider\footnote{Here we use our modifications discussed in Remark~\ref{rema:eta_tail}.}
the tail forcing~$\PP_{\lambda}/G_\eta$.
The $\PP_\lambda$-name $\dot{b}$ can be understood as a  $\PP_{\lambda}/G_\eta$-name for $b$ which is hereditarily below~$\gamma$.

Since $b \subs^* a_{\sigma\restr \gamma}$ holds in~$V[G_{\lambda}]$, we can pick $n \in \om$ and $p \in \PP_{\lambda}/G_\eta$ such that
\[
p \forces \name{b} \setminus n \subs
a_{\sigma \restr \gamma}.
\]

From now on, whenever we say ``almost hereditarily below~$\gamma$'', we shall mean
``almost hereditarily below~$\gamma$
except for~$\sigma \restr \gamma$''.
Note that (the canonical name for) $a_{\sigma \restr \gamma}$ is
almost hereditarily below~$\gamma$;
also
$\name{b}$ is almost hereditarily below~$\gamma$
(because $\name{b}$ is hereditarily below~$\gamma$).

By Lemma~\ref{lemm:hereditarily_is_complete} and
Lemma~\ref{lemm:Resch_lemma},
we can fix~$p'$ which is
almost hereditarily below~$\gamma$
such that
\[
p' \forces \name{b} \setminus n \subs
a_{\sigma \restr \gamma}.
\]
Recall that $\eta$ is minimal with $\sigma\restr\gamma\in V[G_\eta]$, so $\QQ_\eta$ will assign a set $a_{\sigma\restr \gamma}$ to $\sigma\restr \gamma$.
Therefore we can assume
without loss of generality that  $\sigma \restr \gamma \in \dom(p'(\eta))$, and we can assume that $p'$ is a full\footnote{Here we use the modification of Definition~\ref{defi:full}, where $0$ is replaced by $\eta$, i.e., $p'(\eta)$ is full.} condition.

By Lemma~\ref{lemm:finde_r_und_m} there is
$r \in \PP_{\lambda}/G_\eta$ hereditarily below~$\gamma$ and
$m > n, |s^{p'(\eta)}_{\sigma \restr \gamma}|$
such that $r$ is compatible with~$p'$,
and $r \forces m \in \name{b}$.
Apply Lemma~\ref{lemm:statt_oben} to obtain
$p'' \leq p', r$
such that $p''$ is
almost hereditarily below~$\gamma$,
and moreover
\[
p''(\eta)(\sigma \restr \gamma) = p'(\eta)(\sigma \restr \gamma).
\]
It follows that
$p{''} \forces m \in \dot{b}$.
In particular $m>|s^{p''(\eta)}_{\sigma\restr \gamma}|$, thus we can strengthen $p''$ to a condition $q$ (only strengthening $p''(\eta)$) by extending $s^{p''(\eta)}_{\sigma\restr \gamma}$ to length $>m$ with $s^{q(\eta)}_{\sigma\restr \gamma}(m)=0$. Then $q\forces m\in \dot{b} \land m\notin a_{\sigma\restr \gamma}$,
which
is a contradiction to the fact that $p'$
forces $\dot{b}\setminus n\subseteq a_{\sigma\restr \gamma}$.
\end{proof}

\subsection{Levels are mad families}\label{subsection:rows}

Finally, we want to show
that the levels of the generic matrix form 
mad families.

\begin{lemm}
In $V[\PP_\lambda]$,
the family
$A_{\xi+1} = \{ a_{\sigma} \with |\sigma| = \xi +1 \}$ is mad
for each~$\xi < \lambda$.
\end{lemm}

\begin{proof}
 Let $G_\lambda$ be generic for $\PP_\lambda$ and work in $V[G_{\lambda}]$.
The main work
lies in the following claim, which guarantees ``local madness'' below branches.
We will prove it
after finishing the proof of the lemma.

\begin{clai}\label{clai:main_claim_for_mad}
Let $\rho \in \lala$, and let $b\subseteq \omega$ be infinite such that $b\cap a_{\rho \restr \zeta}$ is infinite for every successor $\zeta\leq |\rho|$.
Then there exists an $\newalpha<\lambda$ such that $b\cap a_{\rho \concat \newalpha}$ is infinite.
\end{clai}

Fix $\xi<\lambda$.
By Corollary~\ref{coro:almost_included_disjoint}(2), $A_{\xi+1}$ is an almost disjoint family.
Using the claim, we will show that $A_{\xi+1}$ is actually mad.
Let $b\subseteq \omega$ be infinite.
To find $\sigma\in \lambda^{\xi+1}$ such that $b\cap a_{\sigma}$ is infinite, we construct, by induction on $\zeta$, a
branch
$\seqlangle \rho_\zeta \with \zeta\leq \xi+1 \seqrangle$ with $|\rho_\zeta|=\zeta$ for each $\zeta$, and $\rho_{\zeta'} \unlhd \rho_\zeta$ for $\zeta'\leq\zeta$, such that $b\cap a_{\rho_\zeta}$ is infinite for every successor $\zeta\leq \xi+1$.

Let $\rho_0:= \seqlangle \seqrangle$.
Now assume we have constructed $\seqlangle  \rho_{\zeta'} \with \zeta'<\zeta  \seqrangle$.
If $\zeta$ is a limit, just let
 $\rho_\zeta:= \bigcup \{\rho_{\zeta'} \with \zeta'<\zeta\}$.
If $\zeta=\zeta'+1$ is a successor, $\rho_{\zeta'}$
fulfills the assumptions of the claim by induction. Let $i<\lambda$ be given by the claim, and let $\rho_\zeta:={\rho_{\zeta'}}\concat i$. Then $b\cap a_{\rho_\zeta}$ is infinite, as required. Finally, $\sigma:= \rho_{\xi+1}$ is as desired.
\end{proof}

\begin{proof}[Proof of Claim~\ref{clai:main_claim_for_mad}]
 Assume
towards contradiction that $b\cap a_{\rho\restr \zeta}$ is infinite for every successor $\zeta\leq|\rho|$, but $b\cap a_{\rho\concat i}$ is finite for every $i<\lambda$.

 Let $\eta$ be minimal with $\rho\in V[G_\eta]$
(such an $\eta$ exists by Lemma~\ref{lemm:sigma_appears_somewhere}).
Thus $a_{\rho\concat i}$ (for any $i$) is not defined in $V[G_\eta]$ but it will get defined in the next step of the forcing iteration.
 From now on, we work in $V[G_\eta]$,
and we consider\footnote{Here, again, we use our modifications discussed in Remark~\ref{rema:eta_tail}.}
the tail forcing~$\PP_{\lambda}/G_\eta$,
and
apply Lemma~\ref{lemm:finding_gamma} to get
 a $\mathbb{P}_{\lambda}/G_\eta$-name $\dot{b}$ for $b$ and $\gamma' < \lambda$ such that
$\dot{b}$ is hereditarily below~$\gamma'$.
Let $\gamma < \lambda$ be any ordinal strictly above
$|\rho| + 1$,
$\sup(\rang(\rho))$, and $\gamma'$.

Note that we can pick $n \in \om$ and $p \in \PP_{\lambda}/G_\eta$ such that

\begin{enumerate}

\item $p \forces \name{b} \cap a_{\rho \concat \gamma} \subs n$,

\item
$p \forces \name{b} \cap a_{\rho \concat i} \textrm{ is finite}$, for each $i < \gamma$, and

\item
$p \forces \name{b} \cap a_{\rho \restr \zeta} \textrm{ is infinite}$, for each successor $\zeta \leq |\rho|$.

\end{enumerate}

From now on, whenever we say ``almost hereditarily below~$\gamma$'', we shall mean
``almost hereditarily below~$\gamma$
except for~$\rho \concat \gamma$''.
Note that (the canonical name for) $a_{\rho \concat \gamma}$ is
almost hereditarily below~$\gamma$;
also
$\name{b}$ is almost hereditarily below~$\gamma$
(because $\name{b}$ is hereditarily below~$\gamma$),
and similarly $a_{\rho \concat i}$ is almost hereditarily below~$\gamma$ for each $i < \gamma$, and
$a_{\rho \restr \zeta}$ is almost hereditarily below~$\gamma$ for each successor $\zeta \leq |\rho|$.

By Lemma~\ref{lemm:hereditarily_is_complete} and
Lemma~\ref{lemm:Resch_lemma},
we can fix~$p'$ which is
almost hereditarily below~$\gamma$
such that items~(1), (2), and (3) above hold true for $p'$ in place of~$p$.
Without loss of generality, we can assume that
$\rho \concat \gamma \in \dom(p'(\eta))$, as well as that $p'$ is a full\footnote{Here, again, we use the modification of Definition~\ref{defi:full}, where $0$ is replaced by $\eta$, i.e., $p'(\eta)$ is full.} condition.

Define $R:= \dom(p'(\eta))\cap \{\rho\concat i \mid i<\gamma\}$,
and
$R':= \dom(f^{p'(\eta)}_{\rho\concat \gamma})$.
Let $\name{x}$ be a $\mathbb{P}_{\lambda}/G_\eta$-name
such that
\[
\forces \name{x} = \bigcap \limits_{\tau\in R'} (\dot{b}\cap a_\tau)\setminus \bigcup \limits_{\tau\in R} a_\tau;
\]
since the conditions which are hereditarily below~$\gamma$ form a complete subforcing of~$\PP_{\lambda}/G_\eta$ by
Lemma~\ref{lemm:hereditarily_is_complete}, and all names
which are used to define~$\name{x}$ are hereditarily below~$\gamma$,
we can assume that $\name{x}$ has been chosen to be hereditarily below~$\gamma$ as well.
Note that
since $R$ and $R'$ are finite,
$p'$ forces $\name{x}$ to be infinite.

 By Lemma~\ref{lemm:finde_r_und_m} there is
$r \in \PP_{\lambda}/G_\eta$ hereditarily below~$\gamma$ and
$m > n, |s^{p'(\eta)}_{\rho \concat \gamma}|$
such that $r$ is compatible with~$p'$,
and $r \forces m \in \name{x}$.
Apply Lemma~\ref{lemm:statt_oben} to obtain
$p'' \leq p', r$
such that $p''$ is
almost hereditarily below~$\gamma$,
and moreover
\[
p''(\eta)(\rho \concat \gamma) = p'(\eta)(\rho \concat \gamma).
\]
It follows that
$p{''} \forces m \in \dot{x}$, as well as
$p{''}\forces m\in a_\tau$ for $\tau\in R'$ and $p{''}\forces m\notin a_\tau$ for $\tau\in R$.

 Now extend $p{''}$ to a condition $q$ as follows. Let $q(\alpha)=p{''}(\alpha)$ for $\alpha>\eta$. For $\tau\in (R\cup R')\cap \dom(p{''}(\eta))$ extend $s^{q(\eta)}_\tau$ such that $|s^{q(\eta)}_\tau|>m$.
 It follows
 (for $\tau\in R'$)
 that $s^{q(\eta)}_\tau(m)=1$ for $\tau\in R'\cap \dom(p{''}(\eta))$, and $a_\tau(m)=1$ for $\tau\in R'\setminus \dom(p{''}(\eta))$ because $p{''}\forces m\in a_\tau$ for $\tau\in R'$; moreover, $s^{q(\eta)}_\tau(m)= 0$ for $\tau\in R$ because $p{''}\forces m\notin a_\tau$ for $\tau\in R$. Additionally fill $s^{q(\eta)}_{\rho\concat \gamma}$ with $0$ for entries smaller than $m$ and with $1$ at $m$. That is possible, because the $s^{q(\eta)}_\tau(m)$ are accordingly for $\tau \in R$ and $\tau\in R'\cap \dom(p{''}(\eta))$ respectively and $a_\tau(m)=1$ for $\tau\in R'\setminus \dom(p{''}(\eta))$.

It follows that $q\forces m\in \dot{x}\cap a_{\rho\concat \gamma}$, which is a contradiction to the fact that
$p'$
forces
$\dot{x}\cap a_{\rho\concat \gamma} \subseteq n$.
\end{proof}

This finishes the proof that the generic matrix is a distributivity matrix of height~$\lambda$. 
To
finish the proof of Main Theorem~\ref{maintheo:maintheo_general}, it remains to prove that $\bfrak$ (and hence~$\h$) is small in our final model; this is the subject of
Sections~\ref{sec:B_Canjar_filters} and~\ref{section:h_b_om_1}.

\section{$\B$-Canjar filters}\label{sec:B_Canjar_filters}

In this section, we will give the neccessary preliminaries about $\mathcal{B}$-Canjar filters and the preservation of unboundedness, which are needed in
Section~\ref{section:h_b_om_1}.

For $\mathfrak{F} \subs \mathcal{P}(\omega)$, let 
$\genlangle \mathfrak{F} \genrangle$ denote the filter generated by $\mathfrak{F}$ 
together with the Frech\'et filter.

\begin{defi}\label{defi:Mathias_with_filter}
Let $\mathcal{F}\subseteq \mathcal{P}(\omega)$ be a filter containing the Frech\'et filter. \emph{Mathias forcing with respect to $\mathcal{F}$} (denoted by $\Mathias{\mathcal{F}}$) is the set of pairs
$(s,A)$ with $s\in 2^{<\omega}$ and $A\in \mathcal{F}$, where the order is defined as follows: $(t,B)\leq (s,A)$ if
\begin{enumerate}
\item $t\unrhd s$, i.e., $t$ extends $s$
\item $B\subseteq A$
\item  for each $n\geq |s|$, if $t(n)=1$, then $n\in A$.
\end{enumerate}
\end{defi}

Note that
$\Mathias{\F}$ is $\sigma$-centered: for $s \in 2^{<\om}$, the set $\{ (s,A) \with A \in \F \}$ is clearly centered (i.e., finitely many conditions have a common lower bound).
Also note that Mathias forcing with respect to the Frech\'et filter is forcing equivalent to Cohen forcing~$\CC$.

A filter~$\F$ is Canjar if $\Mathias{\F}$ does not add a dominating real over the ground model (i.e., the ground model reals remain unbounded). We need the following generalization of Canjarness:

\begin{defi}
Let $\B \subs \om^\om$ be an unbounded family.
A filter $\F$ on~$\om$ is \emph{$\mathcal{B}$-Canjar} if
$\Mathias{\F}$
preserves the unboundedness of~$\B$
(i.e., $\B$ is still unbounded in the extension by~$\Mathias{\F}$).
\end{defi}

\subsection{A combinatorial characterization of $\B$-Canjarness}

Later, we
will
prove that certain filters are $\B$-Canjar; for that, we use the following combinatorial characterization of $\B$-Canjarness by Guzm\'{a}n-Hru\v{s}\'{a}k-Mart\'{i}nez \cite{Osvaldo_B_Canjar}.
This characterization generalizes a  characterization of Canjarness by Hru\v{s}\'{a}k-Minami~\cite{Minami}.

Let~$\F$ be a filter on~$\om$; recall that a set $X \subs [\om]^{<\om}$ is in
$\Osvaldo{\F}$ if and only if for each $A \in \F$ there is an~$s \in X$ with $s \subs A$.
Note that if $\mathcal{G}\subseteq \F$ are filters and $X\in \Osvaldo{\F}$, then $X\in\Osvaldo{\mathcal{G}}$.

Given $\bar{X} = \seqlangle X_n \with n \in \om \seqrangle$
(with
$X_n \subs [\om]^{<\om}$ for each $n \in \om$), and $f \in \om^\om$, let
\[
\bar{X}_f = \bigcup_{n \in \om} (X_n \cap \powerset(f(n))).
\]

\begin{theo}
\label{theo:OsvaldoHrusakMartinez}
Let $\B \subs \om^\om$ be an unbounded family.
A filter $\mathcal{F}$ on~$\om$ is $\mathcal{B}$-Canjar if and only if
the following holds:
for each sequence $\bar{X} = \seqlangle X_n \with n \in \om \seqrangle \subs \Osvaldo{\F}$, there exists an $f \in \B$ such that
$\bar{X}_f \in \Osvaldo{\F}$.
\end{theo}

\begin{proof}
See \cite[Proposition~1]{Osvaldo_B_Canjar}.
\end{proof}

It is well-known that Cohen forcing~$\CC$
preserves 
the unboundedness of every unbounded family 
(in fact, $\CC$ is even almost bounding).
As mentioned above,
Mathias forcing with respect to the Frech\'et filter
is forcing equivalent to~$\CC$, and hence
the Frech\'et filter is
$\B$-Canjar for every unbounded family~$\B$.
To illustrate the characterization of $\B$-Canjarness from   Theorem~\ref{theo:OsvaldoHrusakMartinez}, we also want to provide the following easy combinatorial proof of this fact:

\begin{lemm}\label{lemm:Frechet_is_Canjar}
Let $\B$ be an unbounded family.
Then
the Frech\'et filter is $\B$-Canjar.
\end{lemm}

\begin{proof}
Let $\F$ be the Frech\'et filter. To show that
$\F$ is $\B$-Canjar, we use
Theorem~\ref{theo:OsvaldoHrusakMartinez}.
So let $\bar{X} = \seqlangle X_n \with n \in \om \seqrangle \subs \Osvaldo{\F}$. Note that a set
$X \subs [\om]^{<\om}$ is in
$\Osvaldo{\F}$ if and only if for each
$n \in \om$ there is an~$s \in X$ with $\min(s) \geq n$.
For each $n \in \om$, pick $s_n \in X_n$ such that $\min(s_n) \geq n$, and let $g \in \om^\om$ such that $g(n) > \max(s_n)$ for each $n \in \om$. Since $\B$ is unbounded, we can pick $f \in \B$ such that $f(n) > g(n)$ for infinitely many~$n$. It is easy to check that $s_n \in \bar{X}_f$ for infinitely many $n$, and this implies that
$\bar{X}_f \in \Osvaldo{\F}$, as desired.
\end{proof}

The following observation will be crucial later on:

\begin{lemm}\label{lemm:Canjar_countable_ext}
Let $\B\subseteq \omega^\omega$ be an unbounded family, $\F$ a $\B$-Canjar filter extending the Frech\'et filter and $\{ a_n \with n<\omega\}$ such that $\F\cup \{ a_n \with n<\omega\}$ is a filter base.
Then 
$\genlangle \F\cup \{ a_n \with n<\omega\} \genrangle$ is $\B$-Canjar.
\end{lemm}
\begin{proof}
Let $\bar{X} = \seqlangle X_n \with n \in \om \seqrangle \subs \Osvaldo{\genlangle \F\cup \{a_n\with n<\omega\} \genrangle }$.
Let $$Y_n:= \{s\in X_n \with s\subseteq \cap_{k<n} a_k \}$$
and $\bar{Y}:=\seqlangle Y_n \with n\in \omega\seqrangle$.
It is easy to see that $Y_n\in \Osvaldo{\F}$ for each $n$.
By the assumption and Theorem~\ref{theo:OsvaldoHrusakMartinez} there exists $f\in\B$ such that $\bar{Y}_f\in \Osvaldo{\F}$.

To show that $\bar{Y}_f\in \Osvaldo{\genlangle \F\cup \{a_n\with n<\omega\} \genrangle }$ let $B\in \genlangle \F\cup \{a_n\with n<\omega\} \genrangle$, i.e., there exists $A\in \F$ and $n\in \omega$ with $B\supseteq A\cap\bigcap_{k<n}a_k$. Since $\F$ contains the Frech\'et filter and $\bar{Y}_f \in \Osvaldo{\F}$, there exist
infinitely many $s\in \bar{Y}_f$ with $s\subs A$. So there exists $m\geq n$ and $s\in Y_m\cap \bar{Y}_f$ with $s\subs A$; note that $s\in Y_m$ implies $s\subseteq \bigcap_{k<n}a_k$, so $s\subseteq B$, as desired.

Clearly $\bar{Y}_f\subseteq \bar{X}_f$, so $\bar{X}_f\in \Osvaldo{\genlangle \F\cup \{a_n\with n<\omega\} \genrangle }$.
\end{proof}

We also get the following:

\begin{lemm}\label{lemm:countable_Canjar}
Let $\B$ be an unbounded family.
Then
every countably generated filter is $\B$-Canjar.
\end{lemm}

\begin{proof}
This follows immediately from Lemma~\ref{lemm:Frechet_is_Canjar} and Lemma~\ref{lemm:Canjar_countable_ext}.
\end{proof}

\subsection{Preservation of unboundedness at limits}

We will 
use the following theorem by 
Judah-Shelah~\cite{Judah} about preservation of unboundedness in finite support iterations. 
In fact, \cite[Theorem~2.2]{Judah} is
a much more general version than the 
theorem 
presented here.

\begin{theo}
\label{theo:preservation_unboundedness_JudahShelah}
Suppose $\{ \PP_\alpha, \name{\QQ}_\alpha \with \alpha <\delta \}$ is a finite support iteration of c.c.c.\ partial orders of limit length~$\delta$, and $\B \subs \om^\om$ is unbounded; also suppose that 
$\B$ is countably directed, i.e., 
it 
satisfies
\begin{equation}\label{eq:countably_bounded}
\forall \mathcal{A} \subs \B \; (|\mathcal{A}| = \aleph_0 \rightarrow \exists f \in \B \; \forall g \in \mathcal{A} \; g \leq^* f);
\end{equation}
moreover, suppose that
\[
\forall \alpha < \delta \; \forces_{\PP_\alpha} \textrm{``$\B$ is an unbounded family''}.
\]
Then $\forces_{\PP_\delta}$ ``$\B$ is an unbounded family''.
\end{theo}

\begin{proof}
See~\cite[Theorem~3.5.2]{Vera_thesis}.
\end{proof}

\subsection{Preservation of $\B$-Canjarness and finite sums of filters}\label{subsec:method_to_preserve}

The notion of $\B$-Canjarness of a filter 
is 
not absolute in general:

\begin{exam}[from~\cite{Osvaldo_personal}]\label{exam:not_absolute}
Let 
$\B$ be the ground model reals and $\mathcal{U}$ be a $\B$-Canjar ultrafilter. Let $\PP$ be Grigorieff forcing with respect to~$\mathcal{U}$, which forces that $\mathcal{U}$ cannot be extended to a P-point. It is
well-known
that $\PP$ preserves the unboundedness of $\B$, and it can be shown that $\mathcal{U}$ is not a $P^+$-filter in $V[\PP]$;
since any Canjar filter is a $P^+$-filter, it follows that $\mathcal{U}$ is no longer $\B$-Canjar.

Note that Grigorieff forcing is proper, but not c.c.c.;
however, Grigorieff forcing can be decomposed into a $\sigma$-closed and a c.c.c.\ forcing (see~\cite{Repicky}).
Since a $\sigma$-closed forcing does not destroy the $\B$-Canjarness of a filter, the above example also yields
an example 
of a c.c.c.\ forcing destroying the $\B$-Canjarness of a filter. 
\end{exam}

We will now provide
a method
how to guarantee that
the $\B$-Canjarness of a filter is
not destroyed by Mathias forcings with respect to certain other filters.
As a tool, we introduce
finite
sums of filters and consider
Mathias forcings with respect to these sums.

\begin{lemm}\label{lemm:basic_Canjar_equivalence}
Let 
$\F$ be a
filter,
$\B \subs \om^\om$,
and $\PP$ be a forcing notion.
Then
the following are equivalent:

\begin{enumerate}

\item $\PP$ forces that\footnotemark{} $\F$ is $\B$-Canjar.

\footnotetext{To be more precise, one should write
$\genlangle\check{\F}\genrangle$
instead of~$\F$.}

\item $\Mathias{\F} \times \PP$ forces that $\B$ is unbounded.

\end{enumerate}
\end{lemm}

Even though we will apply the lemma only in case 
$\B$ is unbounded 
and 
$\F$ is $\B$-Canjar in the ground model, this is not necessary for the proof. If one of these assumptions fails, 
both (1) and (2) are false.

\begin{proof}[Proof of Lemma~\ref{lemm:basic_Canjar_equivalence}]
Let $\QQ := \Mathias{\F}$.
Note that
(1) holds if and only if
$\PP$ forces
$$\Mathias{\genlangle\check{\F}\genrangle} \forces  ``\B \textrm{ unbounded''}.$$
Further note that $\PP$ forces that $\check{\QQ}$ is (dense in, and hence) forcing equivalent to $\Mathias{\genlangle\check{\F}\genrangle}$.
So, (1) holds if and only if $\PP * \check{\QQ}$ forces that $\B$ is unbounded, which is the same as~(2) (since
$\PP * \check{\QQ}$ is equivalent to $\PP \times \QQ = \QQ \times \PP$).
\end{proof}

\begin{defi}\label{defi:filter_sum}
For two sets $A, B\subseteq \omega$, let
$A \fsum B:= \{ 2n \with n\in A\} \cup \{2m+1\with m\in B\}$.
For two filters $\F_0$ and $\F_1$, let $\F_0\fsum \F_1:=\{ A\fsum B \with A\in \F_0, B\in \F_1\} $.
More generally, inductively define $\Fsum_{k<m+1} \F_k := \left(\Fsum_{k<m} \F_k\right) \fsum \F_{m}$.
\end{defi}

Note that $\F_0\fsum \F_1$ is a filter if $\F_0$ and $\F_1$ are filters, and hence also the finite sum of filters is a filter.
The order of the sum is not important: more presicely, the filter $\Fsum_{k<m}\F_k$ is isomorphic (based on a bijection on $\omega$) to all reorderings of this sum. For example $(\F_0\fsum \F_1)\fsum \F_2$ is isomorphic to $(\F_2\fsum \F_0)\fsum \F_1$.
This implies that the $\B$-Canjarness of a finite sum of filters does not depend on the order of the sum.

\begin{lemm}\label{lemm:Mathias_product}
Let $\F_0$ and $\F_1$ be two filters.
Then $\Mathias{\F_0}\times\Mathias{\F_1}$ is
forcing
equivalent to~$\Mathias{\F_0 \fsum \F_1}$.
\end{lemm}
\begin{proof}
Let $D_\times \subs \Mathias{\F_0}\times\Mathias{\F_1}$ be the
set of all
$((s_0,A_0),(s_1,A_1)) \in \Mathias{\F_0}\times\Mathias{\F_1}$ with $|s_0| = |s_1|$, and let
$D_{\oplus} \subs \Mathias{\F_0 \fsum \F_1}$
be the set of all $(s,A) \in \Mathias{\F_0 \fsum \F_1}$ with $|s|$
being an even number. Note that
$D_\times$ is a dense subforcing of  $\Mathias{\F_0}\times\Mathias{\F_1}$, and $D_{\oplus}$ is a dense subforcing of $\Mathias{\F_0 \fsum \F_1}$.

For $s_0, s_1 \in 2^{<\om}$ with $L:=|s_0| = |s_1|$, let
$s_0 \fsum s_1 \in 2^{<\om}$ be such that $|s_0 \fsum s_1| =
2 L$
and satisfies
$(s_0 \fsum s_1)(2n) = s_0(n)$ and
$(s_0 \fsum s_1)(2n+1) = s_1(n)$.

Define $\iota{:}\ D_\times \rightarrow D_{\oplus}$ as follows:
\[
((s_0,A_0),(s_1,A_1)) \mapsto (s_0 \fsum s_1, A_0 \fsum A_1).
\]
It is easy to see that $\iota$ is an isomorphism between the forcings $D_\times$ and $D_{\oplus}$.
Consequently,
$\Mathias{\F_0}\times\Mathias{\F_1}$
and $\Mathias{\F_0 \fsum \F_1}$
are forcing equivalent.
\end{proof}

The following lemma will be the main ingredient of the
``successor step'' of the induction (for old filters) in Lemma~\ref{lemm:main_Canjar_sum_induction}:

\begin{lemm}\label{lemm:sum_preserves_Canjar}
If $\F_0 \fsum \F_1$ is $\B$-Canjar, then
$\Mathias{\F_1}$ forces that $\F_0$ is $\B$-Canjar.
\end{lemm}
\begin{proof}
By assumption and Lemma~\ref{lemm:Mathias_product},
$\Mathias{\F_0}\times\Mathias{\F_1}$ forces that
$\B$ is unbounded;
apply Lemma~\ref{lemm:basic_Canjar_equivalence} to finish the proof.
\end{proof}

The following lemma will be the main ingredient of the
``limit step'' of the induction (for old filters) in Lemma~\ref{lemm:main_Canjar_sum_induction}:

\begin{lemm}\label{lemm:limits_preserve_B_Canjar}
Let $\B$ be a 
countably directed 
family 
(see~\eqref{eq:countably_bounded} of Theorem~\ref{theo:preservation_unboundedness_JudahShelah}),
let $\alpha$ be a limit, and let 
$\{ \PP_\beta,\name{\QQ}_\beta \with \beta<\alpha\}$
be a finite support iteration.
Suppose that $\PP_\beta$ forces that $\F$ is $\B$-Canjar
for every $\beta<\alpha$.
Then
$\PP_\alpha$ forces that $\F$ is $\B$-Canjar.
\end{lemm}

\begin{proof}
By assumption and Lemma~\ref{lemm:basic_Canjar_equivalence},
$\Mathias{\F} \times \PP_\beta$ forces that $\B$ is unbounded for every $\beta < \alpha$.
Observe that $\Mathias{\F} \times \PP_\alpha$ is the direct limit of the sequence
$\seqlangle \Mathias{\F} \times \PP_\beta \with \beta < \alpha \seqrangle$
(and $\Mathias{\F} \times \PP_\beta$ is complete in $\Mathias{\F} \times \PP_\alpha$),
so
it can be written as the limit of a finite support iteration,
therefore, by Theorem~\ref{theo:preservation_unboundedness_JudahShelah},
also
$\Mathias{\F} \times \PP_\alpha$ forces that
$\B$ is unbounded.
We obtain the conclusion by again applying Lemma~\ref{lemm:basic_Canjar_equivalence}.
\end{proof}

\begin{lemm}\label{lemm:Canjar_oplus_countably}
Let $\F_0$ be $\B$-Canjar and $\F_1$ be countably generated.
Then $\F_0 \fsum \F_1$ is $\B$-Canjar.
\end{lemm}

Using the fact that sums can be reordered (see the remark after Definition~\ref{defi:filter_sum}), we obtain the following stronger statement:
Let $\F_0,\dots,\F_{m-1}$ be filters such that (some of them are countably generated and) the sum of the filters which are not countably generated is $\B$-Canjar; then $\Fsum_{k<m}\F_k$ is $\B$-Canjar.

\begin{proof}[Proof of Lemma~\ref{lemm:Canjar_oplus_countably}]
\begin{sloppypar}
We have to show that $\Mathias{\F_0 \fsum \F_1}$
forces that $\B$ is
unbounded.
By Lemma~\ref{lemm:Mathias_product},
$\Mathias{\F_0 \fsum \F_1}$
is forcing equivalent to
$\Mathias{\F_0}\times\Mathias{\F_1}$.
\end{sloppypar}

Since $\F_0$ is $\B$-Canjar by assumption,
$\B$ is unbounded in the extension by $\Mathias{\F_0}$.
Since $\F_1$ is countably generated,
the same holds in the extension by~$\Mathias{\F_0}$: more precisely, the filter generated by~$\F_1$ is countably generated.
Therefore, in the extension by $\Mathias{\F_0}$,
(the filter generated by) $\F_1$ is
$\B$-Canjar by Lemma~\ref{lemm:countable_Canjar}. So,
by Lemma~\ref{lemm:basic_Canjar_equivalence},
$\Mathias{\F_0}\times\Mathias{\F_1}$ forces that $\B$ is unbounded, as desired.
\end{proof}

\section{Preserving unboundedness: $\h = \bfrak = \om_1$}\label{section:h_b_om_1}

In this section,
we will finish the proof of Main Theorem~\ref{maintheo:maintheo_general} 
by showing 
that $\mathfrak{b}$ is small (i.e., $\bfrak = \omega_1$) in 
the 
final model~$W$. 
 Recall
 the following well-known ZFC inequalities (see~\eqref{eq:ZFC_inequ} in 
Section~\ref{sec:preliminaries}):
$$\omega_1\leq \h\leq \bfrak$$

 So, if we have $\mathfrak{b}=\omega_1$, 
 it follows 
 that $\mathfrak{h}=\omega_1$ 
(in particular, 
there exists
a distributivity matrix of height~$\omega_1$).
We have 
shown that 
there is 
a distributivity matrix of height~$\lambda > \om_1$ in~$W$; 
so 
there exists a distributivity matrix of regular height larger than~$\h$. 

In Section~\ref{subsec:layering_and_F_beta},
we will show that our iteration $\PP_\lambda$ can be represented as a finer iteration whose iterands are Mathias forcings with respect to filters. In Section~\ref{subsec:The_filters_are_B_Canjar}, we show that the filters which are used are $\mathcal{B}$-Canjar
(i.e.,
the corresponding Mathias forcings
preserve the unboundedness of~$\mathcal{B}$), where $\B$ is the set of reals
of $V_0$.
A similar
(but less involved)
argument shows that
Hechler's original forcings~\cite{Hechler}
to add a tower or to add a mad family
can be represented as an iteration of
Mathias forcings
with respect to $\B$-Canjar filters as well (see~\cite{our_TOW_MAD}).

\subsection{Finer iteration via filtered Mathias forcings}\label{subsec:layering_and_F_beta}

As described in Section~\ref{subsec:forcing_definition},
$\{ \PP_\alpha, \name{\mtx}_\alpha \with \alpha < \lambda \}$
is our main finite support iteration
which we force with over~$V$. Its limit~$\PP_\lambda$ adds a distributivity matrix of height~$\lambda$.
We will now represent our iteration
as a ``finer'' iteration:
we write each iterand~$\QQ_\alpha$
as a finite support iteration
of Mathias forcings with respect to certain filters.
Fix~$\alpha<\lambda$.

As a preparation, we
introduce
a ``nice'' enumeration of~$T_\alpha$ (recall that $\sigma\in T_\alpha$ if and only if $a_\sigma$ is added by $\QQ_\alpha$).
We go
through the nodes in
$T_\alpha$ level by level, and ``blockwise''. A \emph{block} is a set of nodes $\{ \rho \concat i \with i < \lambda \}$ for some $\rho \in \lala$.
More precisely,
let
$\{ \sigma_\alpha^\newi \with \newi < \Lambda_\alpha \}$ be an enumeration of
$T_\alpha$
(note that $|T_\alpha| = \cc$ and hence $\Lambda_\alpha$ is an ordinal with
$\cc<\Lambda_\alpha<\cc^+$)
such that

\begin{enumerate}

\item (``level by level'') $|\sigma_\alpha^\newj| < |\sigma_\alpha^\newi| \rightarrow \newj < \newi$,

\item (``blockwise'') for each
$\rho \in \lala$ with $\{ \rho \concat i \with i < \lambda \} \subs T_\alpha$, there is $\newi < \Lambda_\alpha$ such that
$$
\rho \concat i = \sigma_\alpha^{\newi + i}
\textrm{ for each } i < \lambda.
$$

\end{enumerate}

Recall that $\FKWW_\alpha^C$ denotes $\{ p \in \FKWW_\alpha \with \dom(p) \subs C \}$ (for $C \subs \lala$).
For any $\beta \leq \Lambda_\alpha$,
let
\[
\FKWW_\alpha^{<\beta} := \FKWW_\alpha^{\{ \sigma_\alpha^\newi \with \newi < \beta \}},
\]
and for $\beta < \Lambda_\alpha$,
\[
\FKWW_\alpha^{\leq \beta} := \FKWW_\alpha^{\{ \sigma_\alpha^\newi \with \newi \leq \beta \}}.
\]

Note that
$\QQ_\alpha^{<\Lambda_\alpha} = \QQ_\alpha$, and
that $\{ \sigma_\alpha^\newi \with \newi < \beta \}$ is $\alpha$-$\leftup$-closed for each $\beta \leq \Lambda_\alpha$ (due to~(1) and (2) above).
Therefore,
by Lemma~\ref{lemm:C_complete_subforcing}, 
$\FKWW_\alpha^{<\beta}$ is a complete subforcing of~$\QQ_\alpha$. 
By Lemma~\ref{lemm:compl_compl_is_compl},
$\QQ_\alpha^{<\beta}$ is a complete subforcing of $\QQ_\alpha^{\leq\beta}$,
so we can form the quotient~$\QQ_\alpha^{\leq \beta}/\QQ_\alpha^{<\beta}$.
Moreover,
because conditions in~$\QQ_\alpha$ have finite domain,
\[
\FKWW_\alpha^{<\beta} = \bigcup_{\delta < \beta} \FKWW_\alpha^{<\delta}
\]
for each limit
ordinal~$\beta \leq \Lambda_\alpha$;
in other words, $\FKWW_\alpha^{<\beta}$ is the direct limit of the forcings $\FKWW_\alpha^{<\delta}$ for $\delta < \beta$.
So $\FKWW_\alpha$ is forcing equivalent to the finite support iteration of the
quotients
$\QQ_\alpha^{\leq \beta}/\QQ_\alpha^{<\beta}$
for
$\beta < \Lambda_\alpha$.

Recall that
$\Mathias{{\F}}$ denotes Mathias forcing with respect to the filter~${\F}$ (see Definition~\ref{defi:Mathias_with_filter}).
We are now going to
show that $\QQ_\alpha^{\leq \beta}/\QQ_\alpha^{<\beta}$ is forcing equivalent to
$\Mathias{{\F}_\alpha^{\beta}}$ for a filter~${\F}_\alpha^{\beta}$.
Work in an extension by~$\PP_\alpha * \QQ_\alpha^{<\beta}$,
and note that, for each $\tau\in T_\eta$ with $\eta<\alpha$, a set $a_\tau$ has been added by $\PP_\alpha$, and for
each $\nu < \beta$, a set $a_{\sigma_\alpha^\nu}$ has been added by $\PP_\alpha * \QQ_\alpha^{<\beta}$.
These sets\footnotemark{} are used to define $\F_\alpha^\beta$ as follows.
Let $\rho \in \lala$ and $i < \lambda$
be such that $\sigma_\alpha^\beta = \rho \concat i$,
and let
 $$ \FBa_\alpha^\beta:=  \{ a_{\rho \restr (\xi+1)} \with \xi+1 \leq |\rho| \} \cup
\{ \om \setminus a_{\rho \concat j} \with j < i \},$$
i.e., $\FBa_\alpha^\beta$ is the collection of all
sets assigned to the nodes above~$\sigma_\alpha^\beta$
and the complements of the sets assigned to the nodes to the left of~$\sigma_\alpha^\beta$ within the  same block.
Note that $\FBa_\alpha^\beta$ is a filter base, i.e., any intersection of finitely many elements is infinite: indeed, for finite $I\subseteq i$ and $\xi+1\leq|\rho|$, let $j^*\in \lambda\setminus I$; then $a_{\rho\concat j^*} \subs^* a_{\rho\restr (\xi+1)}\cap\bigcap_{j\in I}(\omega\setminus a_{\rho\concat j})$.
Then let
\[ \F^{\beta}_\alpha
:=
\genlangle \FBa_\alpha^\beta \genrangle,
\]
i.e., $\F^{\beta}_\alpha$ is the filter
generated by taking finite intersections
of sets from $\FBa_\alpha^\beta$
and the Frech\'{e}t filter and taking the upwards closure.

\footnotetext{It is possible
(see the base step $\beta^* = 0$ of the proof of
Lemma~\ref{lemm:main_Canjar_sum_induction}(3))
that only sets $a_\tau$ with $\tau\in T_\eta$ for some $\eta<\alpha$ are used. This is the case if $\rho$ is pre-$T_\alpha$-minimal and $i=0$.
}

The quotient~$\QQ_\alpha^{\leq \beta}/\QQ_\alpha^{<\beta}$
adds the set~$a_\sigma$
where $\sigma=\sigma_\alpha^\beta$.
The following lemma
will provide a dense embedding from
$\QQ_\alpha^{\leq \beta}/\QQ_\alpha^{<\beta}$
to~$\Mathias{\F_\alpha^\beta}$
which preserves (the finite appoximations of) the generic
real~$a_\sigma$.
Therefore,
$a_\sigma$ is also the generic real
for~$\Mathias{{\F_\alpha^\beta}}$.
Recall that
the generic real
for $\Mathias{{\F}}$ is a pseudo-intersection of~$\F$, and the definition
of~$\F_\alpha^\beta$ ensures that
a pseudo-intersection of it
is almost contained in
$a_{\rho \restr (\xi+1)}$ whenever  $\xi+1 \leq |\rho|$ and almost disjoint from
$a_{\rho \concat j}$ for each $j < i$,
as it is the case for the real $a_\sigma$.

\begin{lemm}\label{lemm:quotient_is_mathias}
$\FKWW_\alpha^{\leq \beta}/\FKWW_\alpha^{<\beta}$ is densely embeddable into $\Mathias{\F_\alpha^{\beta}}$.
\end{lemm}
\begin{proof}
For simplicity of notation, let $\sigma:={\sigma_\alpha^\beta}$ for the rest of this proof.
Let $G$ be a generic filter for
$\FKWW_\alpha^{<\beta}$. We work in the extension by $G$, so  $$\FKWW_\alpha^{\leq \beta}/\FKWW_\alpha^{<\beta}= \{ p\in \FKWW_\alpha^{\leq \beta} \with \forall q\in G (p \text{ is compatible with } q)\}.$$
Let us define an embedding $\iota{:}\ \FKWW_\alpha^{\leq \beta}/\FKWW_\alpha^{<\beta}\rightarrow \Mathias{\F_\alpha^\beta}$ as follows:
for $p \in \FKWW_\alpha^{\leq \beta}/\FKWW_\alpha^{<\beta}$, let
$p(\sigma) = (s_\sigma,f_\sigma,h_\sigma)$, and
let
$\iota(p)= (s_{\sigma},A)$, where
$$A =
\bigcap\limits_{\tau\in\dom(f_{\sigma})}( a_\tau\cup f_{\sigma}(\tau))\cap \bigcap\limits_{\rho\in\dom(h_{\sigma})}( (\om\setminus a_\rho) \cup h_{\sigma}(\rho))\setminus |s_{\sigma}|.$$
To see that it is a dense embedding, we have to check the following conditions:
\begin{enumerate}
\item (Density) For every condition $(s,A)\in \Mathias{\F_\alpha^{\beta}}$, there exists a condition $p$ such that $\iota(p)\leq (s,A)$.
\item (Incompatibility preserving) If $p$ and $p'$ are incompatible, then so are $\iota(p)$ and $\iota(p')$.
\item (Order preserving) If $p'\leq p$, then $\iota(p')\leq \iota(p)$.
\end{enumerate}

To show (1), let $(s,A)\in \Mathias{\F_\alpha^\beta}$. Since $A\in \F_\alpha^\beta$, there exist
finite sets $\{\rho_i\with i<m\}$, $\{ \tau_j \with j<l\}$ and $N\in \om$ such that
$\bigcap_{j<l}a_{\tau_j}\cap \bigcap_{i<m} (\om\setminus a_{\rho_i})\setminus N \subseteq A$.
 Extend~$s$ with 0's to $s_\sigma$ such that $|s_\sigma|=\max(|s|,N)$, and
let $\dom(h_{\sigma}):=\{\rho_i\with i<m\}$ and $h_\sigma(\rho_i):=|s_\sigma|$ for every $i$, and $\dom(f_\sigma):=\{ \tau_j \with j<l\}$ and $f_\sigma(\tau_j):=|s_\sigma|$ for every $j$.
Let $p:=\{(\sigma,(s_\sigma,f_\sigma,h_\sigma))\}\cup \{ (\tau,(\seqlangle \seqrangle, \emptyset, \emptyset))\with \tau\in (\dom(f_\sigma)\cap T_\alpha)\cup \dom(h_\sigma)\}$.

To see that $p$ is in the quotient, let $q\in G$ 
be arbitrary; it is easy to check that
$q\cup \{(\tau,(s_\tau, f_\tau, h_\tau))\with \tau\in \dom(p)\setminus \dom(q)\} \leq p, q$.

By definition, $\iota(p)= (s_\sigma,A')$, where
$$A'=\bigcap\limits_{\tau\in\dom(f_{\sigma})}( a_\tau\cup f_{\sigma}(\tau))\cap \bigcap\limits_{\rho\in\dom(h_{\sigma})}( (\om\setminus a_\rho) \cup h_{\sigma}(\rho))\setminus |s_{\sigma}|.$$
It follows that $$A'\stackrel{(*)}{=}\bigcap\limits_{\tau\in\dom(f_{\sigma})}a_\tau\cap \bigcap\limits_{\rho\in\dom(h_{\sigma})}(\om\setminus a_\rho)\setminus |s_{\sigma}|\subseteq
\bigcap\limits_{j<l}a_{\tau_j}\cap \bigcap\limits_{i<m} (\om\setminus a_{\rho_i})\setminus N  \subseteq A$$ (where (*) holds because $|s_\sigma|\geq f_\sigma(\tau), h_\sigma(\rho)$ for every $\tau, \rho$ in
the respective domains). Therefore $s_\sigma \unrhd s$, $A'\subseteq A$, and $s_\sigma(n)=0$ for all $n\geq |s|$.
So $\iota(p)=(s_\sigma,A')\leq (s,A)$.

We prove (2) by showing the contrapositive. Assume
$\iota(p)$ and $\iota(p')$ are compatible.
Define $q$ as follows. Let $\dom(q):= \dom(p)\cup \dom(p')$.  For every $\tau\in \dom(q)$, let $s_\tau^q:= s_\tau^p\cup s_\tau^{p'}$, $\dom(f_\tau^q):= \dom(f_\tau^p)\cup\dom(f_\tau^{p'})$ and for $\rho\in \dom(f_\tau^q)$ let $f_\tau^q(\rho)=\min(f_\tau^p(\rho),f_\tau^{p'}(\rho))$, and the same for $h$: $\dom(h_\tau^q):= \dom(h_\tau^p)\cup\dom(h_\tau^{p'})$ and for $\rho\in \dom(h_\tau^q)$ let $h_\tau^q(\rho)=\min(h_\tau^p(\rho),h_\tau^{p'}(\rho))$. It is easy to check
that $q$ is a condition in the quotient and $q\leq p,p'$.

To show (3), let $p'\leq p$. So, by definition, 
$s_\sigma^{p'} \unrhd s_\sigma^p$,
and $\dom(h_\sigma^{p'})\supseteq \dom(h_\sigma^{p})$ and $\dom(f_\sigma^{p'})\supseteq \dom(f_\sigma^{p})$, and $f_\sigma^{p'}(\tau)\leq f_\sigma^{p}(\tau)$ for $\tau\in \dom(f_\sigma^{p})$ and $h_\sigma^{p'}(\rho)\leq h_\sigma^{p}(\rho)$ for $\rho\in \dom(h_\sigma^{p})$; so
\begin{align*}
A' := & \bigcap\limits_{\tau\in\dom(f_{\sigma}^{p'})}( a_\tau\cup f_{\sigma}^{p'}(\tau))\cap  \bigcap\limits_{\rho\in\dom(h_{\sigma}^{p'})}( (\om\setminus a_\rho) \cup h_{\sigma}^{p'}(\rho))\setminus |s_{\sigma}^{p'}| \\
\subseteq &  \bigcap\limits_{\tau\in\dom(f_{\sigma}^{p})}( a_\tau\cup f_{\sigma}^{p}(\tau))\cap  \bigcap\limits_{\rho\in\dom(h_{\sigma}^{p})}( (\om\setminus a_\rho) \cup h_{\sigma}^{p}(\rho))\setminus |s_{\sigma}^{p}|=:A.
\end{align*}
By definition, $\iota(p) = (s_\sigma^p,A)$ and $\iota({p'}) = (s_\sigma^{p'},A')$.
To show that $(s_\sigma^{p'},A')\leq (s_\sigma^p,A)$, it remains to show that
for
  $n\geq |s_\sigma^p|$ 
with
 $s_\sigma^{p'}(n)=1$, we have 
  $n\in A$. 
 First fix $\rho\in\dom(h_\sigma^p)$ and show that $n\in (\om\setminus a_\rho) \cup h_\sigma^p(\rho)$. If $n<h_\sigma^p(\rho)$, this is clear. If $n\geq h_\sigma^p(\rho)$, we know that $s_\sigma^{p'}$ respects $h_\sigma^p$, and so $n\in \om\setminus a_\rho$. So in both cases, $n\in (\om\setminus a_\rho) \cup h_\sigma^p(\rho)$.
 Now fix $\tau\in\dom(f_\sigma^p)$ and show that $n\in a_\tau \cup f_\sigma^p(\tau)$. This is the same argument as for~$h$. If $n<f_\sigma^p(\tau)$, this is clear. If $n\geq f_\sigma^p(\tau)$, we know that $s_\sigma^{p'}$ respects $f_\sigma^p$, and so $n\in a_\tau$. So in both cases, $n\in a_\tau \cup f_\sigma^p(\tau)$, finishing the proof.
\end{proof}

The following fact will be needed in the proof of Claim~\ref{clai:proof_needs_sigma_centered}:

\begin{coro}\label{coro:alles_sigma_centered}
$\PP_\alpha$ is
$\sigma$-centered
for each $\alpha \leq \lambda$.
\end{coro}

More generally, the same holds for $\PP_\alpha/\PP_\eta$ for $\eta<\alpha$.

\begin{proof}[Proof of Corollary~\ref{coro:alles_sigma_centered}]
Since Mathias forcing with respect to a filter is always $\sigma$-centered (see the remark after Definition~\ref{defi:Mathias_with_filter}) and $\QQ_\alpha^{\leq\beta}/\QQ_\alpha^{< \beta}$ is densely embeddable into such a forcing by the above lemma, also  $\QQ_\alpha^{\leq\beta}/\QQ_\alpha^{< \beta}$ is $\sigma$-centered.

Recall that $\Lambda_\eta<\cc^+$ for every $\eta<\alpha$,
and $\alpha \leq \lambda \leq \cc$,
so $\PP_\alpha$ is a finite support iteration of $\sigma$-centered forcings of length strictly less than~$\cc^+$.
As a matter of fact,
the finite support iteration of $\sigma$-centered
forcings
of length strictly less
than~$\cc^+$
is $\sigma$-centered
(the result was mentioned without proof in \cite[proof of Lemma~2]{tall}; for a proof, see \cite{Blass_mathoverflow} or \cite[Lemma~5.3.8]{Fiorella}).
\end{proof}

Also the following 
 lemma will be used in the proof of
Claim~\ref{clai:proof_needs_sigma_centered}.
It is 
similar to 
a well-known fact about branches of certain trees; 
see, e.g., \cite[Lemma~3.8]{Mitchell_paper} or \cite{Mitchell_thesis}.
For the convenience of the reader, we provide an explicit proof here.

\begin{lemm}\label{lemm:square_cc_fresh}
If $\mathbb{P}\times \mathbb{P}$ has the c.c.c.\ and $\cf(\delta)>\omega$, then, 
in $V[\PP]$, 
every new function 
from $\delta$ to the ordinals has an initial segment which is new. 
\end{lemm}

\begin{proof}
Assume 
towards a contradiction that 
there exists $p\in \PP$ and a $\mathbb{P}$-name $\dot{f}$ such that $
p$ forces $\dot{f}{:}\ \delta\rightarrow \ord$ is not in $V$
and $ \dot{f}\restr \gamma\in V$ for each $\gamma <\delta$.
Therefore, we can,
by induction on $i<\om_1$,
construct
$\alpha_i<\delta$, $p_i \leq p$, and $q_i \leq p$
such that $p_i$ and $q_i$ decide $\dot{f}$ up to $\alpha_i$, and $\alpha_i$ is the first point about which $p_i$ and $q_i$ disagree; more precisely,
there is $s_i{:}\ \alpha_i+1\rightarrow \ord$ and $t_i{:}\ \alpha_i+1\rightarrow \ord$ such that
\begin{enumerate}

 \item
 $\alpha_j < \alpha_i$ for each $j < i$,

 \item
 $p_i\Vdash \dot{f}\restr(\alpha_i+1)=s_i$,

 \item
 $q_i\Vdash \dot{f}\restr(\alpha_i+1)=t_i$,

 \item $s_i\neq t_i$, and $s_i\restr \alpha_i=t_i\restr \alpha_i$.

 \end{enumerate}
Consider
$\seqlangle (p_i,q_i) \with i<\om_1 \seqrangle$ and use that
$\mathbb{P}\times \mathbb{P}$ 
has the c.c.c.\
to obtain
$i_0<i_1$ such that
$(p_{i_0},q_{i_0})$ and $(p_{i_1},q_{i_1})$ are compatible, and fix
$(\bar{p},\bar{q})$
with
$(\bar{p},\bar{q}) \leq (p_{i_0},q_{i_0})$
and
$(\bar{p},\bar{q}) \leq (p_{i_1},q_{i_1})$.
It follows that both $\bar{p}$ and $\bar{q}$ force that
$\dot{f}\restr \alpha_{i_1}=s_{i_1}\restr \alpha_{i_1}$. Moreover, $\bar{p}\Vdash \dot{f}\restr (\alpha_{i_0}+1)=s_{i_0}$ and $\bar{q}\Vdash \dot{f}\restr (\alpha_{i_0}+1)=t_{i_0}$,
but $s_{i_0}\neq t_{i_0}$,
which easily yields (using $\alpha_{i_0}<\alpha_{i_1}$) a contradiction.
\end{proof}

\subsection{The filters are $\B$-Canjar}\label{subsec:The_filters_are_B_Canjar}

To finish the proof of Main Theorem~\ref{maintheo:maintheo_general}, we have to show that
$\bfrak = \om_1$
holds true in the final extension.

Recall that the setup
is the following.
Our very ground model~$V_0$ is a model of CH; therefore, its set of reals
$$\B = \omega^\omega \cap V_0$$
has size $\omega_1$. Clearly,
$\B$ is an unbounded family in~$V_0$.
We will show that $\B$ remains unbounded in the course of the iteration, thereby witnessing $\bfrak = \om_1$ in the final model.

First, observe that $\B$ is still unbounded in~$V$, the extension of $V_0$ by
$\mu$ many Cohen reals (due to the fact that $\CC_\mu$ does not add dominating reals).

In Section~\ref{subsec:layering_and_F_beta}, we have defined filters $\F_\alpha^\beta$ for $\alpha < \lambda$ and $\beta < \Lambda_\alpha$
(and their canonical filter bases~$\FBa_\alpha^\beta$)
and have shown that $\QQ_\alpha$ is equivalent to the finite support iteration of the Mathias forcings $\Mathias{\F_\alpha^\beta}$.
In particular,
$\PP_\alpha*\QQ_\alpha^{<\beta}
* \Mathias{\F_\alpha^\beta}
=
\PP_\alpha*\QQ_\alpha^{\leq\beta}$,
and
$\PP_\alpha * \QQ_\alpha^{<\Lambda_\alpha} = \PP_{\alpha + 1}$.
Note that $\B$ is countably directed 
(see~\eqref{eq:countably_bounded} from Theorem~\ref{theo:preservation_unboundedness_JudahShelah});
therefore
it suffices to show that $\B$ remains unbounded at successor steps of our ``fine'' iteration:
for each $\alpha < \lambda$ and each $\beta<\Lambda_\alpha$, the unboundedness of~$\B$ is preserved by~$\Mathias{\F_\alpha^\beta}$.
To achieve this,
we will show that
the filters $\F_\alpha^\beta$ are $\B$-Canjar in $V[\PP_\alpha*\QQ_\alpha^{<\beta}]$
(see
Lemma~\ref{lemm:main_Canjar_sum_induction}(2)).
Actually, we show for every $\beta^*$ that $\F_\alpha^\beta$ is $\B$-Canjar in $V[\PP_\alpha*\QQ_\alpha^{<\beta^*}]$ whenever $\FBa_\alpha^\beta\in V[\PP_\alpha*\QQ_\alpha^{<\beta^*}]$, i.e., we show the $\B$-Canjarness of a filter $\F_\alpha^\beta$ as soon as it exists.

In many cases, we use a genericity argument
to show that the filters are $\B$-Canjar at the stage where they appear, but we need the $\B$-Canjarness in
later stages
of the iteration,
for two reasons: first, we want to force with this filter in a later stage, and second, we want to use the $\B$-Canjarness of an older filter to show the $\B$-Canjarness of a filter which appears later.
As mentioned earlier,
 the notion of $\B$-Canjarness of a filter is not absolute, therefore we will use our
method from Section~\ref{subsec:method_to_preserve}
 to guarantee that
the $\B$-Canjarness of
the
filter is
not destroyed by the other Mathias forcings
along the iteration.
 This method is based on finite sums of filters, therefore we show that all finite sums of filters which exist in $V[\PP_\alpha*\QQ_\alpha^{<\beta^*}]$ are $\B$-Canjar
(see
Lemma~\ref{lemm:main_Canjar_sum_induction}(3)).

\begin{lemm}\label{lemm:main_Canjar_sum_induction}
For every $\alpha<\lambda$, for every $\beta^*<\Lambda_\alpha$,
\begin{enumerate}

\item
$\B$ is unbounded in $V[\PP_\alpha*\QQ_\alpha^{<\beta^*}]$,

\item $\F_\alpha^{\beta^*}$ is $\B$-Canjar in $V[\PP_\alpha * \QQ_\alpha^{<\beta^*}]$,

\item
if $m\in \omega$ and $\beta_0,\dots,\beta_{m-1}<\Lambda_\alpha$ with
$\FBa_\alpha^{\beta_k}\in V[\PP_\alpha*\QQ_\alpha^{<\beta^*}]$ for every $k<m$,
then
$\Fsum_{k<m}\F_\alpha^{\beta_k}$ is $\B$-Canjar in $V[\PP_\alpha*\QQ_\alpha^{<\beta^*}]$.

\end{enumerate}
\end{lemm}

\begin{proof}
First note that, for each $\alpha<\lambda$ and $\beta^*<\Lambda_\alpha$, (2) is a special instance of (3): in fact,
$\FBa_\alpha^{\beta^*}\in V[\PP_\alpha*\QQ_\alpha^{<\beta^*}]$, so (3) for $m=1$ and $\beta_0 = \beta^*$ is (2). However, we need (3) in order to carry out the induction (to preserve $\B$-Canjarness of our filters).

We prove (1) and (3) (and hence (2))
by (simultaneous)
induction
on
the pairs $(\alpha, \beta^*)$
(with the lexicographical ordering).
So suppose that (1) and (3) hold
for each
$(\alpha', \beta')<_{lex}(\alpha,\beta^*)$, i.e., for each pair
with
$\alpha' < \alpha$ (and $\beta'$ arbitrary)
or $\alpha' = \alpha$ and $\beta' < \beta^*$.

\vspace{0.4em}
\textbf{Proof of~(1)}:
\vspace{0.4em}

\emph{For $\alpha = \beta^* = 0$},
just note that $\B$ is unbounded
in~$V[\PP_0 * \QQ_0^{<0}] = V$, since $V$ is the
extension by Cohen forcing~$\CC_\mu$ (which does not add dominating reals)
of our GCH ground model~$V_0$.

\emph{In case~$\beta^* = \beta' + 1$ is a successor ordinal},
we 
use the fact that~(1) holds for
$\alpha = \alpha'$ and
$\beta'$ by induction, so $\B$ is unbounded in the extension by
$\PP_\alpha * \FKWW_\alpha^{<\beta'}$;
by Lemma~\ref{lemm:quotient_is_mathias},
$\FKWW_\alpha^{\leq \beta'}/\FKWW_\alpha^{<\beta'}$
is forcing equivalent to $\Mathias{\F_\alpha^{\beta'}}$, i.e.,
$\FKWW_\alpha^{<\beta^*} = \FKWW_\alpha^{<\beta'} * \Mathias{\F_\alpha^{\beta'}}$; since
(2) holds for~$\beta'$ by induction,
$\Mathias{\F_\alpha^{\beta'}}$ preserves the unboundedness of~$\B$, hence the same is true for $\PP_\alpha * \FKWW_\alpha^{<\beta^*}$, as desired.

\emph{In case $(\alpha,\beta^*)$ is a limit point of the
lexicographical
ordering (i.e., $\beta^* = 0$ or
$\beta^*$ is a limit ordinal)},
we use the fact that
$\PP_\alpha * \FKWW_\alpha^{<\beta^*}$ is the limit of a finite support iteration of c.c.c.\ forcings, and that (1) holds for each
$(\alpha', \beta') <_{lex} (\alpha,\beta^*)$;
so we can apply Theorem~\ref{theo:preservation_unboundedness_JudahShelah}
to conclude (1)
for~$(\alpha,\beta^*)$.

\vspace{0.4em}
\textbf{Proof of~(3)}:
\vspace{0.4em}

Fix $\alpha$. By (1), $\B$ is unbounded in $V[\PP_\alpha]$.
We say that $\rho\in \lala$ is \emph{pre-$T_\alpha$-minimal}
 if
it is the predecessor of a minimal node of~$T_\alpha$; it is straightforward to check that
this is the case if and only if
\begin{itemize}
\item $\rho \in V[\PP_\alpha]$,
\item $\rho\notin V[\PP_\eta]$ for any $\eta<\alpha$, and
\item for every $\gamma<|\rho|$, there exists $\eta<\alpha$ with $\rho\restr \gamma\in V[\PP_{\eta}]$.
\end{itemize}
Note that for $\alpha=0$, the only pre-$T_\alpha$-minimal node is the root $\seqlangle \seqrangle$, and for $\alpha>0$, all pre-$T_\alpha$-minimal nodes have limit length.

We proceed by induction on $\beta^*$.

\vspace{0.4em}
\textbf{Base step $\beta^*=0$}:
\vspace{0.4em}

\begin{sloppypar}
Let $\beta_0,\dots,\beta_{m-1}$
be such that
$\FBa_\alpha^{\beta_k}\in V[\PP_\alpha*\QQ_\alpha^{<0}]$
for each $k < m$.
Since $\FBa_\alpha^{\beta_k}\in V[\PP_\alpha*\QQ_\alpha^{<0}]= V[\PP_\alpha]$, it follows that $\sigma_\alpha^{\beta_k}=\rho_k\concat 0$ for some pre-$T_\alpha$-minimal node~$\rho_k$: indeed, observe that
$\FBa_\alpha^\beta$ contains elements which are only added by $\QQ_\alpha$ (and hence $\FBa_\alpha^\beta \notin V[\PP_\alpha]$) whenever
$\sigma_\alpha^\beta = \rho\concat i$ with $\rho$ not pre-$T_\alpha$-minimal or $i > 0$
(because $a_\tau \in \FBa_\alpha^\beta$ or $\omega \setminus a_\tau \in \FBa_\alpha^\beta$ for some $\tau \in T_\alpha$).

\end{sloppypar}

If $\cf(|\rho_k|)$ is countable for all $k<m$, the filter $\Fsum_{k<m}\F_\alpha^{\beta_k}$ is countably generated, hence it follows by Lemma~\ref{lemm:countable_Canjar} that it is $\B$-Canjar.

In particular, for $\alpha=0$, the only pre-$T_\alpha$-minimal node is $\rho = \seqlangle \seqrangle$, hence this finishes the proof for $\alpha = \beta^* = 0$. So assume $\alpha>0$ for the rest of the proof of the base step.

If $\cf(\alpha)\leq\omega$ (and $\alpha > 0$), all pre-$T_\alpha$-minimal nodes $\rho$ have $\cf(|\rho|)=\omega$:

\begin{clai}\label{clai:proof_needs_sigma_centered}
Let $\rho$ be a pre-$T_\alpha$-minimal node and
$\cf(|\rho|)>\omega$. Then
\begin{enumerate}
\item $\cf(\alpha)>\omega$, and
\item there exists no $\eta<\alpha$ such that $\rho\restr \gamma \in V[\PP_\eta]$ for all $\gamma<|\rho|$.
\end{enumerate}
\end{clai}

\begin{proof}
Let us first show (2). Fix some $\eta<\alpha$. 
Since $\rho$ is pre-$T_\alpha$-minimal, 
$\rho \in V[\PP_\alpha] \setminus V[\PP_\eta]$, i.e.,  
$\rho$ is new (with respect to 
$\PP_\alpha/\PP_\eta$). 
By Corollary~\ref{coro:alles_sigma_centered}, $\PP_\alpha/\PP_\eta$ is $\sigma$-centered, hence  in particular $(\PP_\alpha/\PP_\eta)\times (\PP_\alpha/\PP_\eta)$ has the c.c.c.. 
Therefore, by 
Lemma~\ref{lemm:square_cc_fresh}, the new function $\rho$ has an initial segment which is 
not in $V[\PP_\eta]$ 
(because it is new with respect to 
$\PP_\alpha/\PP_\eta$).

Now let us show (1). Assume towards contradiction that $\cf(\alpha)\leq \omega$, and let $\seqlangle \alpha_n\with n\in \omega\seqrangle$ be increasing cofinal in~$\alpha$ (in case $\alpha$ is a successor, let $\alpha_n$ be its predecessor for every~$n$). For every $\gamma<|\rho|$, let $n\in \omega$ be such that $\rho\restr \gamma \in V[\PP_{\alpha_n}]$,
which is possible since $\rho$ is pre-$T_\alpha$-minimal. Since $\cf(|\rho|)>\omega$, there exists $n^*\in \omega$ such that $\rho\restr \gamma \in V[\PP_{\alpha_{n^*}}]$ for cofinally many $\gamma<|\rho|$ (and hence for all $\gamma<|\rho|$), contradicting~(2).
\end{proof}

So we can assume that
$\cf(\alpha)>\omega$.
We first argue that $\cf(|\rho|)>\omega$ for all pre-$T_\alpha$-minimal nodes~$\rho$.
Assume towards contradiction that $\cf(|\rho|)=\omega$ and $\rho$ is pre-$T_\alpha$-minimal. Let
$\seqlangle \gamma_n\with n\in \omega\seqrangle$
be increasing cofinal in $|\rho|$. For every $n<\omega$, let $\alpha_n<\alpha$ be such that $\rho\restr \gamma_n \in V[\PP_{\alpha_n}]$. Since $\cf(\alpha)>\omega$, there exists $\alpha'<\alpha$ with $\alpha_n<\alpha'$ for every $n$. There are no new countable sequences of elements of $V[\PP_{\alpha'}]$ in  $V[\PP_{\alpha}]$, because $V[\PP_{\alpha}]$ is a limit of uncountable cofinality of a c.c.c.\ iteration, hence there exists $\alpha''<\alpha$ such that $\seqlangle \rho\restr \gamma_n \with n\in \omega \seqrangle \in V[\PP_{\alpha''}]$. Hence also $\rho \in  V[\PP_{\alpha''}]$,
so it is not pre-$T_\alpha$-minimal,
a contradiction.

Now we will show that $\Fsum_{k<m} \F_\alpha^{\beta_k}$ is $\B$-Canjar in $V[\PP_\alpha]$, using the characterization from Theorem~\ref{theo:OsvaldoHrusakMartinez}.
 Let $\seqlangle X_n \with n\in \omega \seqrangle \in V[\PP_\alpha]$ be positive for $\Fsum_{k<m} \F_\alpha^{\beta_k}$.
We want to show that there exists $f\in \B$ such that $\bar{X}_f$ is positive for  $\Fsum_{k<m} \F_\alpha^{\beta_k}$. Since $\seqlangle X_n \with n\in \omega \seqrangle$ is
hereditarily
countable and $\cf(\alpha)>\omega$,
there exists $\eta<\alpha$ with $\seqlangle X_n \with n\in \omega \seqrangle\in V[\PP_\eta]$. Moreover, let $\eta$ be large enough such that for all $j<k<m$ 
with
$\rho_j\neq \rho_k$, there exists a successor $\delta<|\rho_j|, |\rho_k|$ such that $\rho_j\restr \delta \neq \rho_k\restr \delta$ and $a_{\rho_j\restr \delta}, a_{\rho_k\restr \delta} \in V[\PP_\eta]$.
For every $k<m$,
let $\gamma_k<|\rho_k|$ be
such that $a_{\rho_k\restr \gamma_k} \notin V[\PP_\eta]$. Such $\gamma_k$ exist, bexause the $\rho_k$ are pre-$T_\alpha$-minimal, using (2) from the above claim.
Clearly  $a_{\rho_k\restr \gamma_k}\in V[\PP_\alpha]$ for every $k<m$.
The filter $\Fsum_{k<m} \genlangle a_{\rho_k\restr \gamma_k} \genrangle$ is countably generated and hence $\B$-Canjar in $V[\PP_\alpha]$ (see Lemma~\ref{lemm:countable_Canjar}).
Note that  $\Fsum_{k<m} \genlangle a_{\rho_k\restr \gamma_k} \genrangle\subseteq \Fsum_{k<m}\F_\alpha^{\beta_k}$, hence $\seqlangle X_n \with n\in \omega\seqrangle$ is positive for $\Fsum_{k<m} \genlangle a_{\rho_k\restr \gamma_k} \genrangle$.
Therefore we can fix $f\in \B$ such that $\bar{X}_f$ is positive for $\Fsum_{k<m} \genlangle a_{\rho_k\restr \gamma_k} \genrangle$.
Note that $\bar{X}_f\in V[\PP_\eta]$.

We will use a genericity argument to show that $\bar{X}_f$ is positive for $\Fsum_{k<m}\F_\alpha^{\beta_k}$.
 It is enough to show that for all successors $\delta_k<|\rho_k|$, for all $l_k\in \omega$ there exists $s\in \bar{X}_f$ with $s\subseteq\Fsum_{k<m} (a_{\rho_k\restr \delta_k} \setminus l_k)$, because sets of this form are a basis for the filter.
If $\delta_k\leq\gamma_k$ for all $k$, this holds by the choice of~$f$.

We show by induction on $\eta\leq\eta' <\alpha$ that for all successors $\delta_k<|\rho_k|$ and all $l_k<\omega$, if all $a_{\rho_k\restr \delta_k} \in V[\PP_{\eta'}]$ then $V[\PP_{\eta'}]\models `` \exists s\in \bar{X}_f \; s\subseteq \Fsum_{k<m}(a_{\rho_k\restr \delta_k} \setminus l_k)$''.
Note that this holds for $\eta'=\eta$ by choice of $f$, and that at limit steps of the induction no new $a_{\rho_k\restr \delta_k}$ appear, so we only have to show it for successors.
Assume that it holds for $\eta'$ and show it for $\eta'+1$.

For every $k<m$, let $\delta_k<|\rho_k|$ with $a_{\rho_k\restr \delta_k}\in V[\PP_{\eta'+1}]$ and $l_k\in \omega$ be given. Let $p\in \QQ_{\eta'}$. We will show that there exists $q\leq p$ and $s\in \bar{X}_f$ such that $q\forces  s\subseteq \Fsum_{k<m}(a_{\rho_k\restr \delta_k} \setminus l_k)$.
Without loss of generality we can assume that $\rho_k\restr \delta_k \in \dom(p)$ for all $k<m$ with $\rho_k\restr \delta_k\in T_{\eta'}$, and that $p$ is a full condition.

For every $k<m$, define $\Sigma_k$.
If $\rho_k\restr \delta_k\in T_{\eta'}$, let
$$\Sigma_k:=\bigcup\{  \dom(f_{\rho_k\restr \gamma}^p) \cap T'_{\eta'} \with \gamma\leq \delta_k \land \rho_k\restr \gamma\in \dom(p)\}.$$
If $\rho_k\restr \delta_k\notin T_{\eta'}$, let $\Sigma_k:= \{ \rho_k\restr \delta_k\}$.
Let $\Sigma:= \bigcup_{k<m}\Sigma_k$.
For every $k<m$, let $\sigma_k$ be the longest initial segment of $\rho_k$ which belongs to~$\Sigma$ (if there exists one; let\footnote{We just have to choose any initial segment of $\rho_k$ which belongs to $T'_{\eta'}$ and make sure that $\sigma_k=\sigma_j$ if $\rho_k=\rho_j$. Alternatively, in such cases, we could replace $a_{\sigma_k}$ by $\omega$ below.} $\sigma_k:=\rho_k\restr 1$ otherwise).
Note that $\sigma_k=\sigma_j$ if $\rho_k=\rho_j$, and
that $a_{\sigma_k}\in V[\PP_{\eta'}]$ for every $k<m$.
Now let $N\in \omega$ be large enough such that
\begin{itemize}
\item $N\geq l_k$ for every $k<m$,
\item $N\geq|s_\sigma^p|$ for every $\sigma\in \dom(p)$,
\item $a_{\sigma_k}\setminus N \subseteq a_\tau$ for all $\tau\in \Sigma_k$, for all $k<m$.
\end{itemize}
By hypothesis, in $V[\PP_{\eta'}]$,
we can fix
$s\in \bar{X}_f$ with $s\subseteq   \Fsum_{k<m} (a_{\sigma_k} \setminus N)$.

To get $q$, extend $p$ as follows. For every $k<m$, for every $\gamma \leq \delta_k$ with $\rho_k\restr \gamma \in \dom(p)$, let
$$s_{\rho_k\restr \gamma}^q:={ s_{\rho_k\restr \gamma}^p} \concat ( \zero \restr [| s_{\rho_k\restr \gamma}^p|,N)) \concat (a_{\sigma_k}\restr [ N, \max(s)]).$$
Observe that, if $\rho_k\neq \rho_j$,
there is no $\tau\in\dom(p)$ with $\tau\unlhd \rho_k$ and $\tau\unlhd\rho_j$ (by choice of $\eta$ and since $\eta'\geq \eta$); so the above is well-defined, since $\sigma_k=\sigma_j$ if $\rho_k=\rho_j$.

Note that $\eta$ was chosen large enough such that for all $\gamma_k \leq \delta_k$ and $\gamma_j\leq \delta_j$, if $\rho_j\restr \gamma_j \neq \rho_k\restr \gamma_k$, then they are not in the same block; so, in particular, $\rho_j\restr \gamma_j \notin \dom(h_{\rho_k\restr \gamma_k}^p)$ and  $\rho_k\restr \gamma_k \notin\dom( h_{\rho_j\restr \gamma_j}^p)$.
Therefore,
the requirement (8) from Definition~\ref{defi:main_definition} is fulfilled. It is easy to see that the other requirements of Definition~\ref{defi:main_definition} are fulfilled as well, hence $q$ is a condition.

It is easy to check that $q$ forces
$s\subseteq \Fsum_{k<m} (a_{\rho_k\restr \delta_k} \setminus l_k)$, as desired.

\vspace{0.4em}
\textbf{Successor step}:
\vspace{0.4em}

Let us say that a filter $\F_\alpha^\beta$ (and
its filter base $\FBa_\alpha^\beta$) is \emph{new in
$V[\PP_\alpha*\QQ_\alpha^{<\beta^*}]$} if
$\FBa_\alpha^\beta \in V[\PP_\alpha*\QQ_\alpha^{<\beta^*}]$ and  $\FBa_\alpha^\beta \notin V[\PP_\alpha*\QQ_\alpha^{<\delta^*}]$
for all
$\delta^*<\beta^*$.

Now assume that we have shown (3)
for $\beta^*$; let us show
it
for $\beta^*+1$.

If $\FBa_\alpha^{\beta_k}\in V[\PP_\alpha*\QQ_\alpha^{<\beta^*}]$ for every $k<m$, then by induction hypothesis
$\left(\Fsum_{k<m}\F_\alpha^{\beta_k}\right) \fsum \F_\alpha^{\beta^*}$
is $\B$-Canjar in $ V[\PP_\alpha*\QQ_\alpha^{<\beta^*}]$, hence, by Lemma~\ref{lemm:sum_preserves_Canjar},
$\Fsum_{k<m}\F_\alpha^{\beta_k}$ is $\B$-Canjar in $ V[\PP_\alpha*\QQ_\alpha^{<\beta^*}*\Mathias{\F_\alpha^{\beta^*}}] = V[\PP_\alpha*\QQ_\alpha^{<\beta^*+1}]$.

It is easy to check that
there are
exactly
two new filters in $V[\PP_\alpha*\QQ_\alpha^{<\beta^*+1}]$: $\F_\alpha^\beta$, where $\beta$ is such that $\sigma_\alpha^\beta= {\sigma^{\beta^*}_\alpha} \concat 0$, and  $\F_\alpha^{\beta'}$, where $\beta'
= \beta^* + 1$
(i.e., $\sigma_\alpha^{\beta'}= \rho\concat( i+1)$
and
 $\sigma_\alpha^{\beta^*} =\rho\concat i$). Both $\F_\alpha^{\beta}$ and $\F_\alpha^{\beta'}$ are extensions of $\F_\alpha^{\beta^*}$
 by
 one new set.
  Therefore, 
  the filter $\Fsum_{k<m} \F_\alpha^{\beta_k}$ is an extension of
   $\Fsum_{k<m} \F_\alpha^{\tilde{\beta}_k}$ 
   by finitely many sets,
   where $\tilde{\beta}_k=\beta^*$ if $\beta_k = \beta$ or $\beta_k = \beta'$, and
   $\tilde{\beta}_k=\beta_k$ otherwise.
By the above,
$\Fsum_{k<m} \F_\alpha^{\tilde{\beta}_k}$
is
$\B$-Canjar in $V[\PP_\alpha*\QQ_\alpha^{<\beta^*+1}]$,
hence,
by Lemma~\ref{lemm:Canjar_countable_ext},
also $\Fsum_{k<m} \F_\alpha^{\beta_k}$
is
$\B$-Canjar in $V[\PP_\alpha*\QQ_\alpha^{<\beta^*+1}]$.

\vspace{0.4em}
\textbf{Limit step}:
\vspace{0.4em}

Now assume that $\beta^*$ is a limit, and that we have shown
(3) for all $\delta^*<\beta^*$;
let us show it for $\beta^*$. 
If for each $k<m$ there exists $\delta^*_k<\beta^*$ such that $\FBa_\alpha^{\beta_k}\in V[\PP_\alpha*\QQ_\alpha^{<\delta_k^*}]$, then there exists $\bar{\delta}^*<\beta^*$ such that $\FBa_\alpha^{\beta_k}\in V[\PP_\alpha*\QQ_\alpha^{<\bar{\delta}^*}]$ for all $k<m$.
By induction hypothesis, $\Fsum_{k<m} \F_\alpha^{\beta_k}$ is $\B$-Canjar in $ V[\PP_\alpha*\QQ_\alpha^{<\delta^*}]$ for every $\bar{\delta}^*\leq \delta^*<\beta^*$,
hence, by Lemma~\ref{lemm:limits_preserve_B_Canjar}, it is $\B$-Canjar in $ V[\PP_\alpha*\QQ_\alpha^{<\beta^*}]$.

Now we have to consider new filters.
There are two cases: either $\beta^*$ is such that $\sigma_\alpha^{\beta^*}$ has $0$ as its last entry, or such that it has a limit ordinal~$i$ as its last entry.

\vspace{0.4em}
\textbf{First case}:
$\sigma_\alpha^{\beta^*}={\rho}\concat 0$ for some $\rho$
\vspace{0.4em}

Let us first argue
that
there are no new filters unless $|\rho|$ is a limit and $\sigma_\alpha^{\beta^*}$ is the first node of its level in the enumeration (i.e., $|\sigma_\alpha^\delta| < |\rho|$ for each $\delta < \beta^*$).
If $\sigma_\alpha^{\beta^*}$ is not the first node of the level in the enumeration,
then there are no new filters in $V[\PP_\alpha*\QQ_\alpha^{<\beta^*}]$: If $\FBa_\alpha^\beta\in V[\PP_\alpha*\QQ_\alpha^{<\beta^*}]$, then there exists $\delta^*<\beta^*$ such that  $\FBa_\alpha^\beta\in V[\PP_\alpha*\QQ_\alpha^{<\delta^*}]$, because
$\FBa_\alpha^\beta$ contains
-- from the sets of this level --
only sets
within
one block and only boundedly many sets
within
this block.
Similarly, if $|\rho|$ is a successor (and $\sigma_\alpha^{\beta^*}$ is the first node of its level in the enumeration), there are no new filters in $V[\PP_\alpha*\QQ_\alpha^{<\beta^*}]$:
If $\FBa_\alpha^\beta\in V[\PP_\alpha*\QQ_\alpha^{<\beta^*}]$, then there exists $\delta^*<\beta^*$ such that  $\FBa_\alpha^\beta\in V[\PP_\alpha*\QQ_\alpha^{<\delta^*}]$, because
$\FBa_\alpha^\beta$ contains -- from the sets of level $|\rho|$ -- only boundedly many sets.

So we assume from now on that $|\rho|$ is a limit and $\sigma_\alpha^{\beta^*}$ is the first node of its level in the enumeration.
In this case, there are many new filters $\F_\alpha^\beta$ in $V[\PP_\alpha*\QQ_\alpha^{<\beta^*}]$;
in fact, it is easy to check that $\F_\alpha^\beta$ is new if and only if the following holds:
$\sigma_\alpha^\beta= \bar{\rho}\concat 0$ for some $\bar{\rho}$ with $|\bar{\rho}|=|\rho|$ and $\bar{\rho}$ not pre-$T_\alpha$-minimal.
Observe that $\FBa_\alpha^\beta =
\{ a_{\bar{\rho} \restr \gamma} \with \gamma < |\rho| \}$.
Let $\beta_0,\dots,\beta_{m-1}$
be such that
$\FBa_\alpha^{\beta_k}\in V[\PP_\alpha*\QQ_\alpha^{<\beta^*}]$
for each $k < m$.
We want to show that $\Fsum_{k<m}\F_\alpha^{\beta_k}$ is $\B$-Canjar in $V[\PP_\alpha*\QQ_\alpha^{<\beta^*}]$.

In case $\cf(|\rho|)=\omega$, we can use  Lemma~\ref{lemm:Canjar_oplus_countably} and the remark afterwards to finish the proof:  $\Fsum_{k<m}\F_\alpha^{\beta_k}$ is a sum of filters,
in which
the new filters are countably generated, whereas the sum of the filters which are not new is $\B$-Canjar (see the first paragraph of the limit step).

So let us assume from now on that $\cf(|\rho|)  > \om$.
Let $\neu \subs m$ be the set of $k < m$ such that $\F_\alpha^{\beta_k}$ is a new filter, and $\alt = m \setminus \neu$ be the set of $k < m$ such that $\F_\alpha^{\beta_k}$ is not new.
For each $k\in\neu$, we can
fix $\rho_k$ such that $\sigma_\alpha^{\beta_k} = \rho_k \concat 0$ (with $|\rho_k| = |\rho|$ and $\rho_k$ not pre-$T_\alpha$-minimal).

Let $\seqlangle X_n \with n\in \omega\seqrangle \in V[\PP_\alpha*\QQ_\alpha^{<\beta^*}]$
be positive for $\Fsum_{k<m}\F_\alpha^{\beta_k}$.
Since $\seqlangle X_n \with n\in \omega \seqrangle$ is hereditarily countable and $\FKWW_\alpha^{<\beta^*}$ has the c.c.c., there exists a hereditarily countable name~$\dot{X}$ for  $\seqlangle X_n \with n\in \omega \seqrangle$. Since the conditions in $\FKWW_\alpha^{<\beta^*}$ have finite
domain, the
 union of all the domains of conditions which occur in the name $\dot{X}$ is countable.
Let $\gamma<|\rho|$ be a successor ordinal large enough such that the following hold:
\begin{itemize}
\item $\seqlangle X_n \with n\in \omega \seqrangle \in V[\PP_\alpha*\QQ_\alpha^{\lambda^{\leq \gamma}}]$ (this is possible due to $\cf(|\rho|)>\omega$).

\item For all $j,k \in \neu$, if $\rho_j \neq \rho_k$, then $\rho_j$ and $\rho_k$ split before $\gamma$.

\item For all $k \in \alt$, either $|\sigma_\alpha^{\beta_k}|<\gamma$ or\footnotemark{} $|\sigma_\alpha^{\beta_k}|>|\rho|$.

\footnotetext{Note that $|\sigma_\alpha^{\beta_k}|>|\rho|$ is only possible if
$\sigma_\alpha^{\beta_k}=
\tilde{\rho}\concat 0$ for a pre-$T_\alpha$-minimal node $\tilde{\rho}$.}

\item $a_{\rho_k\restr \gamma}\notin V[\PP_\alpha]$ for all $k\in \neu$, i.e., $\rho_k\restr \gamma \in T_\alpha$ (which is possible because $\rho_k$ is not pre-$T_\alpha$-minimal).
\end{itemize}
For every $k<m$ with $k \in \alt$, we have
$\FBa_\alpha^{\beta_k} \in V[\PP_\alpha*\QQ_\alpha^{\lambda^{\leq \gamma}}]$
by choice of $\gamma$;
let $\tilde{\FBa}_\alpha^{\beta_k} := \FBa_\alpha^{\beta_k}$.
For $k \in \neu$, let $\tilde{\FBa}_\alpha^{\beta_k}:=
\{a_{\rho_k\restr \gamma}\}$.

\begin{sloppypar}
As above, we can use  Lemma~\ref{lemm:Canjar_oplus_countably} and the remark afterwards to show that
$\Fsum_{k<m} \genlangle \tilde{\FBa}_\alpha^{\beta_k}\genrangle$ is $\B$-Canjar in $V[\PP_\alpha*\QQ_\alpha^{<\beta^*}]$.
Indeed, $\Fsum_{k \in \alt} \genlangle \tilde{\FBa}_\alpha^{\beta_k}\genrangle$ is $\B$-Canjar in $V[\PP_\alpha*\QQ_\alpha^{<\beta^*}]$ by the first paragraph of the limit step, and
for each $k \in \neu$,
$\genlangle \tilde{\FBa}_\alpha^{\beta_k}\genrangle$ is countably generated.
Moreover, $\Fsum_{k<m}\genlangle \tilde{\FBa}_\alpha^{\beta_k}\genrangle \subseteq \Fsum_{k<m}\F_\alpha^{\beta_k}$, hence $\seqlangle X_n \with n \in \omega\seqrangle$ is positive for $\Fsum_{k<m}\genlangle \tilde{\FBa}_\alpha^{\beta_k}\genrangle$.
So we can fix $f\in\B$ such that $\bar{X}_f$ is positive for  $\Fsum_{k<m}\genlangle \tilde{\FBa}_\alpha^{\beta_k}\genrangle$.
Since $\seqlangle X_n \with n \in \omega \seqrangle$ and $\Fsum_{k<m}\genlangle \tilde{\FBa}_\alpha^{\beta_k}\genrangle$ are in $V[\PP_\alpha*\QQ_\alpha^{\lambda^{\leq \gamma}}]$, this holds in $V[\PP_\alpha*\QQ_\alpha^{\lambda^{\leq \gamma}}]$.
\end{sloppypar}

Now
we use a genericity argument in $\QQ_\alpha^{<\beta^*}/\QQ_\alpha^{\lambda^{\leq \gamma}}$ to show that $\bar{X}_f$ is positive for~$\Fsum_{k<m}\F_\alpha^{\beta_k}$.
We have to show that for all $\seqlangle A_k\with k<m\seqrangle$ with $A_k\in \F_\alpha^{\beta_k}$, there exists $s\in \bar{X}_f$ with $s\subseteq \Fsum_{k<m}A_k$.
For $k\in \neu$, we can assume that $A_k=a_{\rho_k\restr \delta_k}\setminus l_k$ with $\gamma<\delta_k<|\rho|$ and $l_k\in \omega$, because
these sets form filter bases (with respect to upwards closure).
For $k\in \alt$, let $B_k:= A_k$, and for $k\in\neu$ (in this case $|\sigma_\alpha^{\beta_k}|=|\rho|+1 >\gamma$), let $B_k:= a_{\rho_k\restr \gamma}$.
By the choice of $f$,
there exists,
for all $N \in \omega$,
an $s\in \bar{X}_f$ with $s\subseteq\Fsum_{k<m}(B_k\setminus N) $.

Let $p\in \QQ_\alpha^{<\beta^*}/\QQ_\alpha^{\lambda^{\leq \gamma}}$. Without loss of generality we can assume that $\rho_k\restr \delta_k\in \dom(p)$ if  $k\in \neu$.
For every $k \in \neu$, define
$$\Sigma_k:=\bigcup\{  \dom(f_{\rho_k\restr \delta}^p) \cap \lambda^{\leq \gamma} \with \delta\leq \delta_k \land \rho_k\restr \delta\in \dom(p)\}.$$
Now let $N\in \omega$ be large enough such that
\begin{itemize}
\item $N\geq l_k$ for every $k\in \neu$,
\item $N\geq|s_\sigma^p|$ for every $\sigma\in \dom(p)$,
\item $a_{\rho_k\restr \gamma}\setminus N \subseteq a_\tau$ for all $\tau\in \Sigma_k$, for all $k\in \neu$.
\end{itemize}
By the above,
we can fix
$s\in \bar{X}_f$ with
 $s\subseteq   \Fsum_{k<m}(B_k\setminus N)$.

To get $q$, extend $p$ as follows.
For every $k \in \neu$,
for every $\delta \leq \delta_k$ with
$\rho_k \restr \delta \in \dom(p)$,
let
$$s_{\rho_k\restr \delta}^q:={ s_{\rho_k\restr \delta}^p} \concat ( \zero \restr [| s_{\rho_k\restr \delta}^p|,N)) \concat (a_{\rho_k\restr \gamma}\restr [ N, \max(s)]).$$

Note that $\gamma$ was chosen large enough
so that for $j,k\in \neu$, if $\rho_j\neq \rho_k$, then they split before $\gamma$,
therefore for $\gamma<\delta<|\rho|$ either $\rho_k\restr \delta = \rho_j\restr \delta$ or they are not in the same block.  In particular $\rho_j\restr \delta \notin \dom(h_{\rho_k\restr \delta}^p)$ and  $\rho_k\restr \delta \notin \dom(h_{\rho_j\restr \delta}^p)$. So the requirement (8) from Definition~\ref{defi:main_definition} is fulfilled. It is easy to see that the other requirements of Definition~\ref{defi:main_definition} are fulfilled as well, hence $q$ is a condition.

It is easy to check that $q$ forces
$s\subseteq \Fsum_{k<m} A_k$, as desired.

\vspace{0.4em}
\textbf{Second case}: $\sigma_\alpha^{\beta^*}=\rho\concat i$ with $i>0$ 
limit
\vspace{0.4em}

In this case, $\F_\alpha^{\beta^*}$ is the only new filter in $V[\PP_\alpha*\QQ_\alpha^{<\beta^*}]$.
Indeed,
for all other $\beta$ either $\FBa_\alpha^\beta$ appeared in an earlier step already or it is not in this model, because for $\beta\neq \beta^*$ the filter base $\FBa_\alpha^\beta$
either
uses only boundedly many elements of
$\{ a_{\rho\concat j}\with j<i\}$ or it uses $a_{\rho\concat i}$ as well.
Let $\beta_0,\dots,\beta_{m-1}$
be such that
$\FBa_\alpha^{\beta_k}\in V[\PP_\alpha*\QQ_\alpha^{<\beta^*}]$
for each $k < m$.
We want to show that $\Fsum_{k<m}\F_\alpha^{\beta_k}$ is $\B$-Canjar in $V[\PP_\alpha*\QQ_\alpha^{<\beta^*}]$.

Let $\seqlangle X_n \with n\in \omega \seqrangle \in V[\PP_\alpha*\QQ_\alpha^{<\beta^*}]$ be positive for~$\Fsum_{k<m}\F_\alpha^{\beta_k}$.
Since $\seqlangle X_n \with n\in \omega \seqrangle$ is hereditarily countable and $\FKWW_\alpha^{<\beta^*}$ has the c.c.c., there exists a hereditarily countable name $\dot{X}$ for  $\seqlangle X_n \with n\in \omega \seqrangle$. Since the conditions in $\FKWW_\alpha^{<\beta^*}$ have finite
domain, the
 union~$D$ of all the domains of conditions which occur in the name $\dot{X}$ is countable.
Let $\beta^{**}<\beta^*$ be such that $\rho\concat 0=\sigma_\alpha^{\beta^{**}}$.
Let $\bar{D}:= \{ \tau \in D \with \exists j<i \;(\tau=\rho\concat j)\}$.
Let $C := \{\sigma_\alpha^{\nu}\with \nu<\beta^{**}\}  \cup \bar{D}$.
Note that
$$
C=\{ \sigma_\alpha^{\nu}\with \nu<\beta^{**}\land |\sigma_\alpha^\nu|\leq|\rho|\} \cup \bar{C}
$$
with
$
\bar{C} = \{ \sigma_\alpha^{\nu}\with \nu<\beta^{**}\land |\sigma_\alpha^\nu|=|\rho|+1\} \cup
\bar{D}
$.
Since
$\{ \sigma_\alpha^{\nu}\with \nu<\beta^{**}\land |\sigma_\alpha^\nu|\leq|\rho|\} =
\lambda^{\leq |\rho|} \cap T_\alpha$
is $\alpha$-$\leftup$-closed, $C$ 
is $\alpha$-eligible, 
so, 
by Lemma~\ref{lemm:C_complete_subforcing}, 
$\FKWW_\alpha^C$ is a complete subforcing of
$\FKWW_\alpha$, and it is a subset of $\QQ_\alpha^{<\beta^*}$; recall that $\FKWW_\alpha^{<\beta^*}$ is also a complete subforcing of~$\FKWW_\alpha$, so, by Lemma~\ref{lemm:compl_compl_is_compl}, $\FKWW_\alpha^C$ is complete in~$\FKWW_\alpha^{<\beta^*}$.

Observe that for all $k<m$ either $ \FBa_\alpha^{\beta_k}\cap V[\PP_\alpha*\QQ_\alpha^{<\beta^{**}}]=\FBa_\alpha^{\beta_k}$ or $ \FBa_\alpha^{\beta_k}\cap V[\PP_\alpha*\QQ_\alpha^{<\beta^{**}}]=\FBa_\alpha^{\beta^{**}}$.
In particular, for every $k<m$ there exists $\beta'_k$ such that $ \FBa_\alpha^{\beta_k}\cap V[\PP_\alpha*\QQ_\alpha^{<\beta^{**}}]=\FBa_\alpha^{\beta'_k}$ and
$\FBa_\alpha^{\beta'_k}\in V[\PP_\alpha*\QQ_\alpha^{<\beta^{**}}]$.
 Hence $\Fsum_{k<m}\genlangle \FBa_\alpha^{\beta_k}\cap V[\PP_\alpha*\QQ_\alpha^{<\beta^{**}}] \genrangle$ is $\B$-Canjar in $V[\PP_\alpha*\QQ_\alpha^{<\beta^{*}}]$.
 For every $k<m$, $ \FBa_\alpha^{\beta_k}\cap V[\PP_\alpha*\QQ_\alpha^{C}]$ is the set $ \FBa_\alpha^{\beta_k}\cap V[\PP_\alpha*\QQ_\alpha^{<\beta^{**}}]$ together with countably many new sets (some of the sets $\omega\setminus a_\tau$ with $\tau\in \bar{D}$), therefore  $\Fsum_{k<m}\genlangle \FBa_\alpha^{\beta_k}\cap V[\PP_\alpha*\QQ_\alpha^{C}] \genrangle$ is a filter generated by  $\Fsum_{k<m}\genlangle \FBa_\alpha^{\beta_k}\cap V[\PP_\alpha*\QQ_\alpha^{<\beta^{**}}]\genrangle$ together with countably many new sets. Hence, by Lemma~\ref{lemm:Canjar_countable_ext}, $\Fsum_{k<m}\genlangle \FBa_\alpha^{\beta_k}\cap V[\PP_\alpha*\QQ_\alpha^{C}]\genrangle$ is $\B$-Canjar in  $V[\PP_\alpha*\QQ_\alpha^{<\beta^*}]$.
Since $\Fsum_{k<m} \genlangle \FBa_\alpha^{\beta_k}\cap V[\PP_\alpha*\QQ_\alpha^{C}]\genrangle\subseteq \Fsum_{k<m} \F_\alpha^{\beta_k}$, the sets $\seqlangle X_n \with n\in \omega \seqrangle$ are also positive for $\Fsum_{k<m} \genlangle\FBa_\alpha^{\beta_k}\cap V[\PP_\alpha*\QQ_\alpha^{C}]\genrangle$.
So we can fix $f\in\B$ such that $\bar{X}_f$ is positive for
$\Fsum_{k<m}\genlangle \FBa_\alpha^{\beta_k}\cap V[\PP_\alpha*\QQ_\alpha^{C}]\genrangle$.
Since $\seqlangle X_n \with n \in \omega \seqrangle$ and $\Fsum_{k<m}\genlangle\FBa_\alpha^{\beta_k}\cap V[\PP_\alpha*\QQ_\alpha^{C}]\genrangle$ are in $V[\PP_\alpha*\QQ_\alpha^{C}]$, this holds in~$V[\PP_\alpha*\QQ_\alpha^{C}]$.

Now
we use a genericity argument in $\QQ_\alpha^{<\beta^*}/\QQ_\alpha^C$ to show that $\bar{X}_f$ is positive for~$\Fsum_{k<m}\F_\alpha^{\beta_k}$.
We have to show that for all $\seqlangle A_k\with k<m\seqrangle$ with $A_k\in \F_\alpha^{\beta_k}$ there exists $s\in \bar{X}_f$ with~$s\subseteq \Fsum_{k<m}A_k$.
For easier notation, assume that there exists $m'\leq m$ such that $A_k\in V[\PP_\alpha*\QQ_\alpha^C]$
if and only if $k<m'$.

For $m'\leq k<m$ there exists $B_k\in\genlangle \FBa_\alpha^{\beta_k}\cap V[\PP_\alpha*\QQ_\alpha^C]\genrangle$, $\ell_k\in \omega$ and $\seqlangle j_r^k \with r< \ell_k\seqrangle \subs i$
such that $A_k= B_k \cap \bigcap_{r<\ell_k} (\omega \setminus a_{\rho \concat j_r^k})$.
So $\Fsum_{k<m}A_k =\Fsum_{k<m'}A_k\fsum \Fsum_{m'\leq k <m} (B_k \cap \bigcap_{r<\ell_k} (\omega \setminus a_{\rho \concat j_r^k})).$
Let $p\in \QQ_\alpha^{<\beta^*}/\QQ_\alpha^C $. Without loss of generality assume that $\rho \concat j_r^k \in \dom(p)$ if $\rho \concat j_r^k \notin C$.
Let $N>|s_\tau^p|$ for every $\tau\in \dom(p)$.
We can fix $s\in \bar{X}_f$ with $s\subseteq  \Fsum_{k<m'}A_k\fsum \Fsum_{m'\leq k <m} (B_k\setminus N)$.
To get $q$, extend each $s_{\rho\concat j_r^k}^p$ with $0$'s to have length $\max(s) + 1$.

It is easy to
check
that $q$ is a condition, and that $q$
forces
$s \subseteq \Fsum_{k<m}A_k$, as desired.
\end{proof}

By the above Lemma~\ref{lemm:main_Canjar_sum_induction}(1),
$\B$ is unbounded in
$V[\PP_\alpha]$ for every $\alpha < \lambda$,
so by applying
Theorem~\ref{theo:preservation_unboundedness_JudahShelah}
once again, it follows that $\B$ is unbounded in our final model~$V[\PP_\lambda]$.
Since $|\B| = \om_1$, the bounding number~$\bfrak$ is $\om_1$ in our final model, and 
therefore also $\h$ is $\om_1$, 
as desired.

This concludes the proof of Main Theorem~\ref{maintheo:maintheo_general}.

\section{Further discussion and questions}

In this section, we 
discuss the structure of distributivity matrices 
as well as 
a notion of distributivity spectrum. 
For 
basic 
definitions and facts, see Section~\ref{sec:preliminaries}.

\subsection{Branches through distributivity matrices}\label{sec:Dordal_branches}

For the nature of a maximal branch 
 through a distributivity matrix,  
 there are two possibilities: either it is cofinal 
 or not.
It is straightforward to check that 
every maximal branch 
which is not cofinal is a tower. 
We call a distributivity matrix 
normal if  
no single element of~$[\om]^\om$ intersects it. 
Recall from Section~\ref{sec:preliminaries} that 
whenever there 
is a distributivity matrix 
of height~$\lambda$, 
then there is also a normal distributivity matrix 
of height~$\lambda$. 
It is easy to see that 
a distributivity matrix is normal if and only if all its maximal branches are towers.

In case $\tfrak = \h$ (so in particular under~$\h = \om_1$) 
there are no towers of length strictly less than~$\h$, hence
all maximal branches of a distributivity matrix of height~$\h$ are cofinal.

On the other hand, it is possible to have
a distributivity matrix of height~$\h$
which has no cofinal branches.
In fact,
it was shown by Dow
that this is the case in the Mathias model (see~\cite[Lemma~2.17]{Dow_tree_pi_bases}):

\begin{theo} Assume CH. 
In the extension by the 
countable support
iteration
of Mathias forcing
of length~$\omega_2$, 
there exists a 
distributivity matrix 
of
height~$\h$
without cofinal branches\footnote{In fact, there is even a base matrix of this kind.} 
(and $\om_1 = \tfrak <  \h = \cc = \om_2$).
\end{theo}

We do not know whether there exists a 
normal\footnote{Here and in similar cases, 
it is necessary to demand that the distributivity matrix is normal, due to the fact that there are always trivial examples of distributivity matrices with constant cofinal branches.}
distributivity matrix of height $\h$ \emph{with} cofinal branches in the Mathias model; 
this would imply that $\h = \omega_2\in \tomspec$. We also do not know whether $\omega_2\in \tomspec$ holds in the Mathias model. 

 It is actually consistent that
\emph{no} 
normal 
distributivity matrix of height~$\h$
has cofinal branches.
This
was
proved by Dordal
by constructing
a model in which $\h$ does not belong to the tower spectrum 
(see~\cite{Dordal_model} or\footnotemark~
\cite[Corollary~2.6]{Dordal_tower_spectrum}):

\footnotetext{In fact, 
\cite[Corollary~2.6]{Dordal_tower_spectrum}) 
also works for getting 
$\h = \cc$ larger than~$\om_2$,
and for certain tower spectra which are more complicated than~$\{\om_1\}$.
}

\begin{theo}\label{theo:dordal}
It is consistent with ZFC that $\tomspec = \{ \om_1 \}$ and
$\h = \om_2 = \cc$.
\end{theo}

Let us now discuss distributivity matrices of regular height strictly above~$\h$. 
Recall that $\om_1 = \tfrak = \h$ 
holds true
in the model of 
Main Theorem~\ref{maintheo:maintheo_general}; in particular, 
there are distributivity matrices of height~$\om_1$ (all whose maximal branches are cofinal). 
All maximal branches 
through the generic distributivity matrix of height~$\lambda > \om_1$ are cofinal, 
because the forcing construction is based on the tree~$\lambda^{<\lambda}$. 
Moreover, as shown in Section~\ref{subsection:branches}, 
all these maximal branches are actually towers (i.e., the matrix is normal). 
In particular, 
$\lambda$ belongs to $\tomspec$.

In the Cohen model, the situation is 
different.
Again, 
$\om_1 = \tfrak = \h$ 
holds true, 
so 
there are distributivity matrices of height~$\om_1$ 
(all whose maximal branches are cofinal). 
We do not know the following:

\begin{ques}
Is there a distributivity matrix of
regular
height larger
than~$\h$ in the Cohen model?
\end{ques}

In any case,
there is a crucial difference to
the model of our main theorem: in the Cohen model, 
there is no normal distributivity matrix of regular height~$\lambda>\omega_1$ with cofinal branches,  
due to the following 
well-known 
fact.

\begin{prop}
Assume CH, and let $\mu$ be a cardinal with $\cf(\mu) >\om$. 
Then $\tomspec = \{ \om_1 \}$ holds true\footnotemark{} 
in the extension by~$\CC_\mu$ (where $\CC_\mu$ is the forcing for adding $\mu$ many 
Cohen reals).
\end{prop}

\footnotetext{In fact, the following stronger statement 
holds true 
in the extension:
Let $\lambda>\omega_1$ be regular, and let $\seqlangle a_\alpha\with \alpha<\lambda \seqrangle$ be a $\subseteq^*$-decreasing sequence; 
then there exists
an 
$\alpha_0<\lambda$ such that $a_{\alpha_0}=^* a_\beta$ for every $\beta\geq \alpha_0$.}

Finally, let us remark that 
the generic distributivity matrix from Main Theorem~\ref{maintheo:maintheo_general} cannot be a base matrix. This can 
be seen 
by a 
slight generalization of the proof of Lemma~\ref{lemm:infinite_set} 
which 
yields the following. 
For each infinite ground model set $b\subseteq \omega$, each $a_\sigma$ has infinitely many $1$'s (and also infinitely many $0$'s) within $b$. If $b$ is infinite and co-infinite, it follows that $b$ splits $a_\sigma$. In particular, $a_\sigma \not\subseteq^* b$, so $b$ witnesses that the generic matrix is not 
a base matrix.

\subsection{The distributivity spectrum}\label{sec:h_spectrum}

The study of distributivity matrices of various heights naturally gives rise to the
following notion. 
Let
\[
\homspec := \{ \lambda \with \lambda \textrm{ is regular and there is a distributivity matrix of height } \lambda \}
\]
be 
the 
\emph{distributivity spectrum}. 
Recall that the existence of
distributivity matrices
is only a matter of cofinality, i.e., there exists a distributivity matrix of height $\delta$ if and only if there exists one of height $\cf(\delta)$. 
Therefore,
the restriction of the
definition of $\homspec$ to regular cardinals makes sense.
Clearly, the minimum of~$\homspec$ is the distributivity number~$\h$.

Spectra have been considered
for several cardinal characteristics,
but not for~$\h$.
For example,
spectra for
the tower number~$\tfrak$ have been
investigated
in~\cite{Hechler} and \cite{Dordal_tower_spectrum},
spectra for
the almost disjointness number~$\afrak$
in~\cite{Hechler},
\cite{Brendle_mad},
and \cite{Otmar_a_spectrum},
spectra for
the bounding number~$\bfrak$
in~\cite{Dordal_tower_spectrum},
spectra for
the ultrafilter number~$\mathfrak{u}$
in \cite{Shelah_u_spectrum_1},\cite{Shelah_u_spectrum_2}, and \cite{Garti_Magidor_Shelah_u_spectrum},
and spectra for
the independence number~$\mathfrak{i}$ in~\cite{Vera_i_spectrum}. Furthermore, \cite{Blass_simple} develops a framework for dealing with several spectra.

Let $[\h,\cc]_\reg$ denote the set of regular cardinals~$\delta$ with $\h \leq \delta \leq \cc$. 
As already mentioned, 
it is easy to check that there can never be 
a distributivity matrix of regular height larger than~$\cc$, 
hence  
$\homspec\subseteq [\h,\cc]_\reg$. 
Recall that 
the model of Main Theorem~\ref{maintheo:maintheo_general} satisfies
$$\{\om_1, \lambda\} \subs \homspec$$
(where $\lambda > \om_1$ is the regular cardinal chosen there). 
In particular, by choosing $\lambda = \mu = \om_2$, we obtain 
a model in which 
$\{\omega_1,\omega_2\} = \homspec = [\h,\cc]_\reg$.

\begin{ques}
Is it consistent that $\homspec$
contains
more than 2 elements?
\end{ques}

\subsubsection*{Acknowledgment}

We 
want to thank
Osvaldo Guzm\'an for his inspiring tutorial
at
the Winter School 2020
in Hejnice and for helpful discussion
about $\B$-Canjar filters.

\bibliography{bib_files_Vera_h}
\bibliographystyle{plain}

\end{document}